\documentclass[12pt]{amsart}

\usepackage{mystyle}

\usepackage{mathrsfs}

\usepackage{graphicx}
\usepackage[centertags]{amsmath}
\usepackage{newlfont}
\usepackage{hyperref}

\usepackage[page,title,titletoc,header]{appendix}
\usepackage{stmaryrd}

\newcommand{\dom}{\mathrm{dom}}

\theoremstyle{plain}
\newtheorem{theorem1}{Theorem}
\newtheorem{corollary1}[theorem1]{Corollary}

\theoremstyle{definition}

\theoremstyle{remark}

\begin{document}
\pagestyle{headings}
\title{Invariant means and the structure of inner amenable groups}
\author{Robin D. Tucker-Drob}
\address{Rutgers University, Hill Center for the Mathematical Sciences, 110 Frelinghuysen Rd, Piscataway, NJ 08854-8019}
\email{rtuckerd@math.rutgers.edu}
\begin{abstract}
We study actions of countable discrete groups which are amenable in the sense that there exists a mean on $X$ which is invariant under the action of $G$. Assuming that $G$ is nonamenable, we obtain structural results for the stabilizer subgroups of amenable actions which allow us to relate the first $\ell ^2$-Betti number of $G$ with that of the stabilizer subgroups. An analogous relationship is also shown to hold for cost. This relationship becomes even more pronounced for transitive amenable actions, leading to a simple criterion for vanishing of the first $\ell ^2$-Betti number and triviality of cost. Moreover, for any marked finitely generated nonamenable group $G$ we establish a uniform isoperimetric threshold for Schreier graphs $G/H$ of $G$, beyond which the group $H$ is necessarily weakly normal in $G$.

Even more can be said in the particular case of an atomless mean for the conjugation action -- that is, when $G$ is inner amenable. We show that inner amenable groups have cost 1 and moreover they have fixed price. We establish $\mathscr{U}_{\mathrm{fin}}$-cocycle superrigidity for the Bernoulli shift of any nonamenable inner amenable group. In addition, we provide a concrete structure theorem for inner amenable linear groups over an arbitrary field. We also completely characterize linear groups which are stable in the sense of Jones and Schmidt. Our analysis of stability leads to many new examples of stable groups; notably, all nontrivial countable subgroups of the group $H(\R )$, recently studied by Monod, are stable. This includes nonamenable groups constructed by Monod and by Lodha and Moore, as well as Thompson's group $F$.
\end{abstract}
\maketitle
\section*{Introduction}\label{sec:intro}

\subsection{Amenable actions}\label{subsec:intro1} An action of a discrete group $G$ on a set $X$ is said to be {\bf amenable} if there exists a finitely additive probability measure $\bm{m} : \mathscr{P}(X)\ra [0,1]$, henceforth called a {\bf mean}, defined on the powerset of $X$, which is invariant under the action of $G$. This definition goes back to von Neumann's 1929 memoir on paradoxicality \cite{vN29}, where the notion of amenability of a group \emph{simpliciter} was also introduced: by definition, $G$ is amenable if the left translation action of $G$ on itself is amenable in the above sense.

Every action of an amenable group is amenable, and for a long time this simple observation could account for most of the known examples of amenability in actions.\footnote{There are early examples of actions of nonamenable groups with asymptotic fixed points (e.g., in \cite{MvN43}), although the amenability of such actions was not adduced until much later (e.g., in \cite{Ef75, Pr83, BH86, JS87})}  A more systematic study of actions whose amenability could not be traced back to that of some acting group began with van Douwen's constructions of amenable actions of the free group \cite{vD90}, and has continued in recent years with \cite{MP03}, \cite{Pe03}, \cite{GM07}, \cite{GN07}, \cite{Mo10}, \cite{AE11c}, \cite{Mo11}, \cite{Mo11b}, \cite{Fi12}, \cite{JM13}, \cite{JNS13}, \cite{JS13}. A stunning recent application of amenable actions is in the article \cite{JM13} of Juschenko and Monod in which the authors turn the classical implication on its head, deducing amenability of a group from that of an action, thereby providing the first examples of infinite finitely generated simple amenable groups.

If a nonamenable group $G$ acts amenably on $X$ then it is well known that this action is far from being free: any $G$-invariant mean $\bm{m}$ on $X$ must concentrate on the set of points $x\in X$ whose associated stabilizer subgroup $G_x$ is nonamenable. Our first result strengthens this considerably by showing that, on a $\bm{m}$-conull set, the subgroups $G_x$ are in fact so large in $G$ as to be "visible from above." The precise statement uses the following variation of Popa's notions of $q$-normality and $wq$-normality \cite{Pop06b}. A subgroup $H$ of $G$ is said to be {\bf $q^*$-normal} in $G$ if the set $\{ g\in G\csuchthat gHg^{-1}\cap H \text{ is nonamenable}\}$ generates $G$. The subgroup $H$ is {\bf $wq^*$-normal} in $G$ if there exists an ordinal $\lambda$ and an increasing sequence $(H_\alpha )_{\alpha \leq \lambda}$ of subgroups of $G$, with $H_0=H$ and $H_\lambda = G$, such that $\bigcup _{\beta <\alpha} H_\beta$ is $q^*$-normal in $H_\alpha$ for all $\alpha \leq \lambda$. The notions of {\bf $q$-normal} and {\bf $wq$-normal} subgroups are defined in the same way, except with "nonamenable" replaced by "infinite." It is immediate that a $wq$-normal subgroup is necessarily infinite, and a $wq^*$-normal subgroup is necessarily nonamenable.

\begin{theorem1}\label{thm:wqstar}
Let $G$ be a finitely generated nonamenable group. Assume that $G$ acts amenably on $X$ and fix a $G$-invariant mean $\bm{m}$ on $X$. Then $G_x$ is $wq^*$-normal in $G$ for $\bm{m}$-almost every $x\in X$.
\end{theorem1}

\begin{example}\label{ex:fg}
The assumption of finite generation is necessary in the statement of Theorem \ref{thm:wqstar}. Let $G$ be a free group with free generating set $S = \{ s_i \} _{0\leq i<\infty}$, and let $G_n\leq G$ be the subgroup generated by $\{ s_i \} _{0\leq i<n}$. The action $G\cc X=\bigsqcup _{n\geq 0}G/G_n$ is amenable, although the stabilizer of any $x\in X$ is malnormal in $G$.
\end{example}

\begin{remark} Even when $G$ is not finitely generated Theorem \ref{thm:wqstar} can be applied to finitely generated subgroups, as in Corollary \ref{cor:obstr} below, to obtain a statement which holds for all countable groups. Corollary \ref{cor:obstr} also shows that Example \ref{ex:fg} is in fact the prototypical obstruction to $wq^*$-normality of stabilizer subgroups when $G$ is not finitely generated.
\end{remark}

\begin{remark}
Theorem \ref{thm:wqstar} can be strengthened. In \S\ref{sec:normality} we introduce a natural hierarchy of incremental strengthenings of $wq^*$-normality (see Definition \ref{def:nth}). The conclusion of Theorem \ref{thm:wqstar} then remains true when $wq^*$-normality is replaced by any of these strengthenings. In addition, a relativized version of Theorem \ref{thm:wqstar} holds; see Theorem \ref{thm:isoperim}.
\end{remark}

Theorem \ref{thm:wqstar} provides a means of studying measured group theoretical properties of $G$ via its amenable actions, since many such properties are known to be reflected in the structure of $wq$-normal subgroups. For example, Popa has shown that if $G$ contains a $wq$-normal subgroup whose Bernoulli shift is $\mathscr{U}_{\mathrm{fin}}$-cocycle superrigid, then the same holds for the Bernoulli shift of $G$ \cite{Pop07}. In \cite{PT11}, Peterson and Thom show that the first $\ell ^2$-Betti number, $\beta ^{(2)}_1(G)$, of $G$ is bounded above by that of its $wq$-normal subgroups. An analogous statement also holds for the pseudocost, $\mathscr{PC}(G)$, of $G$, see Proposition \ref{prop:Furman2}. We therefore obtain the following corollary, which strengthens a theorem of Promislow \cite{Pr83} concerning actions of free groups.

\begin{corollary1}\label{cor:l2cost}
Let $G$ be a countable nonamenable group. Assume that $G$ acts amenably on $X$ and fix a $G$-invariant mean $\bm{m}$ on $X$.
\begin{enumerate}
\item Suppose that $G$ is finitely generated. Then $\beta ^{(2)}_1(G_x)\geq \beta ^{(2)}_1(G)$ and $\mathscr{PC}(G_x)\geq \mathscr{PC}(G)$ for $\bm{m}$-almost every $x\in X$.
\item In general, if $\beta ^{(2)}_1(G)>r$, then $\beta ^{(2)}_1(G_x)>r$ for $\bm{m}$-almost every $x\in X$. Likewise, if $\mathscr{PC}(G)>r$ then $\mathscr{PC}(G_x)>r$ for $\bm{m}$-almost every $x\in X$.
\end{enumerate}
\end{corollary1}

\begin{proof}[Proof of Corollary \ref{cor:l2cost}]
Part (1) follows from Theorem \ref{thm:wqstar} using Theorem 5.6 of \cite{PT11} in the case of first $\ell ^2$-Betti number, and using Proposition \ref{prop:Furman2} in the case of pseudocost. For part (2), if $\mathscr{PC}(G) >r$ then by Proposition \ref{prop:PC} there exists a finitely generated nonamenable subgroup $H_0\leq G$ such that $\mathscr{PC}(H) >r$ for all $H_0\leq H \leq G$. By Corollary \ref{cor:obstr} there exists a $G$-map $\varphi : X\ra Y$ to a $G$-set $Y$ such that $G_x$ is $wq^*$-normal in $G_{\varphi (x)}$ and $H_0\leq G_{\varphi (x)}$ for $\bm{m}$-almost every $x\in X$. Therefore, by Proposition \ref{prop:Furman2}, $\mathscr{PC}(G_x)\geq \mathscr{PC}(G_{\varphi (x)})>r$ for $\bm{m}$-almost every $x\in X$. An analogous argument goes through for the first $\ell ^2$-Betti number using Corollary 5.13 of \cite{Ga02} in place of Proposition \ref{prop:PC}.
\end{proof}

\subsection{An isoperimetric threshold}\label{subsec:isop} Let $X$ be a $G$-set and for a finite subset $S$ of $G$ denote by $\phi _S(X)$ the {\bf isoperimetric constant} of the Schreier graph with respect to $S$, associated with the action
\begin{equation}\label{eqn:isop}
\phi _{S}(X) = \inf \Big\{ \sum _{s\in S} \frac{|sP\setminus P|}{|P|} \csuchthat P\subseteq X \text{ is finite and nonempty} \Big\} .
\end{equation}
When $S$ generates $G$, the value $\phi _{S}(G)$ is then the isoperimetric constant of the Cayley graph of $G$ with respect to $S$. We always have $\phi _{S}(X)\leq \phi _{S}(G)$, and if $S$ generates $G$ then $X$ is an amenable $G$-set if and only if $\phi _{S}(X) = 0$.

\begin{remark}
There are several variations of the definition \eqref{eqn:isop}. For example, it can often be convenient to work with the {\bf conductance constant} $h_S(X) = \frac{1}{|S|}\phi_S(X)$. See \cite{JN12} for a discussion in the case $X=G$. For our purpose, any fixed multiplicative renormalization of $\phi _S(X)$ would be suitable since our main interest will be in the ratio $\phi _S(X)/\phi _S (G)$.
\end{remark}

Assume now that $S$ generates $G$. Theorem \ref{thm:wqstar} then has a surprising consequence when combined with Kazhdan's trick. Namely, there exists a constant $\epsilon = \epsilon _{G,S} >0$ such that any subgroup $H$ of $G$ satisfying $\phi _{S}(G/H ) <\epsilon$ must be $wq^*$-normal in $G$. Indeed, otherwise, for each $n\geq 0$ there is a subgroup $H_n\leq G$ such that $\phi _{S}(G/H_n ) < 2^{-n}$ but with $H_n$ not $wq^*$-normal in $G$, so the amenable action $G\cc \bigsqcup _n G/H_n$ contradicts Theorem \ref{thm:wqstar}. While this argument does not give any indication about the actual value of $\epsilon _{G,S}$, we obtain a sharp estimate in \S\ref{subsec:isoperim}. (See \S\ref{sec:normality} for the definition of $n$-degree $\mathscr{N}^X$-$wq$-normality.)

\begin{theorem1}\label{thm:isoperim}
Let $G$ be a nonamenable group with finite generating set $S$ and let $H$ be a subgroup of $G$. If $\phi _{S}(G/H ) <\frac{1}{2}\phi _{S}(G)$, then $H$ is $wq^*$-normal in $G$. More generally, for each $G$-set $X$, and for each nonnegative integer $n$ we have the following implication
\[
\phi _S(G/H ) < \tfrac{1}{2^n}\phi _S (X) \ \Ra \ H\text{ is }n\text{-degree }\mathscr{N}^X\text{-}wq\text{-normal in }G .
\]
\end{theorem1}

\begin{example}
Theorem \ref{thm:isoperim} is sharp. Consider the free group $\F _2$ of rank $2$ with free generating set $S= \{ a, b \}$. We have $\phi _{S}(\F _2) = 1$ \cite[Example 47]{C-SGH99}. The subgroup $H=\langle a,bab^{-1}\rangle$ is \emph{not} $wq^*$-normal in $\F _2$ (although it is $q$-normal). An inspection of the Schreier graph of $\F _2 / H$ verifies that $\phi _{S}(\F _2 /H ) =\tfrac{1}{2}$.
\end{example}

\subsection{Transitive actions and a vanishing criterion} In the case of an amenable transitive action we obtain the following strengthening of Corollary \ref{cor:l2cost}.
\begin{theorem1}\label{thm:coamen}
Let $G$ be a countable group. Assume that $G$ acts amenably and transitively on an infinite set $X$ and fix some $x\in X$. If $\beta ^{(2)}_1(G_x)<\infty$ then $\beta ^{(2)}_1(G)=0$. Likewise, if $\mathscr{PC}(G_x )<\infty$ then $\mathscr{PC}(G)=1$.
\end{theorem1}

It follows that if $G$ is a group with $\beta ^{(2)}_1(G)>0$ or with $\mathscr{PC}(G)>1$, then for any finitely generated infinite index subgroup $H\leq G$, the action $G\cc G/H$ is \emph{not} amenable.

Theorem \ref{thm:coamen} is closely related to a vanishing criterion due to Peterson and Thom \cite{PT11}. They define a subgroup $H$ of $G$ to be {\bf $s$-normal} in $G$ if $gHg^{-1}\cap H$ is infinite for every $g\in G$; the notion of {\bf $ws$-normality} is then obtained by iterating $s$-normality transfinitely. Theorem 5.12 of \cite{PT11} states that if $G$ contains a $ws$-normal infinite index subgroup $H$ with $\beta ^{(2)}_1(H)<\infty$ then $\beta ^{(2)}_1(G) =0$. Ioana (unpublished) has shown that the analogous statement also holds for cost (Ioana's argument works for pseudocost as well). Theorem \ref{thm:coamen} would therefore follow from these results if the subgroup $G_x$ were always $ws$-normal in $G$. This turns out not to be the case however, as the following example shows.

\begin{example}\label{ex:ws}
Let $K$ be a group which is isomorphic with one of its proper malnormal subgroups $K_0$ (e.g., any nonabelian free group has this property, see \cite[Example 1]{BMR99}). Fix an isomorphism $\varphi :K\ra K_0$ and let $G=\langle t,K \, | \, tkt^{-1}=\varphi (k) \rangle$ be the associated HNN-extension. By Proposition 2 of \cite{MP03}, the action of $G$ on $G/K$ is amenable. However, $K$ is not $ws$-normal in $G$. To see this note that, from the semidirect product decomposition $G= \big( \bigcup _{n\geq 0}t^{-n}Kt^n \big) \rtimes \langle t \rangle$, it follows that every intermediate subgroup $K\lneq L \leq G$ contains an element of the form $g=t^{-n}kt^m$, where $n, m>0$ and $k\in K-K_0$, and clearly $gKg^{-1}\cap K = 1$.
\end{example}

\subsection{Inner amenability} In their 1943 study of $\mathrm{II}_1$ factors \cite{MvN43}, Murray and von Neumann distinguished the hyperfinite $\mathrm{II}_1$ factor from the free group factor $\mathrm{L}\F _2$ by means of {\bf property Gamma}, that is, the existence of nontrivial asymptotically central sequences. In demonstrating that $\mathrm{L}\F _2$ lacks this property, Murray and von Neumann hinted at a connection with amenability \cite[footnote 71]{MvN43}, remarking that their argument, which makes ancillary use of approximately invariant measures, closely mirrors Hausdorff's famous paradoxical division of the sphere. This connection was not made explicit however until 1975 when Effros \cite{Ef75} introduced the following group theoretic notion:

\begin{definition} A group $G$ is {\bf inner amenable} if the action of $G$ on itself by conjugation admits an atomless invariant mean.
\end{definition}

Effros showed that if a group factor $\mathrm{L}G$ has property Gamma, then $G$ is necessarily inner amenable. An ICC counterexample to the converse statement was found only very recently by Vaes \cite{Va12}.

The proof of Theorem \ref{thm:wqstar} naturally involves exploiting the tension between the nonamenability of $G$ and the amenability of the action. In the case of the conjugation action, this tension leads to remarkably strong consequences for the group theoretic and measured group theoretic structure of $G$. Many of these consequences will in fact be shown to hold in the more general setting of inner amenable pairs: If $H$ is a subgroup of $G$ then we say that {\bf the pair $(G,H)$ is inner amenable} if the conjugation action of $H$ on $G$ admits an atomless invariant mean.\footnote{This is a bit different from the notion, defined by Jolissaint in \cite{Jo13}, of {\bf $H$ being inner amenable relative to $G$}, which amounts to amenability of the conjugation action of $H$ on $G-H$. While Jolissaint's notion does not appear anywhere else in this article, to avoid conflicting terminology we will make sure to refer to inner amenability of the pair $(G,H)$ when referring to the notion defined in the main text.}

\subsection{The cost of inner amenable groups}
We let $\mathscr{C}(G)$ denote the cost of $G$, that is, $\mathscr{C}(G)$ is the infimum of the costs of free probability measure preserving actions of $G$. We let $\mathscr{C}^*(G)$ denote the supremum of the costs of free probability measure preserving actions of $G$. Then $G$ has {\bf fixed price} if $\mathscr{C}(G)=\mathscr{C}^*(G)$.

\begin{theorem1}\label{thm:FP1}
Let $G$ be a countable group.
\begin{enumerate}
\item[(1)] Suppose that $G$ contains a $wq$-normal subgroup $H$ such that $(G,H)$ is inner amenable. Then $\mathscr{C}(G)=1$.
\item[(2)] Suppose that $G$ is inner amenable. Then $\mathscr{C}(G)=1$ and $G$ has fixed price.
\end{enumerate}
\end{theorem1}

\begin{remark}\label{rem:prod}
In part (1) of Theorem \ref{thm:FP1} it would be desirable to additionally obtain that $G$ has fixed price. The proof of part (1) shows that this holds if and only if direct products of infinite groups have fixed price, which is a well known open problem.
\end{remark}

As a consequence of Theorem \ref{thm:FP1} we recover the result of Chifan, Sinclair, and Udrea \cite[Corollary D]{CSU13}, that inner amenable groups have vanishing first $\ell ^2$-Betti number. Moreover, we obtain a strengthening which holds for inner amenable pairs.

\begin{corollary1}\label{cor:l2Betti}
Let $G$ be a countable group and suppose that $G$ contains a wq-normal subgroup $H$ such that the pair $(G,H)$ is inner amenable. Then $\beta ^{(2)}_1(G)=0$. In particular, if $G$ is inner amenable then $\beta ^{(2)}_1(G)=0$.
\end{corollary1}

\begin{proof}
This follows from Theorem \ref{thm:FP1} and the inequality $\beta ^{(2)}_1(G)\leq \mathscr{C}(G)-1$ 
due to Gaboriau \cite{Ga02}. Alternatively, a direct proof may be obtained by observing that each step of the proof of Theorem \ref{thm:FP1} in \S\ref{sec:FP1} has an analogue for the first $\ell ^2$-Betti number.
\end{proof}

In \cite{AN08}, Ab\'{e}rt and Nikolov show that for a finitely generated, residually finite group $G$, the rank gradient of any Farber chain in $G$ is equal to one less than the cost of the associated boundary action of $G$. We therefore obtain the following corollary.

\begin{corollary1}\label{cor:RG}
Let $G$ be a finitely generated, residually finite group which is inner amenable. Then the rank gradient of any Farber chain in $G$ vanishes. In particular, the absolute rank gradient of $G$ vanishes.
\end{corollary1}

The two main ingredients in the proof of Theorem \ref{thm:FP1}.(2) concern the subgroup structure of nonamenable inner amenable groups.

\begin{theorem1}\label{thm:relmain1}
Let $G$ be a nonamenable inner amenable group. Then every nonamenable subgroup of $G$ is $wq$-normal in $G$.
\end{theorem1}

The next result roughly states that very large portions of $G$ commute. To make the statement somewhat less cumbersome we define $\mathscr{N}$ to be the collection of all nonamenable subgroups of $G$.

\begin{theorem1}\label{thm:subgroup}
Let $G$ be a nonamenable inner amenable group. Then at least one of the following holds:
\begin{enumerate}
\item[(1)] For every finite $\mathscr{F}\subseteq \mathscr{N}$ there exists an infinite amenable subgroup $K$ of $G$ such that $L\cap C_G(K)$ is nonamenable for all $L\in \mathscr{F}$.
\item[(2)] For every finite $\mathscr{F}\subseteq \mathscr{N}$ there exists an increasing sequence $M_0\leq M_1\leq \cdots$ of finite subgroups of $G$, with $\lim _{n\ra\infty}|M_n|=\infty$, such that $L\cap C_G(M_n)$ is nonamenable for all $L\in \mathscr{F}$, $n\in \N$.
\item[(3)] For every finite $\mathscr{F}\subseteq \mathscr{N}$ and every $n\in \N$ there exist pairwise commuting nonamenable subgroups $K_0,K_1,\dots , K_{n-1} \leq G$ such that $L\cap C_G(K_i)$ is nonamenable for all $L\in \mathscr{F}$, $i<n$. Moreover, there exists a sequence $(M_n)_{n\in \N}$ of finite subgroups of $G$ with $\lim _{n\ra\infty}|M_n|=\infty$ such that $L\cap C_G(M_n)$ is nonamenable for all $L\in \mathscr{F}$, $n\in \N$.
\end{enumerate}
\end{theorem1}

\begin{corollary1}\label{cor:subgroup}
Let $G$ be a nonamenable inner amenable group. Then $G$ either contains an infinite amenable subgroup or $G$ contains finite subgroups of arbitrarily large order. In addition, $G$ contains an infinite subgroup $K$ such that $C_G(K)$ is infinite.
\end{corollary1}

\begin{proof}[Proof of Corollary \ref{cor:subgroup}]
The first statement follows easily from Theorem \ref{thm:subgroup}. The second statement is clear if either (1) or (3) of Theorem \ref{thm:subgroup} holds. If (2) holds then $G$ contains an infinite locally finite subgroup, hence by \cite{HK64} $G$ contains an infinite abelian subgroup $A$, so $A\leq C_G(A)$.
\end{proof}

\subsection{Cocycle superrigidity} If $H$ is a subgroup of $G$, then a cocycle $w$ of a probability measure preserving action of $G$ is said to {\bf untwist on $H$} if $w$ is cohomologous to a cocycle $w'$ whose restriction to $H$ is a homomorphism. Following \cite{Pop07, Pop08}, let $\mathscr{U}_{\mathrm{fin}}$ denote the class of all Polish groups which embed as a closed subgroup of the unitary group of a finite von Neumann algebra. A free, probability measure preserving action of $G$ is said to be {\bf $\mathscr{U}_{\mathrm{fin}}$-cocycle superrigid} if every cocycle for the action which takes values in some group in $\mathscr{U}_{\mathrm{fin}}$ untwists on the entire group $G$.

Popa's Second Cocycle Superrigidity Theorem (Theorem 1.1 of \cite{Pop08}) provides general conditions for a cocycle which takes values in some group $L\in \mathscr{U}_{\mathrm{fin}}$, to untwist on the centralizer $C_G(H)$ of a nonamenable subgroup $H$ of $G$. The following theorem, which is joint with Adrian Ioana, strengthens Popa's theorem by showing that, under suitable conditions, the untwisting in fact occurs on the centralizer $C_{\mathscr{M}(G)}(H)$ of $H$ in the semigroup $\mathscr{M}(G)$ of all means on $G$.

\begin{theorem1}[with A. Ioana] \label{thm:superrigid1}
Let $H$ be a nonamenable subgroup of $G$ such that the pair $(G, H)$ is inner amenable. Let $G\cc ^{\sigma} (X,\mu )$ be a p.m.p.\ action of $G$ and assume that
\begin{itemize}
\item $\sigma _{|H}$ has stable spectral gap;
\item $\sigma _{|C_{\mathscr{M}(G)}(H)}$ is weakly mixing (see Definition \ref{def:WM});
\item $\sigma$ is $s$-malleable.
\end{itemize}
Then there exists a subgroup $G_0$ of $G$ with $H\leq G_0\leq G$ such that
\begin{enumerate}
\item Every $H$-conjugation-invariant mean on $G$ concentrates on $G_0$;
\item Every cocycle $w : G\times X \ra L$ which takes values in a group $L \in \mathscr{U}_{\mathrm{fin}}$ untwists on $G_0$.
\end{enumerate}
\end{theorem1}

Theorem \ref{thm:superrigid1} applies to the Bernoulli shift of $G$ whenever $H\leq G$ is nonamenable and the pair $(G,H )$ is inner amenable. Applying Lemma 3.5 of \cite{Fu07}, we therefore obtain:

\begin{corollary1}\label{cor:superrigid}
Let $G$ be a countable group containing a $wq$-normal nonamenable subgroup $H$ such that the pair $(G,H)$ is inner amenable. Then the Bernoulli shift of $G$ is $\mathscr{U}_{\mathrm{fin}}$-cocycle superrigid. In particular, the Bernoulli shift of any nonamenable inner amenable group is $\mathscr{U}_{\mathrm{fin}}$-cocycle superrigid.
\end{corollary1}

Corollary \ref{cor:superrigid} strengthens a result of Peterson and Sinclair \cite{PS11}, stating that the Bernoulli shift of $G$ is $\mathscr{U}_{\mathrm{fin}}$-cocycle superrigid provided $G$ is nonamenable and $\mathrm{L}G$ has property Gamma.

The case $H=G$ of Corollary \ref{cor:superrigid} would follow from Popa's theorem combined with Lemma 3.5 of \cite{Fu07} and Theorems \ref{thm:relmain1} and \ref{thm:subgroup} above, provided that alternative (2) could be dropped from the statement of Theorem \ref{thm:subgroup}. However, the following example exhibits an inner amenable group with the property that the centralizer of every nonamenable subgroup is finite; in particular, such a group does not satisfy either of the alternatives (1) or (3) of Theorem \ref{thm:subgroup}.

\begin{example}\label{ex:faction}
Let $\F _2 \cc X$ be a transitive amenable action of the free group $\F _2$ on an infinite set $X$ with the following property: for all $u\in \F _2 - 1$ the set $\{ P\in \mathscr{P}_{\text{f}}(X) \csuchthat u\cdot P = P \}$ is finite, where $\mathscr{P}_{\text{f}}(X)$ denotes the collection of all finite subsets of $X$. Such an action is constructed in Theorem \ref{thm:safree}. The group $G= \mathscr{P}_{\text{f}}(X)\rtimes \F _2 \cong (\bigoplus _{x\in X}\Z _2 ) \rtimes \F _2$ is then inner amenable so Corollary \ref{cor:superrigid} applies to $G$. The group $G$ is also finitely generated and ICC. In addition, the centralizer of any nonamenable subgroup of $G$ is finite, or equivalently, the centralizer of every infinite subgroup $H$ of $G$ is amenable. Indeed, if $(P,u) \in H$ then we have $C_G(H)\leq C_G((P,u)) \leq \mathscr{P}_{\text{f}}(X)\rtimes C_{\F _2}(u)$, which is amenable unless $u = 1$. We may therefore assume that $H\leq \mathscr{P}_{\text{f}}(X)$, in which case, since $H$ is infinite we have
\[
C_G(H) = \mathscr{P}_{\text{f}}(X)\rtimes \{ u\in \F _2 \csuchthat u\cdot P = P \text{ for all }P\in H \} = \mathscr{P}_{\text{f}}(X)\rtimes 1 ,
\]
which is amenable.
\end{example}

\subsection{The structure of inner amenable linear groups}\label{sec:strlinear}
In \S\ref{sec:radicals} we characterize inner amenability for linear groups in terms of a certain amenable characteristic subgroup of $G$. The {\bf AC-center} of a countable group $G$ is the subgroup
\[
\mathscr{AC}(G) = \langle \{ N\leq G\csuchthat N\text{ is normal in }G\text{ and }G/C_G(N)\text{ is amenable}\}\rangle .
\]
The {\bf inner radical} of $G$ is the subgroup
\[
\mathscr{I}(G) = \langle \{ N\leq G\csuchthat N\text{ is normal in }G\text{ and the action }N\rtimes G \cc N\text{ is amenable} \}\rangle .
\]
Here, $N\rtimes G \cc N$ is the action where $N$ acts by left translation and $G$ acts by conjugation. The relevant properties of these subgroups are summarized in the following theorem.

\begin{theorem1}\label{thm:I(G)} Let $G$ be a countable group.
\begin{itemize}
\item[i.] $\mathscr{AC}(G)$ and $\mathscr{I}(G)$ are amenable characteristic subgroups of $G$;
\item[ii.] $\mathscr{AC}(G)\leq \mathscr{I}(G)$;
\item[iii.] The actions $\mathscr{AC}(G)\rtimes G \cc \mathscr{AC}(G)$ and $\mathscr{I}(G)\rtimes G\cc \mathscr{I}(G)$ are amenable;
\item[iv.] $G/C_G(\mathscr{AC}(G))$ is residually amenable;
\item[v.] If $\mathscr{I}(G)$ is infinite then $G$ is inner amenable;
\item[vi.] Let $N$ be a normal subgroup of $G$ with $N\leq \mathscr{I}(G)$. Then $\mathscr{I}(G/N) = \mathscr{I}(G)/N$;
\item[vii.] $\mathscr{I}(G/\mathscr{I}(G))=1$ and $G/\mathscr{I}(G)$ is $\mathrm{ICC}$;
\item[viii.] Every conjugation invariant mean on $G/\mathscr{I}(G)$ is the projection of a conjugation invariant mean on $G$.
\end{itemize}
Moreover, if $G$ is linear then
\begin{itemize}
\item[ix.] $\mathscr{AC}(G)=\mathscr{I}(G)$;
\item[x.] $G/C_G(\mathscr{I}(G))$ is amenable;
\item[xi.] $\mathscr{I}(G)=C_G(C_G(\mathscr{I}(G)))$;
\item[xii.]$G/\mathscr{I}(G)$ is not inner amenable;
\item[xiii.] Every conjugation invariant mean on $G$ concentrates on $\mathscr{I}(G)$;
\item[xiv.] Let $N$ be a normal subgroup of $G$ with $N\leq \mathscr{I}(G)$. Then \emph{ix.}\ through \emph{xiii.}\ all hold with $G/N$ in place of $G$.
\end{itemize}
\end{theorem1}

\begin{remark}
Theorem \ref{thm:I(G)}.xi.\ implies that if $G$ is linear then so are the groups $G/\mathscr{I}(G)$ and $G/C_G(\mathscr{I}(G))$ (see 
Theorem 6.2 of \cite{We69}). It then follows from item x.\ and the Tits alternative that if $G$ is additionally finitely generated, then $\mathscr{I}(G)$ is virtually solvable.
\end{remark}

Using Theorem \ref{thm:I(G)} we are able to show that within the class of linear groups, inner amenability occurs only for the most obvious reasons: every linear inner amenable group is an amenable extension either of a group with infinite center or of a near product group in which one of the factors is infinite and amenable. More precisely, we obtain the following structure theorem for inner amenable linear groups.

\begin{theorem1}\label{thm:linear} Let $G$ be a countable linear group. Then the following are equivalent:
\begin{itemize}
\item[(1)] $G$ is inner amenable.
\item[(2)] $\mathscr{I}(G)$ is infinite.
\item[(3)] There exists a short exact sequence $1\ra N \ra G \ra K \ra 1$, where $K$ is amenable 
and either
\begin{itemize}
\item[$\bullet$] $Z(N)$ is infinite, or
\item[$\bullet$] $N = LM$, where $L$ and $M$ are commuting normal subgroups of $G$ such that $M$ is infinite and amenable, and $L\cap M$ is finite.
\end{itemize}
\end{itemize}
\end{theorem1}

In \cite{Sch86}, Schmidt raises the question of whether every inner amenable group $G$ possesses a free ergodic p.m.p.\ action $G\cc (X,\mu )$ which generates an orbit equivalence relation $\mc{R}^G_X$ for which the outer automorphism group of the full group $[\mc{R}^G_X]$ is not Polish, or equivalently, for which the full group $[\mc{R}^G_X]$ contains an asymptotically central sequence $(T_n)_{n\in \N}$ with $\liminf _n \mu ( \{ x \in X \csuchthat T_n x \neq x \} ) >0$. See also Problem 9.3 of \cite{Ke10}. Using Theorems \ref{thm:I(G)}, \ref{thm:linear}, and Theorem \ref{thm:linearstable} below, we obtain a positive answer to Schmidt's question when $G$ is linear.

\begin{theorem1}\label{thm:Schmidt}
A countable linear group $G$ is inner amenable if and only if there exists a free ergodic p.m.p.\ action $G\cc (X,\mu )$ such that the outer automorphism group of $[\mc{R}^G_X]$ is not Polish.
\end{theorem1}

\subsection{Stability} A discrete probability measure preserving equivalence relation $\mc{R}$ is said to be {\bf stable} if it is isomorphic to its direct product $\mc{R}\times \mc{R}_0$ with the equivalence relation $\mc{R}_0$, of eventual equality on $2^\N$ equipped with the uniform product measure. A countable group $G$ is said to be {\bf stable} if it possesses a free ergodic probability measure preserving action which generates a stable equivalence relation. Stability was introduced by Jones and Schmidt in \cite{JS87}, where it was also shown that stable groups are necessarily inner amenable. The first examples of ICC inner amenable groups which are not stable were recently constructed by Kida \cite{Ki12}; these groups are obtained as HNN extensions of property (T) groups with infinite center. Further results of Kida from \cite{Ki13a} show that if the center $Z(G)$ of a group $G$ is infinite, then the question of whether $G$ is stable is intimately related to the question of whether the pair $(G,Z(G))$ lacks relative property (T). Using Theorem \ref{thm:I(G)}, in \S\ref{sec:stability} we are able to completely characterize stability for linear groups in terms of relative property (T).

\begin{theorem1}\label{thm:linearstable}
Let $G$ be a countable linear group. Then the following are equivalent:
\begin{enumerate}
\item $G$ is stable.
\item The pair $(G, \mathscr{I}(G))$ does not have relative property \emph{(T)}.
\end{enumerate}
\end{theorem1}

\begin{remark}\label{rem:minC}
The hypothesis that $G$ is linear in Theorems \ref{thm:I(G)}.ix.-xiv., \ref{thm:linear}, \ref{thm:Schmidt}, and \ref{thm:linearstable} can be weakened: we only need to assume that $G$ satisfies the {\bf minimal condition on centralizers}, that is, every decreasing sequence $C_G(A_0)\geq C_G(A_1)\geq \cdots$ of centralizers of subsets of $G$ eventually stabilizes. Every linear group satisfies the minimal condition on centralizers, since centralizers of arbitrary subsets of $\mathrm{GL}_n(F)$ are closed in the Zariski topology. \end{remark}

In addition to Theorem \ref{thm:I(G)}, an essential component in the proof of Theorem \ref{thm:linearstable} is the following extension theorem for stability (see \S\ref{sec:Kida} for the definition of stability sequence).

\begin{theorem1}\label{thm:extension}
Let $1\ra N\ra G\ra K \ra 1$ be a short exact sequence of groups in which $K$ is amenable. Assume that there exists a probability measure preserving action $G\cc (X,\mu )$ such that the translation groupoid $N\ltimes (X,\mu )$ admits a stability sequence. Then $G$ is stable.
\end{theorem1}

Theorem \ref{thm:extension} has a variety of applications outside the context of linear groups. Under each of the following hypotheses {\bf (H1)}-{\bf (H6)}, the stability of $G$ will be established in \S\ref{sec:stability} by applying Theorem \ref{thm:extension} to an appropriate input action of $G$. The application of Theorem \ref{thm:extension} to groups satisfying {\bf (H4)} and the ensuing Corollary \ref{cor:gBS} were kindly suggested by Yoshikata Kida (remarking on an earlier draft of this paper), who had obtained stability of $G$ from {\bf (H4)} by different means.

\begin{theorem1}\label{thm:extension2}
Let $1\ra N\ra G\ra K \ra 1$ be a short exact sequence of groups in which $K$ is amenable. Then $G$ is stable provided at least one of the following hypotheses holds:
\begin{enumerate}
\item[{\bf (H1)}] $N = LM$, where $L$ and $M$ are commuting subgroups of $N$ which are normal in $G$, with $M$ amenable and $[N:L]=\infty$.

\item[{\bf (H2)}] There exists a central subgroup $C$ of $N$ such that the pair $(N,C)$ does not have relative property \emph{(T)}.

\item[{\bf (H3)}] There exists a sequence $L_0\leq L_1\leq \cdots$, of subgroups of $N$ with $N=\bigcup _{m\in \N}L_m$, and for each $m\in \N$ there exists a central subgroup $D_m$ of $L_m$ such that the pair $(L_m, D_m )$ does not have relative property \emph{(T)}.

\item[{\bf (H4)}] There exists a commensurated abelian subgroup $A$ of $G$ such that $N$ is the kernel of the modular homomorphism from $G$ into the abstract commensurator of $A$, and the pair $(N,A)$ does not have relative property \emph{(T)}.

\item[{\bf (H5)}] $N$ has the Haagerup property and is asymptotically commutative, i.e., there exists an injective sequence $(c_n)_{n\in \N}$ in $N$ such that each $h\in N$ commutes with $c_n$ for cofinitely many $n\in \N$.

\item[{\bf (H6)}] $N$ is doubly asymptotically commutative, i.e., there exist sequences $(c_n)_{n\in\N}$ and $(d_n)_{n\in \N}$ in $N$ such that $c_nd_n\neq d_nc_n$ for all $n\in \N$, and each $h\in N$ commutes with both $c_n$ and $d_n$ for cofinitely many $n\in \N$.
\end{enumerate}
\end{theorem1}

\begin{corollary1}[Y. Kida] \label{cor:gBS}
Let $G$ be a generalized Baumslag-Solitar group (i.e., the Bass-Serre fundamental group of a finite graph of infinite cyclic groups), or an HNN-extension of $\Z ^n$ relative to an isomorphism between two finite index subgroups. Then $G$ is stable.
\end{corollary1}

\begin{proof}[Proof of Corollary \ref{cor:gBS}]
Suppose first that $G$ is a generalized Baumslag-Solitar group. Then any vertex group $A\leq G$ is commensurated by $G$, and if $N$ denotes the kernel of the modular homomorphism from $G$ into the abstract commensurator $\mathrm{comm}(A)$ of $A$, then $G/N$ is abelian, since $\mathrm{comm}(A)$ is isomorphic to $\Q ^*$. By Corollary 1.7 of \cite{CV12}, $G$ has the Haagerup property, so the pair $(N,A)$ does not have property (T). The hypothesis {\bf (H4)} is therefore satisfied, so $G$ is stable by Theorem \ref{thm:extension2}.

The case where $G$ is an HNN-extension of $\Z ^n$ relative to an isomorphism between two finite index subgroups is similar. The image of $G$ under the modular homomorphism into the abstract commensurator of $\Z ^n$ is cyclic and, letting $N$ denote the corresponding kernel, the pair $(N,\Z ^n )$ does not have property (T) since by Corollary 1.7 of \cite{CV12}, $G$ has the Haagerup property. Hypothesis {\bf (H4)} once again holds, so $G$ is stable by Theorem \ref{thm:extension2}.
\end{proof}

\begin{example} (i) Let $K$ be an infinite amenable group acting on a countable set $X$, and let $H$ be any countable group. Then the restricted wreath product $H\wr _X K$ is stable. This is clear if $H$ is amenable, and it follows from Theorem \ref{thm:extension2} via {\bf (H1)} if $X$ is finite. In the remaining case, the group $\bigoplus _X H$ is doubly asymptotically commutative, so Theorem \ref{thm:extension2} applies to $G$ via {\bf (H6)}, using the short exact sequence $1\ra \bigoplus _X H \ra H\wr _X K \ra K \ra 1$.

(ii) Let $H$ be a group which is doubly asymptotically commutative. Let $\varphi : H\ra H$ be an injective homomorphism and let $G = \langle t, H \, | \, tht^{-1} = \varphi (h)\rangle$ be the associated ascending HNN-extension. Theorem \ref{thm:extension2} then shows that $G$ is stable, since we have a short exact sequence $1\ra N \ra G \ra \Z \ra 1$ in which the group $N = \bigcup _{i\in \N} t^{-i}Ht^i$ is doubly asymptotically commutative, and hence the hypothesis {\bf (H6)} holds. Similarly, if we instead assume that $H$ is an increasing union $H=\bigcup _m H_m$, where for each $m\in \N$ the pair $(H_m, Z(H_m))$ does not have property (T), then the any ascending HNN extension of $H$ will be stable via {\bf (H3)}.
\end{example}

Notably, Theorem \ref{thm:extension2} also applies to the group $H(\R )$, recently studied by Monod \cite{Mo13}, consisting of all homeomorphisms of the projective line $\bm{\mathrm{P}}^1$ which fix $\infty$ and are piecewise in $\mathrm{PSL}_2(\R )$ with respect to a finite subdivision of $\bm{\mathrm{P}}^1$. It is shown by Monod in \cite{Mo13} that $H(\R )$ does not contain any nonabelian free subgroups, and Theorem 1 of \cite{Mo13} exhibits a family of countable nonamenable subgroups of $H(\R )$. An explicit finitely presented nonamenable subgroup of $H(\R )$ is constructed by Lodha and Moore in \cite{LM13}. We now have the following.

\begin{theorem1}\label{thm:HR}
Every nontrivial countable subgroup of $H(\R )$ is stable.
\end{theorem1}

Since Thompson's group $F$ is a subgroup of $H(\R )$, Theorem \ref{thm:HR} implies:

\begin{corollary1}\label{cor:F}
Thompson's group $F$ is stable. In particular, $F$ and $F\times A$ are measure equivalent, where $A$ is any amenable group.
\end{corollary1}

Corollary \ref{cor:F} yields a new proof of the fact, due to L\"{u}ck \cite{Lu02} and also proved by Bader, Furman, and Sauer in \cite{BFS12}, that all $\ell ^2$-Betti numbers of $F$ vanish. Indeed, Gaboriau has shown that vanishing of $\ell ^2$-Betti numbers is an invariant of measure equivalence \cite{Ga02}, and by \cite{CG86} all $\ell ^2$-Betti numbers of $F\times \Z$ vanish.

\newpage

\setcounter{tocdepth}{3}
\tableofcontents

\section{Weak forms of normality}\label{sec:normality}
In this section we gather some facts about $wq^*$-normality (defined in \S\ref{subsec:intro1}), and we discuss several related normality conditions. We work in the following general setting. Fix an ambient group $G$ along with a nonempty collection $\mathscr{L}$ of subgroups of $G$ which is upward closed in $G$. Let $H\leq M$ be subgroups of $G$. We say that $H$ is {\bf $\mathscr{L}$-$q$-normal} in $M$, denoted $H\leq ^{\mathscr{L}}_{q} M$, if the set $\{ g\in M\csuchthat gHg^{-1}\cap H \in \mathscr{L}\}$ generates $M$. We say that $H$ is {\bf $\mathscr{L}$-$wq$-normal} in $M$, denoted $H\leq ^{\mathscr{L}}_{wq}M$, if there exists an ordinal $\lambda$ and an increasing sequence $(H_\alpha )_{\alpha \leq \lambda}$ of subgroups of $M$, with $H_0=H$ and $H_\lambda = M$, such that $\bigcup _{\beta <\alpha} H_\beta \, \leq ^{\mathscr{L}}_q H_\alpha$ for all $\alpha \leq \lambda$. The notions of $wq$-normality and $wq^*$-normality then correspond to taking $\mathscr{L}$ to be, respectively, the collection $\mathscr{I}$, of infinite subgroups of $G$, and the collection $\mathscr{N}$, of nonamenable subgroups of $G$. Given a $G$-set $X$, we will be interested in the collection
\[
\mathscr{N}^X = \{ H\leq G \csuchthat H\cc X \text{ is nonamenable}\} .
\]
For $S\subseteq G$ finite and $r>0$, the collection $\mathscr{L}^{S,r} = \{ H\leq G\csuchthat \phi _S (G/H ) < r \}$ (where $\phi _S (G/H )$ is defined by \ref{eqn:isop}) will also be of interest, albeit less directly than $\mathscr{N}^X$. Both of these collections are upward closed in $G$ (for $\mathscr{L}^{S,r}$ this follows from Lemma \ref{lem:push} below) and both are invariant under conjugation by $G$.

The following characterization of $\mathscr{L}$-$wq$-normality, along with its proof, is a straightforward extension of \cite[Lemma 5.2]{PT11}.

\begin{lemma}\label{lem:PT} Let $H\leq M$ be subgroups of $G$. Then $H\leq _{wq}^{\mathscr{L}} M$ if and only if for any intermediate proper subgroup $H\leq K\lneq M$ there exists $g\in M \setminus K$ such that $gKg^{-1}\cap K\in \mathscr{L}$.
\end{lemma}

Let $H\leq M$ be subgroups of $G$ and let $\wh{H}$ denote the union of all subgroups $L\leq M$ with $H\leq _{wq}^{\mathscr{L}}L$. We call $\wh{H}$ the {\bf $\mathscr{L}$-$wq$-closure of $H$ in $M$}.

\begin{proposition}\label{prop:wqclosure}
$\wh{H}$ is a subgroup of $M$. Moreover, $\wh{H}$ is the unique subgroup of $M$ satisfying $(i)$ $H\leq _{wq}^{\mathscr{L}} \wh{H}$ and $(ii)$ $g\wh{H}g^{-1} \cap \wh{H} \not\in \mathscr{L}$ for every $g\in M\setminus \wh{H}$.
\end{proposition}

\begin{proof}
By Zorn's Lemma the set $\{ K\leq M \csuchthat H\leq _{wq}^{\mathscr{L}}K \}$ contains a maximal element $L$. By definition, $L\subseteq \wh{H}$. A consequence of Lemma \ref{lem:PT} is that if $L_0,L_1\leq M$ are two subgroups of $M$ with $H\leq _{wq}^{\mathscr{L}} L_0$ and $H\leq _{wq}^{\mathscr{L}}L_1$, then $H\leq _{wq}^{\mathscr{L}} \langle L_0,L_1 \rangle$. It follows that $L=\wh{H}$. Properties $(i)$ and $(ii)$ are immediate. If $K$ is a subgroup of $M$ satisfying properties $(i)$ and $(ii)$ in place of $\wh{H}$, then $H\leq K\leq \wh{H}$ by property $(i)$ for $K$ and the definition of $\wh{H}$, and since $H\leq _{wq}^{\mathscr{L}}\wh{H}$, Lemma \ref{lem:PT} and property $(ii)$ for $K$ imply that $K=\wh{H}$.
\end{proof}

\begin{definition}\label{def:nth}
Let $\mathscr{L}_0=\mathscr{L}$ and for each $n\geq 0$ define $\mathscr{L}_{n+1} = \{ H \leq G \csuchthat H\leq _{wq}^{\mathscr{L}_n}G \}$, which is upward closed by Lemma \ref{lem:PT} and induction. A subgroup $H$ of $G$ is said to be {\bf $n$-degree $\mathscr{L}$-$wq$-normal} in $G$ if $H\in \mathscr{\mathscr{L}}_n$.
\end{definition}

Note that if the collection $\mathscr{L}$ is invariant under conjugation by $G$, then so are each of the collections $\mathscr{L}_n$, $n\in \N$. By applying this definition to the collection $\mathscr{N}^X$, for a $G$-set $X$, we obtain the sequence $\mathscr{N}^X_n$, $n\in \N$. If $H\in \mathscr{N}_n=\mathscr{N}^G_n$, then we say that $H$ is {\bf $n$-degree $wq^*$-normal in $G$}. Thus, $H$ is $0$-degree $wq^*$-normal in $G$ if and only if $H$ is nonamenable, and $H$ is $1$-degree $wq^*$-normal in $G$ if and only if $H$ is $wq^*$-normal in $G$ in the previously defined sense.

For the next proposition, we equip the space of subgroups of $G$ with the subspace topology inherited from the product topology on $2^G$. Note that for any $G$-set $X$, the collection $\mathscr{N}^X$ is an open set, since for a subgroup $H\leq G$, nonamenability of the action $H\cc X$ is witnessed by a finite subset of $H$. The same holds for the collection $\mathscr{L}^{S,r}$, as well as for each $\mathscr{N}^X_n$, $n\in \N$, when $G$ is finitely generated. Taking $\mathscr{L}=\mathscr{N}^X_n$ in the following proposition then shows that, when $G$ is finitely generated, the transfinite sequence in the definition of $\mathscr{L}$-$wq$-normality can be replaced by a finite sequence.

\begin{proposition}\label{prop:Chaubaty}
Let $\mathscr{L}$ be an upward closed collection of subgroups of the countable group $G$. Suppose in addition that $\mathscr{L}$ is open in the space of subgroups of $G$.
\begin{enumerate}
\item[(i)] Let $H$ and $M$ be subgroups of $G$. Assume that $M$ is finitely generated and $H \leq _{wq}^{\mathscr{L}}M$. Then there exist subgroups $H_0,\dots , H_n$ such that
\begin{equation}\label{eqn:nseq}
H=H_0\leq _q^{\mathscr{L}}H_1\leq _q^{\mathscr{L}}\cdots \leq _q^{\mathscr{L}} H_n= M .
\end{equation}
Moreover, for any such sequence $(H_i)_{i=0}^n$ there exists a sequence $(H_i')_{i=0}^n$ with $H_0'\leq _q^{\mathscr{L}}H_1'\leq _q^{\mathscr{L}}\cdots \leq _q^{\mathscr{L}} H_n'= M$, where $H_i'$ is finitely generated and $H_i'\leq H_i$ for all $i$.

\item[(ii)] If $G$ is finitely generated then $\mathscr{L}_n$ is open for all $n\geq 0$.
\end{enumerate}
\end{proposition}

\begin{proof}
For $M,K\leq G$ define $f_M(K) = \langle \{ g\in M\csuchthat gKg^{-1}\cap K \in \mathscr{L} \}\rangle$. Then $f_M$ is monotone and, since $\mathscr{L}$ is open and upward closed, the function $f_M$ is lower semicontinuous, that is, $f_M(\liminf _i K_i ) \leq \liminf _i f_M(K_i)$ for any sequence $(K_i)_{i\in \N}$ of subgroups of $G$, where $\liminf _i K_i$ denotes the subgroup of elements of $G$ which are in cofinitely many $K_i$. It follows that for any finite sequence $M_0, M_1, \cdots , M_n \leq G$, the function $f_{M_n}\circ\cdots \circ f_{M_1}\circ f_{M_0}$ is lower semicontinuous.

(i): Let $H_{\omega} = \bigcup _{n\in \N}f^n_M (H)$. Then semicontinuity of $f_M$ implies $f_M(H_\omega )=H_\omega$. This shows that $H_\omega$ is the $\mathscr{L}$-$wq$-closure of $H$ in $M$, hence $H_\omega = M$. Since $M$ is finitely generated and the sequence $f^n_M(H)$ is nondecreasing there exists an $n$ with $f^n_M(H) = M$. This shows the first part of (i). Fix now any sequence $(H_i)_{i=0}^{n}$ as in \eqref{eqn:nseq}. Then $f_{H_n}\circ \cdots \circ f_{H_1}(H_0) = M$, so there exists a finitely generated $H_0'\leq H_0$ with $f_{H_n}\circ \cdots \circ f_{H_1}(H_0')=M$. Assume now that $k< n-1$ and $f_{H_n}\circ\cdots \circ f_{H_{k+1}}(H_k') = M$. Let $Q_0\subseteq Q_1\subseteq \cdots$ be a sequence of finite sets which exhaust $\{ g\in H_{k+1}\csuchthat gH_k'g^{-1} \cap H_k'\in \mathscr{L}\}$. Then $\bigcup _i \langle H_k', Q_i \rangle = f_{H_{k+1}}(H_k' )$, so there exists some $i$ such that $f_{H_n}\circ \cdots \circ f_{H_{k+2}}(\langle H_k', Q_i \rangle ) = M$. Take $H_{k+1}'=\langle H_k', Q_i\rangle$. The resulting groups $H_0',H_1',\dots ,H_{n-1}',H_n'=M$ satisfy the conclusion of (ii) by construction.

(ii): It suffices to show $\mathscr{L}_1$ is open. This follows from (i) and semicontinuity of $f_G^n$.
\end{proof}

\begin{remark}
In \cite{BFS12}, Bader, Furman, and Sauer define higher order notions of $s$-normality and establish a connection with higher $\ell ^2$-Betti numbers. It seems reasonable to expect a similar connection to hold between higher degree $wq$-normality (or some variant) and higher $\ell ^2$-Betti numbers, although this is largely speculative.
\end{remark}

\section{Amenable actions}\label{sec:isoperim}

Let $X$ be a $G$-set. Let $S\subseteq G$ be finite and let $\epsilon >0$. A nonempty finite subset $P$ of $X$ is said to be {\bf $(S,\epsilon )$-invariant} if $\sum _{s\in S}|sP\setminus P | < \epsilon |P|$. Equivalently, $P$ is $(S,\epsilon )$-invariant if $\sum _{s\in S}|sP\cap P| >(|S|-\epsilon )|P|$.

\begin{remark}\label{rem:bound}
Assume that $S$ generates $G$ and that every $G$-orbit has cardinality greater than $1/\epsilon$. Then any $(S,\epsilon )$-invariant set $P$ has cardinality $|P|>1/\epsilon$. Otherwise we would have $\sum _{s\in S}|sP\setminus P| < 1$, so $P$ would be a $G$-invariant set of cardinality at most $1/\epsilon$, a contradiction.
\end{remark}

\begin{remark}\label{rem:GM07}
We will make use of the observation \cite[Remark 2.12]{GM07} that if $P\subseteq X$ is $(S,\epsilon )$-invariant, then there exists a single $G$-orbit $X_0\subseteq X$ such that $P\cap X_0$ is $(S,\epsilon )$-invariant.
\end{remark}

\subsection{An estimate with F{\o}lner sets}\label{sec:Folner}
For each $n\geq 1$ we let $X^{\circledast n}$ denote the set of all $n$-tuples of distinct points in $X$ which lie in the same $G$-orbit
\[
X^{\circledast n} = \{ (x_0,\dots , x_{n-1}) \in X^n \csuchthat i\neq j \, \Ra \, G x_i = Gx_j \text{ and } x_i\neq x_j \} .
\]
Then we have a natural action $G\cc X^{\circledast n}$ under which the inclusion map $X^{\circledast n} \hookrightarrow X^n$ is a $G$-map to the diagonal product action. For a subset $P\subseteq X$ let $P^{\circledast n} = P^n\cap X^{\circledast n}$.

\begin{lemma}\label{lem:ndiag}
Let $S$ be a finite subset of $G$. Let $n\geq 1$ and let $\epsilon >0$. Let $P\subseteq X$ be an $(S,\epsilon )$-invariant set which is contained in a single $G$-orbit, and assume $|P|\geq n$. Then $P^{\circledast n}$ is $(S, n\epsilon )$-invariant in $X^{\circledast n}$.
\end{lemma}

\begin{proof}
For each $s\in S$ let $\epsilon _s = \tfrac{|sP\setminus P|}{|P|}$. 
Then $\sum _{s\in S}\epsilon _s < \epsilon$, so it suffices to show that for all $k\leq n$ we have
\begin{equation}\label{eqn:inter}
|sP^{\circledast k}\cap P^{\circledast k}|\geq |P^{\circledast k}|(1-k\epsilon _s ) .
\end{equation}
If $k=1$ then we have equality, so assume inductively that \eqref{eqn:inter} holds, where $k<n$, and we will show that it holds with $k+1$ in place of $k$. Note that $\epsilon _s >0$ implies $\epsilon _s \geq 1/|P|$. It follows that $(1- \tfrac{|P|\epsilon _s}{|P|-k} )(1-k\epsilon _s )\geq (1-(k+1)\epsilon _s )$, and hence
\begin{align*}
|s P^{\circledast (k+1)}\cap P^{\circledast (k+1)} | &=|(s P\cap P )^{\circledast (k+1)} | \\
&= (|s P \cap P | - k)|(s P\cap P )^{\circledast k}| \\
&\geq (|P|(1- \epsilon _s ) - k)|P^{\circledast k}|(1-k\epsilon _s ) \\
&= (1- \tfrac{|P|\epsilon _s}{|P|-k} )(|P|-k)|P^{\circledast k}|(1-k\epsilon _s ) \\
&\geq |P^{\circledast (k+1)}|(1-(k+1)\epsilon _s ). \qedhere
\end{align*}
\end{proof}

\subsection{Proof of Theorems \ref{thm:wqstar} and \ref{thm:isoperim}}\label{subsec:isoperim} Assume now that $G$ is a finitely generated by $S$. Let $\phi _S(X)$ denote the isoperimetric constant of $X$ with respect to $S$, defined in \eqref{eqn:isop}.

\begin{lemma}\label{lem:push}
Let $X$ and $Y$ be $G$-sets and assume that there exists a $G$-map $\varphi : X\ra Y$ from $X$ to $Y$. Then, given any $P\subseteq X$ which is $(S,\epsilon )$-invariant in $X$, we may find some $Q\subseteq \varphi (P)$ which is $(S,\epsilon )$-invariant in $Y$. In particular $\phi _S (Y)\leq \phi _S(X)$.
\end{lemma}

\begin{proof}
See the first proof in {\S}1.2 of \cite{Gr08}.
\end{proof}

\begin{lemma}\label{lem:peramen}
Let $X$ be a $G$-set and let $H$ be a subgroup of $G$. Assume that the action $H\cc X$ is amenable. Then $\phi _S(G/H ) \geq \phi _S (X)$.
\end{lemma}

\begin{proof}
Let $Y_0$ denote the $G$-set $G/H \times X$ equipped with the diagonal product action of $G$. Then $\phi _S(Y_0)\geq \phi _S (X)$ by Lemma \ref{lem:push}, so it suffices to show that $\phi _S (G/H ) \geq \phi _S (Y_0)$ (so in fact $\phi _S (G/H ) = \phi _S (Y_0)$ by Lemma \ref{lem:push}). Fix a section $\sigma :G/H \ra G$ for the map $G\ra G/H$, and let $\rho :G\times G/H \ra H$ be the corresponding Schreier cocycle given by $\rho (g,kH)= \sigma (gkH)^{-1}g\sigma (kH)$. Let $Y_1$ denote the $G$-set $G/H \times X$ equipped with the action $g\cdot (kH , x) = (gkH, \rho (g,kH ) \cdot x )$, for $kH\in G/H$, $x\in X$ (the $G$-set $Y_1$ is isomorphic to the $G$-set obtained by inducing from the $H$-set $X$, described in $\S$2.C of \cite{GM07}). The map $\varphi : Y_0 \ra Y_1$ given by $\varphi (kH ,x ) = (kH ,\sigma (kH)^{-1}\cdot x )$ provides an isomorphism between the $G$-sets $Y_0$ and $Y_1$, so it suffices to show that $\phi _S(G/H ) \geq \phi _S (Y_1)$. Let $\emptyset\neq P\subseteq G/H$ be finite. Then, for each finite $\emptyset\neq Q\subseteq X$, we have $P\times Q \subseteq Y_1$, so $\phi _S(Y_1)$ is bounded above by
\begin{align*}
\sum _{s\in S}\frac{|s\cdot (P\times Q) \setminus (P\times Q) |}{|P\times Q |} = \frac{1}{|P|}\sum _{s\in S}\sum _{kH \in sP\cap P}\frac{|\rho (s,s^{-1}kH)\cdot Q\setminus Q |}{|Q|} + \sum _{s\in S}\frac{|sP\setminus P|}{|P|} .
\end{align*}
Since $H\cc X$ is amenable, taking the infimum over all such $Q\subseteq X$ shows that $\phi _S(Y_1) \leq \sum _{s\in S} \frac{|sP\setminus P|}{|P|}$, so taking the infimum over $P$ shows that $\phi _S(Y_1)\leq \phi _S(G/H )$.
\end{proof}

\begin{proof}[Proof of Theorem \ref{thm:isoperim}]
The base case $n=0$ is immediate from Lemma \ref{lem:peramen}. Assume now that $\phi _S (G/H )<\frac{1}{2^n}\phi _S (X)$, where $n>0$, and we will show that $H$ is $n$-degree $\mathscr{N}^X$-$wq$-normal in $G$. By Lemma \ref{lem:PT}, given $H\leq L \lneq G$, it suffices to find some $g\in G\setminus L$ such that $gLg^{-1}\cap L$ is $(n-1)$-degree $\mathscr{N}^X$-$wq$-normal in $G$. The quotient map $G/H \ra G/L$ is a $G$-map, so by Lemma \ref{lem:push} we have $\phi _S(G/L)\leq \phi _S(G/H)< \frac{1}{2^n}\phi _S(X)$. Then $\phi _S((G/L )^{ \circledast 2} ) < \frac{1}{2^{n-1}}\phi _S(X)$ by Lemma \ref{lem:ndiag}, so by Remark \ref{rem:GM07} there is some point $x=(g_0L,g_1L)\in (G/L)^{\circledast 2}$ with $\phi _S (G/G_x ) < \frac{1}{2^{n-1}}\phi _S(X)$. By the induction hypothesis, the group $G_x=g_0Lg_0^{-1}\cap g_1Lg_1^{-1}$ is $(n-1)$-degree $\mathscr{N}^X$-$wq$-normal in $G$. Then $g=g_1^{-1}g_0\in G \setminus L$ and $gLg^{-1}\cap L$ is $(n-1)$-degree $\mathscr{N}^X$-$wq$-normal in $G$.
\end{proof}

\begin{proof}[Proof of Theorem \ref{thm:wqstar}]
This follows immediately from Theorem \ref{thm:isoperim} and the observation that for any $\epsilon >0$, the set $\{ x\in X \csuchthat \phi _S(G/G_x)<\epsilon \phi _S(G) \}$ is $\bm{m}$-conull.
\end{proof}

\subsection{An extension to infinitely generated groups} Example \ref{ex:fg} shows that a direct translation of Theorem \ref{thm:wqstar} does not hold in the general infinitely generated setting. However, a refined version of Theorem \ref{thm:wqstar} still holds in general. In what follows, for each $G$-set $X$ let $X_0 = \{ x\in X\csuchthat G_x \text{ is nonamenable}\}$.

\begin{corollary}\label{cor:obstr}
Let $G$ be a nonamenable group. For each $G$-set $X$ there is a $G$-map $\varphi _X : X\ra \wh{X}$ to a $G$-set $\wh{X}$ with the following properties:
\begin{enumerate}
\item[(i)] $G_x$ is $wq^*$-normal in $G_{\varphi _X(x)}$ for all $x\in X_0$;
\item[(ii)] If $\bm{m}$ is any $G$-invariant mean on $X$, then for any finitely generated subgroup $H$ of $G$ we have $H\leq G_{\varphi _X(x)}$ for $\bm{m}$-almost every $x\in X$.
\end{enumerate}
Moreover, this assignment can be made functorial: if $\psi : X\ra Z$ is a $G$-map then there exists a unique $G$-map $\wh{\psi}: \wh{X} \ra \wh{Z}$ with $\varphi _Z\circ \psi = \wh{\psi}\circ \varphi _X$.
\end{corollary}

\begin{proof}
For each nonamenable subgroup $H\leq G$ let $\wh{H}$ denote the $wq^*$-closure of $H$ in $G$ (see \S\ref{sec:normality}). For $x\in X$ let $O(x)$ denote the $G$-orbit of $X$. Let $G$ act on the set $\wh{X}_0 = \{ (\wh{G}_x , O(x)) \csuchthat x\in X_0 \}$ by conjugating the first coordinate, and let $\wh{X} = \wh{X}_0\sqcup X\setminus X_0$. Define $\varphi _X : X\ra \wh{X}$ by $\varphi _X(x) = (\wh{G}_x , O(x))$ for $x\in X_0$, and $\varphi _X(x)=x$ for $x\in X\setminus X_0$. Then $\varphi _X$ is a $G$-map, and for each $x\in X_0$ we have $G_{\varphi _X(x)} = \wh{G}_x$ since $\wh{G}_x$ is self-normalizing. This verifies (i). For (ii), let $\bm{m}$ be a $G$-invariant mean on $X$ and let $H\leq G$ be finitely generated. After making $H$ larger we may assume that $H$ is nonamenable. By Theorem \ref{thm:wqstar}, $H_x$ is $wq^*$-normal in $H$ for $\bm{m}$-almost every $x\in X$. For each such $x$, since $H_x\leq G_x$, we have $H\leq \wh{G}_x = G_{\varphi _X(x)}$.

If $\psi : X \ra Z$ is a $G$-map, then we must show that $\varphi _Z(\psi (x))$ only depends on $\varphi _X(x)$. This is clear for $x\in X\setminus X_0$. Suppose now that $(\wh{G}_{x_0} , O(x_0) ) = (\wh{G}_{x_1},O(x_1))$ where $x_0 , x_1\in X_0$. Find $g\in G$ with $gx_0 = x_1$. Then $\wh{G}_{x_0}=\wh{G}_{x_1} = g\wh{G}_{x_0}g^{-1}$, so $g\in \wh{G}_{x_0} \leq \wh{G}_{\psi (x_0)}$. It follows that $\wh{G}_{\psi (x_1)} = \wh{G}_{\psi (gx_0)} = g\wh{G}_{\psi (x_0)}g^{-1} = \wh{G}_{\psi (x_0)}$, hence $\varphi _Z(\psi (x_0))=\varphi _Z(\psi (x_1))$.
\end{proof}

\section{Transitive amenable actions}

\subsection{Weak normality for groupoids}\label{sec:groupoid}
To prove Theorem \ref{thm:coamen} we need an extension of the results of \cite{PT11} on weakly normal inclusions of discrete probability measure preserving (p.m.p.) groupoids. We adopt the notation and conventions for discrete p.m.p.\ groupoids from \cite[\S 6]{PT11}, and we will need a few additional definitions.

Let $(\mc{G},\mu )$ be a discrete p.m.p.\ groupoid. We do not distinguish two subgroupoids $\mc{H}$ and $\mc{K}$ of $\mc{G}$ if they agree off of a $\mu$-null set. Recall that a {\bf local section} of $\mc{G}$ is a measurable map $\phi : \dom (\phi ) \ra \mc{G}$, with $\dom (\phi ) \subseteq \mc{G}^0$ and $s(\phi x) = x$ for all $x\in \dom (\phi )$, such that the assignment $\phi ^0 : x\mapsto r(\phi x )$ is injective. We do not distinguish two local sections whose domains and values agree off of a $\mu$-null set. Let $[[\mc{G}]]$ denote the collection of all local sections of $\mc{G}$. The {\bf inverse} of $\phi \in [[\mc{G}]]$ is the local section $\phi ^{-1} :\ran (\phi ^0 ) \ra \mc{G}$ given by $\phi ^{-1}(y) = \phi ((\phi ^0)^{-1}y)^{-1}$. The {\bf composition} of two local sections $\phi , \psi \in [[\mc{G}]]$ is the local section $\phi\circ \psi : (\psi ^0)^{-1}(\ran (\psi ^0 ) \cap \dom (\phi ^0 ))  \ra \mc{G}$, $x\mapsto \phi (\psi ^0(x))\psi (x)$.

We equip $[[\mc{G}]]$ with the separable complete metric $d(\phi , \psi  ) = \mu (\dom (\phi )\triangle \dom (\psi ) ) + \mu (\{ x\in \dom (\phi ) \cap \dom (\psi ) \csuchthat \phi (x) \neq \psi (y) \} )$. A consequence of separability of the metric $d$ is that if $\Phi$ is any subset of $[[\mc{G}]]$ then up to a $\mu$-null set there is a unique smallest subgroupoid $\mc{K}$ of $\mc{G}$ with $\Phi \subseteq [[\mc{K}]]$; we call $\mc{K}$ the {\bf subgroupoid generated by $\Phi$} and denote it by $\langle \Phi \rangle$.

For measurable subsets $R\subseteq \mc{G}$ and $A\subseteq \mc{G}^0$ we let $R_A = \{ \gamma \in R \csuchthat s(\gamma ),r(\gamma )\in A \}$. For $\phi \in [[\mc{G}]]$ and $\gamma \in \mc{G}_{\ran (\phi ^0)}$ let $\gamma ^{\phi }=\phi ^{-1}(r(\gamma ))\gamma \phi ^{-1}(s(\gamma ))^{-1} \in \mc{G}_{\dom (\phi ^0)}$. The {\bf $q$-normalizer} of $R$ in $\mc{G}$ is the set
\[
Q_{\mc{G}}(R) = \{ \phi \in [[\mc{G}]] \csuchthat (R_A)^{\phi } \cap R_{(\phi ^0)^{-1}A} \text{ has infinite measure for all non-null }A\subseteq\ran (\phi ^0) \} .
\]
A subgroupoid $\mc{H}$ of $\mc{G}$ is said to be {\bf $q$-normal} in $\mc{G}$ if $Q_{\mc{G}}(\mc{H})$ generates $\mc{G}$. As usual, we obtain the corresponding notion of $wq$-normality by iterating $q$-normality transfinitely. Then the analogue of Lemma \ref{lem:PT} holds: $\mc{H}$ is $wq$-normal in $\mc{G}$ if and only if for every intermediate proper subgroupoid $\mc{H}\subseteq \mc{K}\subsetneq \mc{G}$ there exists a local section $\phi \in [[\mc{G}]] \setminus [[\mc{K}]]$ with $\phi \in Q_{\mc{G}}(\mc{K})$. While Theorem 6.9 of \cite{PT11} is stated for $ws$-normal subgroupoids, we note that the proof holds more generally for $wq$-normal subgroupoids.

\begin{theorem}[\emph{cf.} {\cite[Theorem 6.9]{PT11}}]\label{thm:restr}
Let $\mc{H}$ be a subgroupoid of the discrete p.m.p.\ groupoid $(\mc{G},\mu )$. If $\mc{H}$ is $wq$-normal in $\mc{G}$ then the restriction map $H^1(\mc{G},\mc{U}(\mc{G},\mu )) \ra H^1(\mc{H},\mc{U}(\mc{G},\mu ))$ is injective.
\end{theorem}

\begin{proof}
The proof of Theorem 6.9 of \cite{PT11} shows that if $c$ is a $\mc{G}$-cocycle with values in $\mc{U}(\mc{G},\mu )$ which vanishes on $[[\mc{H}]]$, then $c$ vanishes on $Q_{\mc{G}}(\mc{H})$, and therefore on $\langle Q_{\mc{G}}(\mc{H})\rangle$ since the set where $c$ vanishes is closed under compositions and inverses, and $c$ respects countable decompositions. The theorem follows. (We note the following minor correction to the proof of Theorem 6.9 of \cite{PT11}: using the notation from that proof, the fact that $\mc{H}_A$ is $s$-normal in $\mc{G}_A$ is irrelevant to the proof; what is being used is that $(\chi _A \psi )^{-1}\mc{H}_A(\chi _A \psi ) \cap \mc{H}_{(\psi ^0)^{-1}A}$ has infinite measure, which holds since $\mc{H}$ is $s$-normal in $\mc{G}$ and hence $\psi \in Q_{\mc{G}}(\mc{H})$. The rest of the proof remains unchanged after replacing $A$ by $(\psi ^0)^{-1}A$ in the appropriate places.)
\end{proof}

\subsection{Recurrence and normality} Let $G\cc (Y,\nu )$ be a free probability measure preserving action of $G$. We let $\mc{R}^G$ denote the orbit equivalence relation generated by this action. Then $(\mc{R}^G , \nu )$ is a discrete p.m.p.\ groupoid so that the notation and terminology of \S\ref{sec:groupoid} applies. In this case, we will identify each local section $\phi \in [[ \mc{R}^G]]$ with the corresponding partial isomorphism $\phi ^0$ of $(Y,\nu )$, and we identify elements of $G$ with their image in $[[\mc{R}^G]]$.

For each subset $P\subseteq G$ let $R^P\subseteq \mc{R}^G$ denote the graph $R^P = \{ (y,k y ) \csuchthat y\in Y, \ k\in P \}$. Define the sets
\begin{align*}
Q(P) &= \{ g\in G\csuchthat (\forall n )(\exists \text{distinct }k_0,\dots , k_{n-1}\in G )\ ( k_i^{-1}k_j \in P\cap P^g \text{ for all }i<j<n ) \} \\
L(P) &= \{ g\in G\csuchthat (\forall n )(\exists \text{distinct }k_0,\dots , k_{n-1}\in G )\ ( k_i^{-1}k_j \in P\cap g^{-1}P \text{ for all }i<j<n ) \} .
\end{align*}
For a subgroup $H\leq G$ we then have $R^H = \mc{R}^H$ and $Q(H) = \{ g\in G\csuchthat H\cap H^g\text{ is infinite} \}$.

\begin{lemma}\label{lem:thick1}
Let $G\cc (Y,\nu )$ be a free probability measure preserving action of $G$. Let $A\subseteq Y$ be measurable and let $P=P^{-1}$ be a subset of $G$.
\begin{enumerate}
\item[(i)] If $g\in Q(P)$ then $g_{|(A\cap g^{-1}A)} \in Q_{\mc{R}^G}(R^P_A)$. In particular, if $H$ is an infinite subgroup of $G$ then $\mc{R}^H_A$ is $q$-normal in the equivalence relation generated by $R^{Q(H)}_A$.
\item[(ii)] Let $d^P_A,\, d^{L(P)}_A:Y\times Y\ra \N \cup \{ \infty\}$ denote the extended graph metrics on $R^P_A$ and $R^{L(P)}_A$ respectively. Then $d^P_A(x,y) \leq 2d^{L(P)}_A(x,y)$ for almost every $(x,y)\in \mc{R}^G$. In particular, if $L(P) =G$ then $R^P_A$ generates $\mc{R}^G_A$.
\end{enumerate}
\end{lemma}

\begin{proof}[Proof of Lemma \ref{lem:thick1}]
(i) Fix $g\in Q(P)$. It suffices to show that for almost every $y\in A\cap g^{-1}A$, the set $\{ z\in A \csuchthat (z,y), (gz,gy)\in R^{P}_A \}$ is infinite. Suppose toward a contradiction that there exists an $m>0$ such that the set
\[
C = \{ y\in A\cap g^{-1}A \csuchthat |\{ z\in A\csuchthat (z,y), (gz,gy)\in R^P_A \} | < m \}
\]
has positive measure, say $\nu (C)=\epsilon >0$. By the Poincar\'{e} recurrence theorem there exists some $n\in \N$, depending only on $\epsilon$ and $m$, such that if $(C_i)_{i<n}$ is any sequence of measurable sets in $Y$, each with $\nu (C_i)\geq \epsilon$, then there exists $i_0<i_1<\cdots <i_{m}<n$ with $\nu (\bigcap _{j<m}C_{i_j}) >0$. Using this $n$, let $(k_i)_{i<n}$ be a sequence as in the definition of $g\in Q(P)$. By our choice of $n$ there exists $i_0<i_1<\cdots <i_{m}<n$ with $\nu (\bigcap _{j\leq m}k_{i_j}C ) > 0$. For each $0\leq j <m$ let $h_j = k_{i_m}^{-1}k_{i_j}$ so that $h_j^{-1} \in P\cap P^g$ and $\nu (C \cap \bigcap _{j<m} h_j C ) >0$, and the elements $h_0,\dots , h_{m-1}$ are pairwise distinct. Fix $y\in C\cap \bigcap _{j<m} h_jC$ and fix any $j<m$ and put $h=h_j$. Then $y, h^{-1}y\in C \subseteq A\cap g^{-1}A$, so $y, h^{-1}y, gy, gh^{-1}y\in A$. Moreover, $h^{-1}\in P$ and $gh^{-1}g^{-1}\in P$, so it follows that $(h^{-1}y, y ) \in R^P_A$ and $(gh^{-1}y, gy ) = ((gh^{-1}g^{-1})gy, gy )\in R^P_A$. This shows that $\{ h_j ^{-1}y \} _{j=0}^{m-1}\subseteq \{  z\in A \csuchthat (z,y), (gz,gy ) \in R^P_A \}$, which contradicts that $y\in C$.

(ii) It suffices to show that $d^P_A(gy, y) \leq 2$ for all $g\in L(P)$ and almost every $y\in A\cap g^{-1}A$. Suppose toward a contradiction that there exists some $g\in L(P)$ such that the set
\[
D = \{ y\in A\cap g^{-1}A \csuchthat d^P_A(gy, y) > 2 \}
\]
has positive measure. Let $n\in \N$ be so large that $\tfrac{1}{n} < \nu (D)$. Let $(k_i)_{i<n}$ be a sequence as in the definition of $g\in L(P)$. Then $\nu (k_i D ) = \nu (D) >\tfrac{1}{n}$ for all $0\leq i<n$, so there exists $i<j<n$ with $\nu (k_i D  \cap k_j D) > 0$. Let $k=k_i^{-1}k_j$, so that $k\in P\cap g^{-1}P$, and the set $D_0 := D\cap kD$ is non-null. Fix $y\in D_0$. Then we have $y, gy ,k^{-1}y \in A$ and $k, gk \in P$, so $(k^{-1}y ,gy ) = (k^{-1}y, (gk)k^{-1}y) \in R^P_A$ and $(k^{-1}y,y)= (k^{-1}y,k(k^{-1}y))\in R^P_A$. This shows that $d^P_A(gy , y ) \leq 2$, which contradicts that $y\in D$.
\end{proof}

Example \ref{ex:ws} shows that there are nonamenable groups having a transitive amenable action $G\cc X$ such that $G_x$ is not $ws$-normal in $G$. The next Lemma shows that $G_x$ is still very close to being $s$-normal in $G$. Recall that a subset $B$ of $G$ is said to be {\bf thick} in $G$ if for every finite subset $F\subseteq G$ the intersection $\bigcap _{g\in F}gB$ is nonempty (equivalently: infinite). Observe that if $B\subseteq G$ is thick then $L(B) = G$ since given $g\in G$ we can define $k_0=1$ and inductively let $k_{n+1}$ be any element of $\big( \bigcap _{i\leq n}k_i(B\cap g^{-1}B)\big)\setminus \{k_0,\dots , k_n \}$, so that $(k_n)_{n\geq 0}$ witnesses that $g\in L(B)$.

\begin{lemma}\label{lem:thick}
Let $G\cc X$ be a transitive amenable action of a nonamenable group $G$. Fix any element $x\in X$ and let $H=G_x$. Then $Q(H)$ is thick in $G$. In particular, $L(Q(H)) = G$.
\end{lemma}

\begin{proof} Since $H$ is a subgroup of $G$ we have $Q(H) = \{ g\in G \csuchthat H\cap H^g \text{ is infinite} \}$. Note that $H$ is nonamenable since $G$ is nonamenable and the action $G\cc G/H$ is amenable. Let $\bm{m}$ be a $G$-invariant mean on $G/H$. Then $\bm{m}$ is also $H$-invariant, so we obtain
\begin{equation}\label{eqn:K0}
\bm{m}(\{ gH\in G/H \csuchthat H\cap gHg^{-1} \text{ is nonamenable} \} )=1 .
\end{equation}
Let $\pi : G\ra G/H$ be the projection map $\pi (g)= gH$. Then \eqref{eqn:K0} implies that $\bm{m}(\pi (Q(H)))= 1$. If $g\in G$ then by $G$-invariance of $\bm{m}$ we have $\bm{m} (\pi (gQ(H))) = \bm{m}(g\pi (Q(H))) = \bm{m}(\pi (Q(H))) = 1$. Therefore, for any finite subset $F$ of $G$ we have $\bm{m}(\pi(\bigcap _{g\in F}gQ(H))) = \bm{m}(\bigcap _{g\in F}\pi (gQ(H))) = 1$, where the equality $\pi (\bigcap _{g\in F}gQ(H)) = \bigcap _{g\in F}\pi (gQ(H))$ follows from $Q(H)$ being a union of left cosets of $H$. In particular, $\bigcap _{g\in F}gQ(H)\neq \emptyset$.
\end{proof}

\subsection{Proof of Theorem \ref{thm:coamen}}

Using Theorem \ref{thm:restr} and Lemmas \ref{lem:thick1} and \ref{lem:thick}, we can now argue as in Theorem 5.12 of \cite{PT11}.

\begin{proof}[Proof of Theorem \ref{thm:coamen}]
We can of course assume that $G$ is nonamenable. Let $H= G_x$ and assume that $\beta ^{(2)}_1(H)<\infty$. Since $H$ is infinite index in $G$ there exists a free ergodic p.m.p.\ action $G\cc (Y,\nu )$ of $G$ whose restriction to $H$ has a continuum of ergodic components. Such an action may be obtained, e.g., by coinducing from any free p.m.p.\ action of $H$ with a continuum of ergodic components. Fix $n\geq 1$ and let $A_0,\dots , A_{n-1}$ be a partition of $Y$ into $H$-invariant sets of equal measure. Let $\mc{R}_n = \bigsqcup _{i<n}\mc{R}^G_{A_i}$. Observe that $\mc{R}^H\subseteq \mc{R}_n \subseteq \mc{R}^G$. For each $i<n$ let $\mc{S}_i$ be the equivalence relation generated by $R^{Q(H)}_{A_i}$. Then by Lemma \ref{lem:thick1}.(i), $\mc{R}^H$ is $q$-normal in $\bigsqcup _{i<n}\mc{S}_i$. By Lemma \ref{lem:thick} we have $L(Q(H))=G$, so Lemma \ref{lem:thick1}.(ii) implies that $\mc{S}_i=\mc{R}^G_{A_i}$ for each $i<n$, and hence $\bigsqcup _{i<n}\mc{S}_i = \mc{R}_n$. This shows that $\mc{R}^H$ is $q$-normal in $\mc{R}_n$ and hence
\[
\beta ^{(2)}_1 (\mc{R}_n) \leq \beta ^{(2)}_1 (\mc{R}^H )
\]
by Theorem \ref{thm:restr}. Since $G\cc (Y,\nu )$ is ergodic, each of the sets $A_i$ is a complete section for $\mc{R}^G$, so $[\mc{R}^G:\mc{R}_n] = n$. By Corollary 3.16 and Proposition 5.11 of \cite{Ga02} we therefore have
\[
\beta ^{(2)}_1(G) = \beta ^{(2)}_1(\mc{R}^G) = \frac{1}{n}\beta ^{(2)}_1(\mc{R}_n) \leq \frac{1}{n}\beta ^{(2)}_1 (\mc{R}^H) = \frac{1}{n}\beta ^{(2)}_1(H).
\]
Since $n\geq 1$ was arbitrary and $\beta ^{(2)}_1(H)<\infty$, we conclude that $\beta ^{(2)}_1(G)=0$.

Assume now that $\mathscr{PC}(H)<\infty$ and let $H\cc (Y_0,\nu _0)$ be a free p.m.p.\ action of $H$ with $\mathscr{PC}(\mc{R}^H_{Y_0}) < \infty$. After taking the product of this action with an identity action of $H$ on an atomless probability space, we may assume that $H\cc (Y_0,\nu _0)$ has continuous ergodic decomposition. Let $G\cc (Y,\nu )$ be the coinduced action, which is free and ergodic. Since $H\cc (Y,\nu )$ factors onto $H\cc (Y_0, \nu _0)$, we have $\mathscr{PC}(\mc{R}^H_Y)\leq \mathscr{PC}(\mc{R}^H_{Y_0}) <\infty$. The proof now proceeds as above, using Proposition \ref{prop:Furman1} in place of Theorem \ref{thm:restr}, and Proposition 25.7 of \cite{KM04} in place of Proposition 5.11 of \cite{Ga02} (the proof of Proposition 25.7 of \cite{KM04} works just as well for pseudocost as for cost).
\end{proof}

\section{The cost of inner amenable groups}

\subsection{Proof of Theorem \ref{thm:relmain1}}
We will often use the following well-known classical fact, which is a weakening of Theorem \ref{thm:wqstar}.

\begin{lemma}\label{lem:null}
Let $G\cc X$ be an amenable action of a nonamenable group $G$ and let $\bm{m}$ be a $G$-invariant mean on $X$. Then $G_x$ is nonamenable for $\bm{m}$-almost every $x\in X$.
\end{lemma}

The following simple consequence will be very useful.

\begin{lemma}\label{lem:conull}
Assume that the pair $(G,H)$ is inner amenable and let $\bm{m}$ be an $H$-conjugation invariant mean on $G$. Let $\mathscr{F}$ be a finite collection of nonamenable subgroups of $H$. Then
\[
\bm{m} ( \{ g\in G\csuchthat L\cap C_G(g) \text{ is nonamenable for all }L\in \mathscr{F} \} ) = 1 .
\]
\end{lemma}

\begin{proof}
It suffices to prove the lemma in the case where $\Es{F}= \{ L \}$ is a singleton. This follows from Lemma \ref{lem:null} by taking $X=G$ along with the conjugation action $L\cc X$.
\end{proof}

Theorem \ref{thm:relmain1} is an immediate consequence of the following more detailed analysis.

\begin{theorem}\label{thm:relmain}
Assume that the pair $(G,H)$ is inner amenable and that $H$ is nonamenable. For each nonamenable subgroup $L\leq H$ let $K_L =\langle \{ g\in G\csuchthat L\cap C_G(g)\text{ is nonamenable} \}\rangle$.
\begin{enumerate}
\item[(i)] Let $L$ be a nonamenable subgroup of $H$. Then $L\leq _{q^*}LK_L\leq _q \langle H,K_L\rangle \leq _{q^*}HK_H$. In particular, $L$ is $wq$-normal in $HK_{H}$, and $H$ is $q^*$-normal in $HK_H$.
\item[(ii)] Every $H$-conjugation invariant mean $\bm{m}$ on $G$ concentrates on $HK_H$. In particular, $\bm{m}$ concentrates on the $wq^*$-closure of $H$ in $G$.
\end{enumerate}
\end{theorem}

\begin{proof} Fix an atomless $H$-conjugation invariant mean $\bm{m}$ on $G$. Let $L\leq H$ be nonamenable. Let $S_L = \{ g\in G\csuchthat L\cap C_G(g)\text{ is nonamenable} \}$. Then $\bm{m}(S_L)=1$ by Lemma \ref{lem:conull}, and $S_L\subseteq \{ g\in G\csuchthat gLg^{-1}\cap L\text{ is nonamenable}\}$ implies that $L\leq _{q^*} LK_L$. Since $S_L\subseteq K_L$ we have $\bm{m}(LK_L) =1$, so $\bm{m}(hLK_Lh^{-1}\cap LK_L ) =1$ for all $h\in H$, and since $\bm{m}$ is atomless this shows that $LK_L\leq _q \langle H, K_L\rangle$. Then $\langle H,K_L\rangle \leq _{q^*}HK_H$ follows from $H\leq _{q^*}HK_H$.
\end{proof}

\subsection{Proof of Theorem \ref{thm:subgroup}}
We begin with the version of Theorem \ref{thm:subgroup} for inner amenable pairs. Let $\mathscr{N}(H)$ denote the collection of all nonamenable subgroups of $H$.

\begin{theorem}\label{thm:pairsubgroup}
Assume that the pair $(G,H)$ is inner amenable and that $H$ is nonamenable. Then at least one of the following holds:
\begin{enumerate}
\item[(1)] For every finite $\mathscr{F}\subseteq \mathscr{N}(H)$ there exists an infinite amenable subgroup $K$ of $G$ such that $L\cap C_G(K)$ is nonamenable for all $L\in \mathscr{F}$.
\item[(2)] For every finite $\mathscr{F}\subseteq \mathscr{N}(H)$ there exists an increasing sequence $M_0\leq M_1\leq \cdots$ of finite subgroups of $G$, with $\lim _{n\ra\infty}|M_n|=\infty$, such that $L\cap C_G(M_n)$ is nonamenable for all $L\in \mathscr{F}$, $n\in \N$.
\item[(3)] For every finite $\mathscr{F}\subseteq \mathscr{N}(H)$ there exists a nonamenable subgroup $K$ of $G$ such that $L\cap C_G(K)$ is nonamenable for all $L\in \mathscr{F}$.
\end{enumerate}
\end{theorem}

\begin{proof}
Assume that neither (1) nor (2) holds, as witnessed by the collections $\mathscr{F}_1$ and $\mathscr{F}_2$ respectively. We will show that (3) holds. Toward this end, fix $\mathscr{F} \subseteq \mathscr{N}(H)$ finite. We may assume that $\mathscr{F}_1\cup \mathscr{F}_2\subseteq \mathscr{F}$. Fix an atomless $H$-conjugation invariant mean $\bm{m}$ on $G$. Let $M_0 = \{ e \}$ and for each $L\in \mathscr{F}$ let $\varphi _0(L) = L$. Assume for induction that $\varphi _0(L)\geq \cdots \geq \varphi _n(L)$, $L\in \mathscr{F}$, and $M_0\lneq \cdots \lneq M_n$ have been defined with $\varphi _n (L)$ nonamenable and commuting with $M_n$ for all $L\in \mathscr{F}$. If $M_n$ is infinite then we stop; otherwise, since $\bm{m}$ is atomless we have $\bm{m}(M_n)=0$, so by Lemma \ref{lem:conull} there exists $g\in G\setminus M_n$ such that $\varphi _{n+1}(L):= \varphi _n(L)\cap C_G(g)$ is nonamenable for all $L\in \mathscr{F}$. The induction continues with $M_{n+1}:= \langle M_n,g\rangle$.

We claim that this process stops at some stage, i.e., there is some $n>0$ such that $M_n$ is infinite. Otherwise, if the process never stops, we would obtain an infinite sequence $M_0\lneq M_1\lneq \cdots$ of finite groups such that $\varphi _n(L)\leq L\cap C_G(M_n)$ for all $L\in \mathscr{F}_2$, contradicting our choice of $\mathscr{F}_2$. Let $n$ be the stage at which the process stops and take $K:=M_n$. Then for each $L\in \mathscr{F}$ we have $\varphi _n (L)\leq L\cap C_G(K)$, so $L\cap C_G(K)$ is nonamenable. Then $K$ must be nonamenable since $\mathscr{F}_1\subseteq \mathscr{F}$.
\end{proof}

We can now prove Theorem \ref{thm:subgroup}.

\begin{proof}[Proof of Theorem \ref{thm:subgroup}]
Assume that neither (1) nor (2) of Theorem \ref{thm:subgroup} holds and fix $\mathscr{F}\subseteq \mathscr{N}$ finite and $n\in \N$ toward the goal of verifying (3). We already know that the pair $(G,G)$ satisfies alternative (3) of Theorem \ref{thm:pairsubgroup}. We may therefore find a nonamenable subgroup $K_0^0\leq G$ such that the group $\psi _0(L):= L\cap C_G(K^0_0)$ is nonamenable for all $L\in \mathscr{F}$. Let $k\geq 0$ and assume for induction that we have defined the nonamenable subgroups $\psi _k(L)$, $L\in \mathscr{F}$, and nonamenable $K^k_0,K^k_1,\dots , K^k_k$, such that
\begin{itemize}
\item $K_i^k$ and $K_j^k$ commute for all $0\leq i<j\leq k$,
\item $\psi _k(L)\leq L$ and $\psi _k(L)$ and $K_i$ commute for all $L\in \mathscr{F}$ and $0\leq i\leq k$.
\end{itemize}
Now apply alternative (3) of Theorem \ref{thm:pairsubgroup} to $(G,G)$, using the finite collection $\{ \psi _k (L)\csuchthat L\in \mathscr{F}\}\cup \{ K^k_i\csuchthat 0\leq i\leq k \}$, to obtain a nonamenable subgroup $K^{k+1}_{k+1}$ of $G$ such that $\psi _{k+1}(L):=\psi _k(L)\cap C_G(K^{k+1}_{k+1})$ is nonamenable for all $L\in \mathscr{F}$, and $K^{k+1}_i := K^k_i\cap C_G(K^{k+1}_{k+1})$ is nonamenable for all $i\leq k$. This continues the induction. For each $i<n$ let $K_i = K^{n-1}_i$. Then $K_0,K_1, \dots , K_{n-1}$ are the desired subgroups for the first part of (3). For the second statement, we may assume that $\mathscr{F}$ contains a subset $\mathscr{F}_1$ witnessing that alternative (1) fails. For each $0\leq i<n$ inductively let $g_i$ be any element of $K_i\setminus \langle g_0,\dots , g_{i-1}\rangle$; this set is nonempty since $K_i$ is nonamenable and $\langle g_0,\dots , g_{i-1}\rangle$ is abelian. The group $M_n := \langle g_0,g_1,\dots ,g_{n-1}\rangle$ is an abelian group which commutes with $\psi (L)$ for all $L\in \mathscr{F}$, hence $L\cap C_G(M_n)$ is nonamenable for all $L\in \mathscr{F}$. Then $M_n$ is finite since $\mathscr{F}_1\subseteq \mathscr{F}$. By construction, $M_n$ has order $|M_n|\geq 2^n$.
\end{proof}

\subsection{Proof of Theorem \ref{thm:FP1}}\label{sec:FP1}

\begin{lemma}\label{lem:fin}
Let $M$ be a finite normal subgroup of $G$. Then $|M|(\mathscr{C}^*(G) -1)\leq \mathscr{C}^*(G/M) - 1$.
\end{lemma}

\begin{proof}
Let $G \cc (X,\mu )$ be a free p.m.p.\ action of $G$. Let $Y\subseteq X$ be a measurable transversal for $\mc{R}^M_X$, so $\mu (Y) =\tfrac{1}{|M|}$. Let $\mu _Y$ be the normalized restriction of $\mu$ to $Y$. Then $G/M$ acts on $(Y,\mu _Y)$ by the rule $gM\cdot y_0 = y_1$ if and only if $gMy_0 = My_1$. This is an action since $M$ is normal in $G$, and it is free and measure preserving. Fix $\epsilon >0$ and let $R$ be a graphing of $\mc{R}^{G/M}_Y$ with $\mathscr{C}_{\mu _Y} (R)<\mathscr{C}^*(G/M)+\epsilon$. Let $T$ be a treeing of $\mc{R}^M_X$. Then $\mathscr{C}_\mu (T)= 1-\tfrac{1}{|M|}$ and $T\cup R$ is a graphing of $\mc{R}^G_X$, hence
\begin{align*}
\mathscr{C}_\mu (\mc{R}^G_X)&\leq \mathscr{C}_\mu (T\cup R) = \mathscr{C}_\mu (T) + \mu (Y)\mathscr{C}_{\mu _Y} (R)\\
 &< 1-\frac{1}{|M|} + \frac{1}{|M|}(\mathscr{C}^*(G/M) + \epsilon ) = 1 + \frac{\mathscr{C}^*(G/M) -1}{|M|} + \epsilon /|M| .
\end{align*}
Since $\epsilon >0$ was arbitrary this shows that $\mathscr{C}_\mu (\mc{R}^G_X)\leq 1 + \frac{\mathscr{C}^*(G/M) -1}{|M|}$. Since this holds for all free p.m.p.\ actions $G\cc (X,\mu )$ the proof is complete.
\end{proof}

\begin{lemma}\label{lem:cost2}
Let $G$ be a nonamenable group.
\begin{enumerate}
\item There exists a finitely generated nonamenable subgroup $H\leq G$ with $\mathscr{C}^*(H) \leq 2$.
\item Let $M$ be a finite normal subgroup of $G$. Then there exists a finitely generated nonamenable subgroup $H\leq G$ containing $M$ with $\mathscr{C}^*(H)\leq 1+1/|M|$.
\end{enumerate}
\end{lemma}

\begin{proof}
(1): Let $F$ be a finite subset of $G$ which is minimal (under inclusion) with respect to the property that $\langle F\rangle$ is nonamenable. Take $H= \langle F\rangle$. Let $H\cc (X,\mu )$ be a free p.m.p.\ action of $H$. By minimality of $F$, for any $g\in F$ the group $K= \langle F\setminus \{ g \} \rangle$ is amenable, hence $\mathscr{C}_\mu (\mc{R}^H_X)\leq \mathscr{C}_\mu (\mc{R}^K_X)+\mathscr{C}_\mu (\mc{R}^{\langle g\rangle}_X)\leq 2$.

(2): By applying part (1) to $G/M$ we may find a finitely generated nonamenable subgroup $H\leq G$ containing $M$ such that $\mathscr{C}^*(H/M ) \leq 2$. Then by Lemma \ref{lem:fin} we have
\[
\mathscr{C}^*(H) \leq 1+ \frac{\mathscr{C}^*(H /M) - 1}{|M|} \leq 1+ \frac{1}{|M|} . \qedhere
\]
\end{proof}

\begin{proof}[Proof of Theorem \ref{thm:FP1}]
(1): If $H$ is amenable then $\mathscr{C}^*(H)=1$, so $\mathscr{PC}^*(H) =1$, hence $\mathscr{PC}^*(G)=1$ by Proposition \ref{prop:Furman2} and thus $\mathscr{C}^*(G)=1$, so we are done. Assume now that $H$ is nonamenable.
Fix an $H$-conjugation invariant mean $\bm{m}$ on $G$. If $L$ is any subgroup of $G$ such that $L\cap H$ is nonamenable, then Theorem \ref{thm:relmain} implies that $(L\cap H )\leq _{wq} HK_H\leq _{wq}G$, so $L\leq _{wq}G$, and hence $\mathscr{PC}^*(G)\leq \mathscr{PC}^*(L)$ by Proposition \ref{prop:Furman2}. This shows that
\begin{equation}\label{eqn:inf}
\mathscr{PC}^*(G) = \inf \{ \mathscr{PC}^*(L)\csuchthat L\leq G\text{ and }L\cap H\text{ is nonamenable}\} .
\end{equation}
Apply Theorem \ref{thm:pairsubgroup} and take $\mathscr{F}=\{ H\}$. If alternative (1) holds then the subgroup $L= (H\cap C_G(K))K$ has fixed price $1$, and $L\cap H$ is nonamenable, so $\mathscr{PC}^*(G)=1$ by \eqref{eqn:inf}, and hence $\mathscr{C}^*(G)=1$. If alternative (2) holds then by Lemma \ref{lem:cost2} we may find a sequence $(L_n)_{n\in \N}$ of nonamenable subgroups $L_n\leq (H\cap C_G(M_n))M_n$ with $\mathscr{PC}^*(L_n) \leq 1+1/|M_n| \longrightarrow 1$ as $n\ra \infty$. Since $L_n\cap H$ is nonamenable for all $n\in \N$ we conlude once again that $\mathscr{PC}^*(G)=1$ and hence $\mathscr{C}^*(G)=1$. Finally, suppose that alternative (3) holds and let $L=(H\cap C_G(K))K$. Then $\mathscr{C}(L)=1$ since $H\cap C_G(K)$ commutes with $K$ and both groups are infinite \cite{Ga00}. Since $L\leq _{wq}G$ it follows from Proposition \ref{prop:Furman2} that $\mathscr{PC}(G)=1$ and hence $\mathscr{C}(G)=1$.

(2): We may assume that $G$ is nonamenable, and by the proof of part (1) we may assume that alternative (3) of Theorem \ref{thm:subgroup} holds. Then using the sequence $(M_n)_{n\in \N}$ we obtain that $\mathscr{C}^*(G)=1$ as in the proof for alternative (2) above.
\end{proof}

\section{Cocycle superrigidity}\label{sec:superrigid}

\subsection{Amenable actions and stable spectral gap}

Fix a finite subset $S$ of $G$. Given a unitary representation $\pi$ of $G$ on the Hilbert space $\Es{H}$, we define
\begin{align*}
\phi _S (\pi ) = \inf \Big\{ \sum _{s\in S}\frac{\| \pi _s \xi - \xi \|}{\| \xi \| } \csuchthat 0\neq \xi \in \Es{H} \Big\} .
\end{align*}
The representation $\pi$ is said to have {\bf stable spectral gap} if $1\not\prec \pi\otimes \ol{\pi}$, where $\prec$ denotes weak containment of unitary representations. If $S$ generates $G$, then the representation $\pi$ has stable spectral gap if and only if $\phi _S (\pi \otimes \ol{\pi} ) >0$.

\begin{lemma}\label{lem:square}
Let $\kappa$ and $\pi$ be unitary representation of $G$.
\begin{enumerate}
\item If $\kappa \prec \pi$ then $\phi _S (\pi ) \leq \phi _S (\kappa )$.
\item $\phi _S (\pi \otimes \kappa )\geq \tfrac{1}{2}\phi _S (\pi \otimes \ol{\pi} )^2$.
\end{enumerate}
\end{lemma}

\begin{proof}
(1) is clear, and (2) follows from the Powers-St{\o}rmer inequality (see \cite[Lemma 3.2]{Pop08}).
\end{proof}

\begin{lemma}\label{lem:ssg}
Assume that $S$ generates $G$. Let $\pi$ be a unitary representation of $G$ with stable spectral gap. Let $H$ be a subgroup of $G$ and let $\lambda _{G/H}$ denotes the quasi-regular representation of $G$ on $\ell ^2(G/H )$. If $\phi _S (\lambda _{G/H})< \tfrac{1}{2}\phi _S (\pi \otimes \ol{\pi} )^2$ then $\pi _{|H}$ has stable spectral gap.
%

It follows that if $G\cc X$ is an amenable action of a countable group $G$ on a set $X$ with $G$-invariant mean $\bm{m}$, and if $\pi$ is a representation of $G$ with stable spectral gap, then $\pi _{|G_x}$ has stable spectral gap for $\bm{m}$-almost every $x\in X$.
\end{lemma}

\begin{proof}
We prove the contrapositive. Suppose that $\pi _{|H}$ does not have stable spectral gap, i.e., $1_{|H}\prec (\pi \otimes \ol{\pi} )_{|H}$. Then $\lambda _{G/H}\prec \mathrm{Ind}_H^G ((\pi \otimes \ol{\pi}) _{|H}) \cong \pi \otimes \ol{\pi} \otimes \lambda _{G/H}$. Applying (1) and then (2) of Lemma \ref{lem:square}, we obtain $\phi _S(\lambda _{G/H}) \geq \phi _S(\pi \otimes (\ol{\pi} \otimes \lambda _{G/H} ) ) \geq \tfrac{1}{2}\phi _S (\pi \otimes \ol{\pi} )^2$.

For the last statement, note that it is enough to prove this in the case where $G$ is finitely generated, say by $S$. Then for any $\epsilon >0$ we have $\bm{m} (\{ x\in X \csuchthat \phi _S (\lambda _{G/G_x}) <\epsilon \} ) =1$, so we can take $\epsilon = \tfrac{1}{2}\phi _S (\pi \otimes \ol{\pi} )^2$ and apply the first part of the lemma.
\end{proof}

\subsection{The space of means}\label{subsec:means}
Let $\mathscr{M}=\mathscr{M}(G)$ denote the space of means on $G$. Then $\mathscr{M}$ is a weak${}^*$-closed subset of the unit ball of $\ell ^\infty (G)^*$, hence by the Banach-Alaoglu Theorem $\mathscr{M}$ is compact in the weak${}^*$-topology. Let $\Es{P} = \mathscr{M}\cap \ell ^1(G)$ denote the collection of all probability vectors on $G$. Then $\Es{P}$ is a weak${}^*$-dense subset of $\mathscr{M}$.

For $g\in G$ and $\bm{m}\in \mathscr{M}$ we define the means $g\bm{m}$ and $\bm{m}g$ respectively by $(g\bm{m})(A) = \bm{m}(g^{-1}A)$ and $(\bm{m}g)(A) = \bm{m}(Ag^{-1})$ for $A\subseteq G$. The assignments $(g,\bm{m})\mapsto g\bm{m}$ and $(\bm{m},g)\mapsto \bm{m}g$ define left and right actions respectively of $G$ on $\mathscr{M}$ which commute. The {\bf convolution} of two means $\bm{m}$ and $\bm{n}$ on $G$ is defined to be the mean $\bm{m}\ast \bm{n} = \int _{g\in G} g\bm{n}\, d\bm{m}(g)$. This gives $\mathscr{M}$ the structure of a semigroup in which multiplication is weak${}^*$-continuous in the left variable. We identify $G$ with the collection of point masses $\{ \delta _g \} _{g\in G}\subseteq \Es{P}$. Then we have $g\bm{m}= \delta _g\ast \bm{m}$ and $\bm{m}g=\bm{m} \ast \delta _g$.

For a subgroup $H\leq G$, let $C_{\mathscr{M}}(H)= C_{\mathscr{M}(G)}(H)$ denote the subset of $\mathscr{M}$ consisting of all means on $G$ which are invariant under conjugation by $H$. Observe that $C_{\mathscr{M}}(H)$ is a closed, convex, subsemigroup of $\mathscr{M}$ with $C_G(H) \subseteq C_{\mathscr{M}}(H)$.

\begin{lemma}\label{lem:Mazur}
Let $H$ be a subgroup of $G$, let $\bm{m}\in C_{\mathscr{M}}(H)$, and let $\mathscr{A}\subseteq \ell ^\infty (G)$ be a separable subalgebra. Then there exists a sequence $(p_n)_{n\in \N}$ in $\Es{P}$ such that $\lim _n \| hp_nh^{-1} - p_n \| _1 = 0$ for all $h\in H$, and $\lim _n p_n (\phi ) = \bm{m}(\phi )$ for all $\phi \in \mathscr{A}$.
\end{lemma}

\begin{proof}
Fix finite sets $S\subseteq H$ and $\mathscr{A}_0\subseteq \mathscr{A}$, along with $\epsilon >0$. It suffices to show that there exists some $p\in \Es{P}$ with $\sup _{s\in S} \| sps^{-1} - p \| _1 <\epsilon$ and $\sup _{\phi \in \mathscr{A}_0} |p(\phi ) - \bm{m}(\phi ) | <\epsilon$. Since $\Es{P}$ is weak${}^*$-dense in $\mathscr{M}$, the convex set $K_0 = \{ p \in \Es{P} \csuchthat \sup _{\phi \in \mathscr{A}_0} | p (\phi ) - \bm{m}(\phi ) | <\epsilon \}$ contains $\bm{m}$ in its weak${}^*$-closure. Since $\bm{m}\in C_{\mathscr{M}}(H)$, the convex subset $\{ (sps^{-1} - p)_{s\in S}  \csuchthat p \in K_0 \}$ of $\ell ^1 (G)^S$ contains $0\in \ell ^1(G)^S$ in its weak closure, hence in its norm closure by Mazur's Theorem. This implies that there exists $p\in K_0$ such that $\sup _{s\in S} \| sps^{-1} - p \| _1 <\epsilon$.
\end{proof}

\subsection{Weak mixing for subsemigroups of $\mathscr{M}$}\label{subsec:WM} Let $\Es{H}$ be a Hilbert space and let $\varphi : G \ra \mathscr{B}(\Es{H})$, $g\mapsto \varphi _g$, be a map from $G$ into the bounded linear operators on $\Es{H}$ whose image is contained in the unit ball of $\mathscr{B}(\Es{H})$. We extend $\varphi$ to a map $\mathscr{M}\ra \mathscr{B}(\Es{H})$, by taking $\varphi _{\bm{m}} = \int _G \varphi _g \, d\bm{m}$, i.e., $\varphi _{\bm{m}}$ is the unique bounded linear operator satisfying $\langle \varphi _{\bm{m}}\xi ,\eta \rangle = \int _{g\in G}\langle \varphi _g\xi , \eta \rangle \, d\bm{m}(g)$ for all $\xi ,\eta \in \Es{H}$. In particular, each unitary representation $\pi : G\ra \mathscr{U}(\Es{H})$ of $G$ extends to a map $\mathscr{M}\ra \mathscr{B}(\Es{H})$ by taking $\pi _{\bm{m}} = \int _G \pi _g \, d\bm{m}$ for each $\bm{m}\in \mathscr{M}$.

\begin{proposition}\label{prop:contract} Let $\pi : G\ra \mathscr{U}(\Es{H})$ be a unitary representation of $G$. Then the extended map $\pi : \mathscr{M}\ra \mathscr{B}(\Es{H})$ has the following properties:
\begin{enumerate}
\item[i.] $\pi$ is an affine semigroup homomorphism.
\item[ii.] For each $\bm{m}\in \mathscr{M}$ we have $\pi _{\bm{m}}^* = \pi _{\check{\bm{m}}}$, where $\check{\bm{m}} (A) = \bm{m}(A^{-1})$ for $A\subseteq G$.
\item[iii.] For each $\bm{m}\in \mathscr{M}$ the operator $\pi _{\bm{m}}$ is a contraction, i.e., $\| \pi _{\bm{m}}\| _{\infty}\leq 1$.
\item[iv.] $\pi$ is continuous when $\mathscr{M}$ is given the weak${}^*$-topology and when $\mathscr{B}(\Es{H})$ is given the weak operator topology.
\item[v.] $\{ \pi _{\bm{m}} \} _{\bm{m}\in \mathscr{M}} \subseteq W^*(\pi (G))$.
\end{enumerate}
\end{proposition}

\begin{proof}
Properties i.\ through iv.\ follow from the definitions, and v.\ follows from iv.
\end{proof}

If $\pi$ and $\kappa$ are unitary representations of $G$ on $\Es{H}$ and $\Es{K}$ respectively, then for $\bm{m}\in \mathscr{M}$, the operators $(\pi \otimes \kappa )_{\bm{m}} = \int _G \pi _g\otimes \kappa _g \, d\bm{m}$ and $\pi _{\bm{m}}\otimes \kappa _{\bm{m}}$ are generally distinct. We will only make use of the operator $(\pi \otimes \kappa )_{\bm{m}}$.

\begin{definition}
Let $\pi$ be a unitary representation of $G$ and let $\mathscr{M}_0$ be a subsemigroup of $\mathscr{M}$. We say that $\pi _{| \mathscr{M}_0}$ is {\bf weakly mixing} if $(\pi \otimes \bar{\pi})_{|\mathscr{M}_0}$ has no nonzero invariant vectors.
\end{definition}

The next proposition extends several well-known characterizations of weak mixing from the group setting to the setting of subsemigroups of $\mathscr{M}$.

\begin{proposition}\label{prop:WM}
Let $\pi :G\ra \mathscr{U}(\Es{H})$ be a unitary representation of $G$ and let $\mathscr{M}_0$ be a subsemigroup of $\mathscr{M}$. Then the following are equivalent:
\begin{enumerate}
\item[i.] $\pi _{|\mathscr{M}_0}$ is weakly mixing.
\item[ii.] $(\pi \otimes \kappa )_{|\mathscr{M}_0}$ has no nonzero invariant vectors for every unitary representation $\kappa$ of $G$.
\item[iii.] For every finite $F\subseteq \Es{H}$ and $\epsilon >0$ there exists $\bm{m}\in \mathscr{M}_0$ such that
\[
\int _{g\in G}|\langle \pi _g\xi , \eta \rangle | ^2 d\bm{m}(g) <\epsilon
\]
for all $\xi ,\eta \in F$.
\item[iv.] There is no nonzero finite dimensional subspace $\Es{L}$ of $\Es{H}$ such that $P_{\Es{L}} = \int _G \pi _g P_{\Es{L}}\pi _g^* \, d\bm{m}$ for all $\bm{m}\in \mathscr{M}_0$. Here, $P_{\Es{L}}$ denotes the orthogonal projection onto $\Es{L}$.
\end{enumerate}
\end{proposition}

\begin{proof}
Using the properties in Proposition \ref{prop:contract}, the proof is a routine extension of the proof for the case $\mathscr{M}_0=G$ (see, e.g., \cite{Pe11}).
\end{proof}

\begin{example}
Let $\lambda :G \ra \mathscr{U}(\ell ^2(G))$ be the left regular representation of $G$. Then $\lambda$ is a mixing representation of $G$, so if $\bm{m}$ is any atomless mean on $G$ then $\lambda _{\bm{m}} = 0$ in $\mathscr{B}(\ell ^2(G))$. It follows that if $\mathscr{M}_0$ is a subsemigroup of $\mathscr{M}$ whose weak${}^*$-closure contains a mean which is atomless, then $\lambda _{|\mathscr{M}_0}$ is weakly mixing.
\end{example}

The next proposition shows that weak mixing for the Koopman representation associated to a p.m.p.\ action of $G$ behaves as expected.

\begin{proposition}\label{prop:ERG} Let $G\cc (X,\mu )$ be a p.m.p.\ action of $G$ and let $\kappa$ denote the associated Koopman representation on $L^2(X,\mu )$. Let $\mathscr{M}_0$ be a subsemigroup of $\mathscr{M}$. Then the collection $\{ A\subseteq X \csuchthat \kappa _{\bm{m}}(1_A)=1_A \text{ for all }\bm{m}\in \mathscr{M}_0 \}$ is a $\| \cdot \| _2$-norm closed sigma subalgebra of the measure algebra of $(X,\mu )$. Furthermore, a function $\xi \in L^2(X,\mu )$ is $\kappa _{|\mathscr{M}_0}$-invariant if and only if $1_A$ is $\kappa _{|\mathscr{M}_0}$-invariant for every $\xi$-measurable set $A\subseteq X$.
\end{proposition}

\begin{proof}
A function $\xi \in L^2(X,\mu )$ is $\kappa _{\bm{m}}$-invariant if and only if $\int _G \| \kappa _g (\xi ) - \xi \| _2 ^2 \, d\bm{m} = 0$. Therefore, if $f_0,f_1\in L^\infty (X,\mu )$ are both $\kappa _{\bm{m}}$-invariant, then
\[
\int _G\| \kappa _g(f_0f_1)- f_0f_1\| _2 ^2 \, d\bm{m} \leq \| f_1 \|_\infty ^2 \int _G \| \kappa _g (f_0) - f_0 \| _2^2 \, d\bm{m} + \| f_0 \| _{\infty} ^2 \int _G \| \kappa _g (f_1)- f_1 \| _2^2 \, d\bm{m} = 0,
\]
hence $f_0f_1$ is also $\kappa _{\bm{m}}$-invariant.

Assume that $\xi \in L^2(X,\mu )$ is $\kappa _{|\mathscr{M}_0}$-invariant. It suffices to show that sets of the form $A_r = \{ x\in X\csuchthat \xi (x)\geq r \}$, $r\in \R$, are $\kappa _{|\mathscr{M}_0}$-invariant. Suppose toward a contradiction that $A_r$ is not $\kappa _{\bm{m}}$-invariant for some $r\in \R$ and $\bm{m}\in \mathscr{M}_0$. Then we have $\int _G \mu (A_r \setminus gA_r ) \, d\bm{m} > 0$, so there is some $\epsilon >0$ such that $\bm{m}(D_{\epsilon})>0$ where $D_\epsilon = \{ g\in G\csuchthat \mu (A_r\setminus gA_r ) > \epsilon \}$. Find $\delta >0$ such that $\mu (A_{r-\delta }\setminus A_r ) = \mu ( \{ x\in X\csuchthat r > \xi (x) \geq r-\delta \} ) < \epsilon /2$. Then for $g\in D_\epsilon$ we have $\mu (A_r \setminus gA_{r-\delta} ) > \epsilon /2$ and hence $\| \xi - \kappa _g(\xi ) \| _2^2 \geq \delta \epsilon /2$. Therefore, $0 = \int _{g\in D_{\epsilon}} \| \xi - \kappa _g(\xi ) \| _2 ^2 \, d\bm{m} \geq \bm{m}(D_\epsilon )\delta \epsilon /2 > 0$, a contradiction. For the reverse implication, approximate $\xi$ in $\| \cdot \| _2$-norm by $\xi$-measurable simple functions.
\end{proof}

\begin{definition}\label{def:WM}
Let $G\cc ^\sigma (X,\mu )$ be a p.m.p.\ action of $G$ and let $\kappa$ denote the associated Koopman representation on $L^2(X,\mu )$. Let $\mathscr{M}_0$ be a subsemigroup of $\mathscr{M}$. We say that $\sigma _{|\mathscr{M}_0}$ is {\bf ergodic} if every $\kappa _{|\mathscr{M}_0}$-invariant function in $L^2(X,\mu )$ is essentially constant. We say that $\sigma _{|\mathscr{M}_0}$ is {\bf weakly mixing} if $(\sigma \otimes \sigma )_{|\mathscr{M}_0}$ is ergodic.
\end{definition}

\begin{proposition}\label{prop:WMaction}
Let $G\cc ^\sigma (X,\mu )$ be a p.m.p.\ action of $G$ and let $\kappa$ denote the associated Koopman representation on $L^2(X,\mu )$. Let $\mathscr{M}_0$ be a subsemigroup of $\mathscr{M}$. Then the following are equivalent:
\begin{enumerate}
\item[i.] $\sigma _{|\mathscr{M}_0}$ is weakly mixing;
\item[ii.] $(\sigma \otimes \rho )_{|\mathscr{M}_0}$ is ergodic for every ergodic p.m.p.\ action $\rho$ of $G$;
\item[iii.] The restriction of $\kappa _{|\mathscr{M}_0}$ to $L^2(X,\mu ) \ominus \C 1_X$ is weakly mixing.
\end{enumerate}
\end{proposition}

\begin{proof}
This follows from Propositions \ref{prop:WM} and \ref{prop:ERG}.
\end{proof}

\subsection{Proof of Theorem \ref{thm:superrigid1}}

\begin{proof}[Proof of Theorem \ref{thm:superrigid1}]
For $g\in G$ let $C_H(g) = H\cap C_G(g)$. Define the set
\[
D = \{ g\in G \csuchthat \sigma _{|C_H(g)} \text{ has stable spectral gap} \} ,
\]
and let $G_0 = \langle H , D \rangle$. If $\bm{m}$ is a mean on $G$ which is invariant under conjugation by $H$ then, since $\sigma _{|H}$ has stable spectral gap, Lemma \ref{lem:ssg} implies that $\bm{m}(D) = 1$, so in particular $\bm{m}(G_0)=1$.

Let $w:G\times X \ra L$ be a cocycle with values in a group $L\in \mathscr{U}_{\mathrm{fin}}$. To show that $w$ untwists on $G_0$, we claim that it is enough to show that $w$ untwists on $H$. Indeed, assume that $w$ untwists on $H$. If $g\in D$ then $C_H(g)\leq gHg^{-1}\cap H$ and $\sigma _{|C_H(g)}$ has stable spectral gap, so in particular $\sigma _{|C_H(g)}$ is weakly mixing. We can therefore apply the last statement in Lemma 3.5 of \cite{Fu07} to conclude that $w$ untwists on all of $G_0$.

It remains to show that $w$ untwists on $H$. We may assume that $L$ is a closed subgroup of the unitary group $\mathscr{U}(N)$ of some finite von Neumann algebra $N$. Let $A= L^\infty (X,\mu )$ and view $w$ as a cocycle $w:G\ra \mathscr{U}(A\otimes N )$ for the action $\sigma \otimes \mathrm{id}_N$, i.e., satisfying $w_{gh}=w_g(\sigma _g \otimes \mathrm{id}_N)(w_h )$. We will use Popa's setup from Theorem 4.1 of \cite{Pop08}, with $A$ here taking the place of $P$. Namely, we let $M=(A\otimes N)\rtimes _{\sigma \otimes \mathrm{id}_N}G$ 
and we let $\wt{M} = (A\otimes A \otimes N)\rtimes _{\sigma \otimes \sigma \otimes \mathrm{id}_N}G$, and we view $M$ as a subalgebra of $\wt{M}$ so that the canonical unitaries $\{ u_g \} _{g\in G}\subset M$ implement $\sigma \otimes \mathrm{id}_N$ and $\sigma \otimes \sigma \otimes \mathrm{id}_N$ on $M$ and $\wt{M}$ respectively. We let $\tau$ denote the trance on $\wt{M}$. 
Let $\{ \alpha _t \} _{t\in \R} \cup \{ \beta \} \subseteq \mathrm{Aut}(A\otimes A)$ denote the $s$-malleable deformation, and we extend $\beta$ and $\alpha _t$, $t\in \R$, to automorphisms of $\wt{M}$ by letting $\beta (x) = x = \alpha _t (x)$ if $x\in N\otimes \mathrm{L}G$. Let $\wt{u}_g = w_gu_g$ for $g\in G$, so that $g\mapsto \wt{u}_g$ is a homomorphism. Let $\wt{\pi}$ denote the representation of $G$ on $L^2(\wt{M}) =L^2(M)\otimes L^2(A)$ determined by $\wt{\pi} (g) ((xu_h )\otimes y ) = \mathrm{Ad}(\wt{u}_g)(xu_h)\otimes \sigma _g (y)$, for $x\in A\otimes N$, $y\in A$, $g,h\in G$.

\begin{claim}\label{claim:pv}
$\lim _{t\ra 0}\big( \sup _{\bm{m}\in C_{\mathscr{M}}(H)}\int _{g\in G} \| \alpha _t (\wt{u}_g)- \wt{u}_g \| _2 ^2 \, d\bm{m}(g) \big) = 0$.
\end{claim}

\begin{proof}[Proof of Claim \ref{claim:pv}]
Fix $\epsilon >0$. It suffices to show that there exists $t_\epsilon >0$, along with $S\subseteq H$ finite and $\delta >0$, such that if $p\in \Es{P}$ satisfies $\sup _{s\in S} \| sps^{-1} - p \| _1 < \delta$ then for all $t$ with $0\leq |t|\leq t_{\epsilon}$ we have $\int _G \| \alpha _t (\wt{u}_g)- \wt{u}_g \| _2 ^2 \, dp (g) < \epsilon$, since the claim will then follow using Lemma \ref{lem:Mazur}. Since $H\cc ^{\sigma }(X,\mu )$ has stable spectral gap, there exists a finite set $S\subseteq H$ and $\delta _0 >0$ such that if $\eta \in L^2(\wt{M})$ is a unit vector satisfying $\sup _{s\in S} \| \wt{\pi} (s)\eta - \eta \| _2 < \delta _0$, then $\| \eta - e(\eta ) \| _2 < \epsilon ^{1/2}/2$, where $e: L^2(\wt{M})\ra L^2(M)$ denotes the orthogonal projection. Since $S$ is finite, there exists $t_1>0$ such that for all $0\leq t_0\leq t_1$ we have $\sup _{s\in S} \| \alpha _{t_0} (\wt{u}_s)-\wt{u}_s \| _2 < \delta _0 /4$. Let $t_\epsilon = 2t_1$ and fix $t_0$ with $0\leq t_0\leq t_1$. Let $\delta = \delta _0 ^2/4$ and fix $p\in \Es{P}$ with $\sup _{s\in S} \| sps^{-1}-p \| _1 < \delta$.

Let $Q= \{ \wt{u}_g \} _{g\in G}''$. Then we may identify $\ell ^2(G)$ with $L^2(Q) \subset L^2(M)$ via $\delta _g \mapsto \wt{u}_g$. For each $q\in \Es{P}$, let $\eta _q = \sum _{g\in G} q(g)^{1/2}\wt{u}_g \in L^2(Q)$. We have $\tau (\alpha _{t_0}(\wt{u}_g)\wt{u}_h^*) = 0$ for all $g\neq h$, which implies that $\| \alpha _{t_0} (\wt{u}_g)\xi - \wt{u}_g\xi \| _2 = \| \alpha _{t_0} (\wt{u}_g) -\wt{u}_g \| _2 = \| \xi \alpha _{t_0}(\wt{u}_g) - \xi \wt{u}_g \| _2$ for every unit vector $\xi \in L^2(Q)$, and therefore $\sup _{s\in S} \| \alpha _{t_0} (\wt{u}_s) \eta _p - \eta _p \alpha _{t_0} (\wt{u}_s) \| _2 <  \delta _0/2 + \sup _{s\in S}\| \eta _{sps^{-1}} - \eta _p \| _2 < \delta _0$. By replacing $t_0$ by $-t_0$ and applying $\alpha _{t_0}$ we obtain $\sup _{s\in S} \| \wt{\pi}(s) \alpha _{t_0}(\eta _p) - \alpha _{t_0}(\eta _p) \| _2< \delta _0$. Our choice of $\delta _0$ then implies that $\| \alpha _{t_0}(\eta _p)- e(\alpha _{t_0}(\eta _p )) \| _2 < \epsilon ^{1/2} / 2$, and hence by Popa's Transversality Lemma \cite[Lemma 2.1]{Pop08},
\begin{align*}
\int _{G}\| \alpha _{2t_0}(\wt{u}_g) - \wt{u}_g \| _2 ^2 \, dp  &\leq 4\int _{G} \| \alpha _{t_0} (\wt{u}_g) - e(\alpha _{t_0}(\wt{u}_g)) \| _2 ^2 \, dp = 4 \| \alpha _{t_0}(\eta _p)- e(\alpha _{t_0}(\eta _p)) \| _2^2 < \epsilon .
\end{align*}
\qedhere[Claim \ref{claim:pv}]
\end{proof}

\begin{claim}\label{claim:2}
$\lim _{t\ra 0} \big( \sup _{h\in H}\| \alpha _t (\wt{u}_h)- \wt{u}_h \| _2^2 \big) =0$.
\end{claim}

\begin{proof}[Proof of Claim \ref{claim:2}]
Fix $\epsilon >0$. By Claim \ref{claim:pv} there exists $t_\epsilon >0$ such that $\int _G \| \alpha _t (\wt{u}_g)- \wt{u}_g \| _2 ^2 \, d\bm{m} <\epsilon /8$ for all $\bm{m}\in C_{\mathscr{M}}(H)$ and all $t$ with $0\leq |t|\leq t_\epsilon$. Fix any $h\in H$ along with $0\leq |t|\leq t_\epsilon$, and we will show that $\| \alpha _t (\wt{u}_h)-\wt{u}_h \| _2^2 < \epsilon$. For $\bm{m}\in C_{\mathscr{M}}(H)$ we have
\begin{align*}
\int _{G} \| \alpha _t &(\wt{u}_{hgh^{-1}})\wt{u}_h - \wt{u}_h\alpha _t (\wt{u}_g) \| _2^2 \, d\bm{m} \\
&\leq \Big[ \Big( \int _{G}\| \alpha _t (\wt{u}_{hgh^{-1}})- \wt{u}_{hgh^{-1}}\| _2^2 \, d\bm{m} \Big) ^{1/2} + \Big( \int _G \| \alpha _t (\wt{u}_g )- \wt{u}_g \| _2^2 \, d\bm{m} \Big) ^{1/2}\Big]^2 < \epsilon /2 .
\end{align*}
By replacing $t$ with $-t$ and applying $\alpha _t$ we obtain $\int _G \| \wt{u}_{hgh^{-1}}\alpha _t (\wt{u}_h)-\alpha _t(\wt{u}_h)\wt{u}_g \| _2^2 \, d\bm{m} <\epsilon /2$. Let $\eta _h  = \alpha _t (\wt{u}_h)- e(\alpha _t(\wt{u}_h))$. Then by projecting onto $L^2(\wt{M})\ominus L^2(M)$ and using the last inequality we obtain
\begin{equation}\label{eqn:etah}
\int _G \| \wt{u}_{hgh^{-1}}\eta _h -\eta _h \wt{u}_g \| _2^2 \,d\bm{m} <\epsilon /2 , \ \ (\bm{m} \in C_{\mathscr{M}}(H) ) .
\end{equation}
{\bf Subclaim:} $\inf _{C_{\mathscr{M}}(H)} \sum _{i,j <n} \int _G |\langle \wt{u}_{hgh^{-1}}\eta _i , \eta _j \wt{u}_g \rangle | \, d\bm{m} = 0$ for any $\eta _0,\dots ,\eta _{n-1} \in L^2(\wt{M})\ominus L^2(M)$.

\begin{proof}[Proof of Subclaim] We may assume each $\eta _i$ is of the form $\eta _i = (a_i\otimes b_i \otimes c_i )u_{k_i}$, where $a_i \in A\otimes 1$, $b_i \in 1\otimes A$, $\tau (b_i ) = 0$, $c_i \in N$, $k_i\in G$, since such vectors span a dense subspace of $L^2(\wt{M})\ominus L^2(M)$. Since $w_l \in A \otimes 1 \otimes N$, we have
\begin{align*}
|\langle \wt{u}_{hgh^{-1}}\eta _i , \eta _j \wt{u}_g \rangle | &= |\langle w_{hgh^{-1}}u_{hgh^{-1}}(a_i\otimes b_i\otimes c_i)u_{k_i}, (a_j\otimes b_j\otimes c_j)u_{k_j}w_gu_g\rangle | \\
&\leq |\langle w_{hgh^{-1}}(\sigma _{hgh^{-1}}(a_i)\otimes c_i), (a_j\otimes c_j)(\sigma _{k_j}\otimes \mathrm{id}_N)(w_g)\rangle | \, | \langle \sigma _{hgh^{-1}}(b_i), b_j \rangle |  \\
&\leq \| a_i \otimes c_i \| _2 \| a_j \otimes c_j \| _2 | \langle \sigma _{hgh^{-1}}(b_i), b_j \rangle | .
\end{align*}
Therefore, letting $C=\max _{i<n}\| a_i\otimes c_i \| _2$, for $\bm{m}\in C_{\mathscr{M}}(H)$ we have
\begin{equation}\label{eqn:C}
\sum _{i,j<n} \int _G |\langle \wt{u}_{hgh^{-1}}\eta _i , \eta _j \wt{u}_g \rangle | \, d\bm{m}(g) \leq C^2 \sum _{i,j<n}\int _G |\langle \sigma _{g}(b_i), b_j \rangle | \, d\bm{m} (g) .
\end{equation}
By Proposition \ref{prop:WM}, since $C_{\mathscr{M}}(H)\cc ^{\sigma}A$ is weakly mixing, the infimum over $\bm{m} \in C_{\mathscr{M}}(H)$ of the right hand side of \eqref{eqn:C} is $0$.\qedhere[Subclaim]
\end{proof}

By \eqref{eqn:etah} and the subclaim, we have
\[
\epsilon /2 \geq  \sup _{C_{\mathscr{M}}(H)} \int _G \| \wt{u}_{hgh^{-1}}\eta _h -\eta _h \wt{u}_g \| _2^2 \, d\bm{m} \geq 2\| \eta _h \| _2 ^2  - 2\inf _{C_{\mathscr{M}}(H)} \int |\langle \wt{u}_{hgh^{-1}}\eta _h , \eta _h \wt{u}_g \rangle |\, d\bm{m} = 2 \| \eta _h \| _2 ^2 .
\]
Popa's Transversality Lemma now shows $\| \alpha _t (\wt{u}_h) - \wt{u}_h \| _2 ^2 \leq 4\| \eta  _h \| _2 ^2 \leq \epsilon$. \qedhere[Claim \ref{claim:2}]
\end{proof}

By \cite{Pop07}, Claim \ref{claim:2} implies that $w$ untwists on $H$.
\end{proof}

\section{The AC-center, the inner radical, and linear groups}\label{sec:radicals}

\subsection{Proof of Theorem \ref{thm:I(G)}, parts i.\ through viii.}

\begin{proof}[Proof of Theorem \ref{thm:I(G)}, parts i.\ through viii.] We begin with the statements involving $\mathscr{AC}(G)$. It is clear that $\mathscr{AC}(G)$ and $\mathscr{I}(G)$ are characteristic subgroups of $G$. If $N_0$ and $N_1$ are normal subgroups of $G$ with both $G/C_G(N_0)$ and $G/C_G(N_1)$ amenable, then $G/C_G(N_0N_1) = G/(C_G(N_0)\cap C_G(N_1))$ is amenable. Therefore, $\mathscr{AC}(G)$ may be written as an increasing union $\mathscr{AC}(G)=\bigcup _{i\in \N} N_i$, with each $N_i$ normal in $G$ and $G/C_G(N_i)$ amenable. It follows that $G/C_G(\mathscr{AC}(G))$ is residually amenable since $C_G(\mathscr{AC}(G))= \bigcap _{i\in \N}C_G(N_i)$. Each of the groups $N_i$ is amenable since $N_i/Z(N_i)$ is isomorphic to a subgroup of the amenable group $G/C_G(N_i)$. This shows that $\mathscr{AC}(G)$ is amenable. Moreover, for each $i\in \N$, the action $N_i\rtimes G\cc N_i$ is amenable since it descends to an action of the amenable group $N_i\rtimes (G/C_G(N_i))$. If $\bm{m}_i$ is a $N_i\rtimes G$-invariant mean on $N_i$, then any accumulation point of $(\bm{m}_i)_{i\in \N}$ in the space of means on $G$ will be a mean witnessing that the action $\mathscr{AC}(G)\rtimes G \cc \mathscr{AC}(G)$ is amenable. It follows that $\mathscr{AC}(G) \leq \mathscr{I}(G)$.

To prove the remaining statements involving $\mathscr{I}(G)$ we will use the following lemma.

\begin{lemma}\label{lem:extendI(G)} Let $H$ and $K$ be normal subgroups of $G$.
\begin{enumerate}
\item Assume that $H\leq K$ and that the actions $H\rtimes G\cc H$ and $K/H \rtimes G/H \cc K/H$ are both amenable, with invariant means $\bm{m}_H$ and $\bm{m}_{K/H}$ respectively. Then the action $K\rtimes G \cc K$ is amenable with invariant mean $\bm{m}_K= \int _{kH\in K/H}k\bm{m}_H\, d\bm{m}_{K/H}$.
\item Assume that the actions $H\rtimes G\cc H$ and $K\rtimes G\cc K$ are both amenable, with invariant means $\bm{m}$ and $\bm{n}$ respectively. Then the action $HK\rtimes G\cc HK$ is amenable with invariant mean $\bm{m}\ast \bm{n}$.
\end{enumerate}
\end{lemma}

\begin{proof}[Proof of Lemma \ref{lem:extendI(G)}] Since $\bm{m}_H$ is invariant under left translation by $H$, for each $g\in G$ the mean $g\bm{m}_H$ only depends on the coset $gH\in G/H$. The mean $\bm{m}_K$ is therefore well-defined and it is straightforward to verify that it is $K\rtimes G$-invariant. This shows (1), and (2) can either be deduced from (1) or verified directly. \qedhere[Lemma \ref{lem:extendI(G)}]
\end{proof}
It follows from Lemma \ref{lem:extendI(G)}.(2) that $\mathscr{I}(G)$ may be written as an increasing union $\mathscr{I}(G)=\bigcup _{i\in \N}M_i$, where each $M_i$ is normal in $G$ and the action $M_i\rtimes G \cc M_i$ is amenable. If $\bm{m}_i$ is an invariant mean for the action $M_i\rtimes G\cc M_i$, then any accumulation point $\bm{m}$ of $(\bm{m}_i)_{i\in \N}$ will be an invariant mean for the action $\mathscr{I}(G)\rtimes G \cc \mathscr{I}(G)$. In particular, $\bm{m}$ witnesses that $\mathscr{I}(G)$ is amenable, and if $\mathscr{I}(G)$ is infinite then $\bm{m}$ also witnesses that $G$ is inner amenable. The proof of i.\ through v.\ is now complete.

vi. Let $\pi : G\ra G/N$ denote the projection map. Then the image under $\pi$ of an $\mathscr{I}(G)\rtimes G$-invariant mean on $\mathscr{I}(G)$ is an $\mathscr{I}(G)/N \rtimes G/N$-invariant mean on $\mathscr{I}(G)/N$. This shows that $\mathscr{I}(G)/N \leq \mathscr{I}(G/N )$. The reverse containment then follows by applying part (1) of Lemma \ref{lem:extendI(G)} to the groups $H=\mathscr{I}(G)$ and $K= \pi ^{-1}(\mathscr{I}(G/N ) )$.

vii. Part vi.\ implies that $\mathscr{I}(G/\mathscr{I}(G))=1$, and this in turn implies that $G/\mathscr{I}(G)$ is ICC since every finite conjugacy class in $G/\mathscr{I}(G)$ is contained in $\mathscr{AC}(G/\mathscr{I}(G))\leq\mathscr{I}(G/\mathscr{I}(G))=1$.

viii. This in fact holds more generally with $\mathscr{I}(G)$ replaced by any normal subgroup $N$ of $G$ for which $N\rtimes G\cc N$ is amenable. To see this, fix an invariant mean $\bm{n}$ for the action $N\rtimes G\cc N$. As in Lemma \ref{lem:extendI(G)}, we obtain a well-defined map
\[
\bm{m}\mapsto \bm{m}\ast\bm{n} = \int _{gN\in G/N}g\bm{n}\, d\bm{m}(gN)
\]
taking means on $G/N$ to means on $G$. This map is a section for the projection map on means, and since $\bm{n}$ is invariant under conjugation by $G$, this map takes conjugation invariant means on $G/N$ to conjugation invariant means on $G$.
\end{proof}

\subsection{Proof of Theorem \ref{thm:I(G)}, parts ix.\ through xiv.}

The second half of Theorem \ref{thm:I(G)} will be deduced from the following spectacular theorem of S.G.\ Dani from \cite{Da85}, which appears to have been overlooked since its publication in 1985. In what follows, if $G\cc X$ is an action of a group $G$ then for $A\subseteq X$ let $\mathrm{stab}_G(A)$ denote the pointwise stabilizer of $A$ in $G$, and for $D\subseteq G$ let $\mathrm{fix}_X(D)$ denote the set of points in $X$ which are fixed by every element of $D$.

\begin{theorem}[Theorem 1.1 of \cite{Da85}]\label{thm:Dani}
Let $G\cc X$ be an amenable action of a group $G$ on a set $X$ and let $\bm{m}$ be a $G$-invariant mean on $X$. Suppose that the action satisfies the following two conditions:
\begin{enumerate}
\item[(1)] For every subset $A\subseteq X$ there exists a finite $A_0\subseteq A$ such that $\mathrm{stab}_G(A)=\mathrm{stab}_G(A_0)$.
\item[(2)] For every subset $D\subseteq G$ there exists a finite $D_0\subseteq D$ such that $\mathrm{fix}_X(D)=\mathrm{fix}_X(D_0)$.
\end{enumerate}
Then there exists a normal subgroup $N$ of $G$ such that $G/N$ is amenable and $\bm{m}(\mathrm{fix}_X(N)) =1$.
\end{theorem}

\begin{proof}[Proof of Theorem \ref{thm:I(G)}, parts ix.\ through xiv.]
Assume that $G$ is linear. We first show that any conjugation invariant mean $\bm{m}$ on $G$ must concentrate on a normal subgroup $M$ of $G$ such that $G/C_G(M)$ is amenable. Consider the conjugation action $G\cc G$. For $A, D\subseteq G$ we have $\mathrm{stab}_G(A) = C_G(A)$, and $\mathrm{fix}_G(D) = C_G(D)$. Conditions (1) and (2) of Theorem \ref{thm:Dani} are therefore satisfied since $G$ satisfies the minimal condition on centralizers (see Remark \ref{rem:minC}). We conclude that there exists a normal subgroup $N$ of $G$ such that $G/N$ is amenable, and $\bm{m}(C_G(N))=1$. Take $M=C_G(N)$. Then $N\leq C_G(M)$, so $G/C_G(M)$ is amenable, as was to be shown. This also shows that xiii.\ holds.

ix.\ and x. By part iii., we may find an invariant mean $\bm{m}$ for the action $\mathscr{I}(G)\rtimes G\cc \mathscr{I}(G)$. Since $\bm{m}$ is conjugation invariant there exists a normal subgroup $M$ of $G$ with $G/C_G(M)$ amenable and $\bm{m}(M)=1$. Then $M\leq \mathscr{AC}(G)\leq \mathscr{I}(G)$ and, since $\bm{m}$ is invariant under left translation by $\mathscr{I}(G)$, we have equality $M=\mathscr{I}(G)$.

xi. Parts x.\ and ii.\ show that $\mathscr{I}(G)\leq C_G(C_G(\mathscr{I}(G)))\leq \mathscr{AC}(G)\leq \mathscr{I}(G)$.

xii. This follows from viii.\ and xiii.

xiv. From vi., we have $C_{G/N}(\mathscr{I}(G/N )) = C_{G/N}(\mathscr{I}(G)/N ) \geq C_G(\mathscr{I}(G))N /N$, hence $(G/N )/C_{G/N}(\mathscr{I}(G/N) )$ is amenable, and $\mathscr{I}(G/N)\leq \mathscr{AC}(G/N)$. Part ii.\ gives the reverse inclusion. It follows as in xi.\ above that $\mathscr{I}(G/N )$ coincides with its double centralizer. The group $(G/N)/\mathscr{I}(G/N) =(G/N)/(\mathscr{I}(G)/N) \cong G/\mathscr{I}(G)$ is not inner amenable by xii. Now let $\bm{m}_0$ be a conjugation invariant mean on $G/N$, and let $\bm{m}_1$ denote the projection of $\bm{m}_0$ to $G/\mathscr{I}(G)$. Then $\bm{m}_1$ is a conjugation invariant mean on $G/\mathscr{I}(G)$, so by viii.,\ $\bm{m}_1$ is the projection of some conjugation invariant mean $\bm{m}$ on $G$. By xiii.,\ $\bm{m}$ concentrates on $\mathscr{I}(G)$, hence $\bm{m}_1$ is the point mass at the identity, and therefore $\bm{m}_0$ concentrates on $\mathscr{I}(G)/N = \mathscr{I}(G/N)$.
\end{proof}

\subsection{Proof of Theorems \ref{thm:linear} and \ref{thm:Schmidt}}\label{sec:linear}

\begin{proof}[Proof of Theorem \ref{thm:linear}] The implication (1)$\Ra$(2) follows from Theorem \ref{thm:I(G)}.xiii., and (2)$\Ra$(1) is Theorem \ref{thm:I(G)}.v. Assume now that (2) holds and let $N=C_G(\mathscr{I}(G))\mathscr{I}(G)$. Then $G/N$ is amenable by Theorem \ref{thm:I(G)}.x., and $Z(N)=C_G(\mathscr{I}(G))\cap \mathscr{I}(G)$ by Theorem \ref{thm:I(G)}.xi., so (3) follows. If (3) holds then in the first alternative $Z(N)$ infinite and $Z(N)\leq \mathscr{AC}(G)\leq \mathscr{I}(G)$, and in the second alternative $M$ is infinite and $M\leq \mathscr{AC}(G)\leq \mathscr{I}(G)$, so (2) holds either way.
\end{proof}

\begin{proof}[Proof of Theorem \ref{thm:Schmidt}]
Assume that $G$ is inner amenable and we will construct the desired action of $G$. This is straightforward if $G$ is stable, so we may assume that $G$ is inner amenable, but not stable. Then the group $N=C_G(\mathscr{I}(G))$ has infinite center $C$ (see Remark \ref{rem:final}), and by Theorem \ref{thm:I(G)} the group $K= G/N$ is amenable. Since $C$ is a countable abelian group, it possesses a free p.m.p.\ action $C\cc (Y,\nu )$ which is compact (for example, using a countable dense subset of $\widehat{C}$, inject $C$ as a subgroup of $\T ^\N$ and let $C$ act by translation on $\T ^\N$ equipped with Haar measure). Let $G\cc (X,\mu ) = (Y, \nu ) ^{G/C}$ be the coinduced action. This is a free weakly mixing action of $G$, and the restriction of this action to $C$ is an infinite diagonal product of compact actions of $C$, hence is itself a compact action. It follows that there exists a sequence $(c_n)_{n\in \N}$ in $C - 1$ which converges to the identity automorphism in the group $\mathrm{Aut}(X,\mu )$ equipped with the weak topology. The sequence $(c_n)_{n\in \N}$ is then asymptotically central in $[\mc{R}^N_X]$, and since $C$ acts freely, the sequence $(c_n)_{n\in \N}$ witnesses that the outer automorphism group of $[\mc{R}^N_X]$ is not Polish. Let $K\cc (Z,\eta )$ be a free ergodic action of $K$, and let $G\cc (X,\mu ) \otimes (Z ,\eta )$ be the diagonal product action where $G$ acts on $(Z,\eta )$ via the quotient map to $K$. This action of $G$ is free and ergodic, and as observed in Remark \ref{rem:overview} below, the construction in the proof of Theorem \ref{thm:extension} below yields a sequence $(T_n)_{n\in \N}$ witnessing that the outer automorphism group of $[\mc{R}^G_{X\times Z}]$ is not Polish.

For the converse, which holds even without the assumption that $G$ is linear, see \cite{JS87}.
\end{proof}

\section{Stability}\label{sec:stability}

\subsection{Kida's stability criterion}\label{sec:Kida}

In this section we employ the notation from \S\ref{sec:groupoid}. Let $(\mc{G},\mu )$ be a discrete p.m.p.\ groupoid and let $[\mc{G}]$ denote the {\bf full group} of $\mc{G}$, i.e., the collection of all local sections with domain equal to all of $\mc{G}^0$.

\begin{definition}\label{def:stabseq}
A sequence $(T_n)_{n\in \N}$ in $[\mc{G}]$ is said to be {\bf asymptotically central} if
\begin{enumerate}
\item[(i)] $\mu ( T_n ^0 A\triangle A) \ra 0$ for all measurable $A\subseteq \mc{G}^0$;
\item[(ii)] $\mu ( \{ x\in \mc{G}^0 \csuchthat (T_n\circ S)x = (S\circ T_n)x \} ) \ra 1$ for all $S\in [\mc{G}]$.
\end{enumerate}
A sequence $(T_n)_{n\in \N}$ in $[\mc{G}]$ is called a {\bf stability sequence} for $\mc{G}$ if it is asymptotically central and if furthermore there exists a sequence $(A_n)_{n\in \N}$ of measurable subsets of $\mc{G}^0$ such that
\begin{enumerate}
\item[(iii)] $(A_n)_{n\in \N}$ is asymptotically invariant for $\mc{G}$, i.e., $\mu (S^0A_n \triangle A_n ) \ra 0$ for all $S\in [\mc{G}]$;
\item[(iv)] $\mu (T_n^0A_n \triangle A_n ) \not\ra 0$.
\end{enumerate}
\end{definition}

\begin{remark}\label{rem:checkG}
Suppose that $\mc{G}=G\ltimes (X,\mu )$ is the translation groupoid associated to a p.m.p.\ action $G\cc (X,\mu )$ of a countable group $G$. We view each $T\in [\mc{G}]$ as a map from $X$ to $G$, so that $T^0 (x) = T(x)\cdot x$ for $x\in X$. We will make use of the observation \cite[\S 3.1]{Ki13a} that in this situation, a sequence $(T_n)_{n\in \N}$ in $[\mc{G}]$ is asymptotically central if and only if it satisfies (i) along with
\begin{enumerate}
\item[(ii${}'$)] $\mu ( \{ x\in X \csuchthat T_n(g\cdot x) = g T_n(x)g^{-1} \} ) \ra 1$ for all $g\in G$.
\end{enumerate}
Likewise, a sequence $(A_n)_{n\in \N}$ of measurable subsets of $X$ is asymptotically invariant for $\mc{G}$ if and only if $\mu (g\cdot A_n \triangle A_n ) \ra 0$ for all $g\in G$.
\end{remark}

The following theorem, due to Kida \cite{Ki13a}, provides an important criterion for demonstrating stability of a group.

\begin{theorem}[Theorem 1.4 of \cite{Ki13a}]\label{thm:criterion}
Let $G$ be a countable group and suppose that there exists a p.m.p.\ action $G\cc (X,\mu )$ of $G$ whose associated translation groupoid $G\ltimes (X,\mu )$ admits a stability sequence. Then $G$ is stable.
\end{theorem}

\subsection{Proof of Theorem \ref{thm:extension}}

\begin{proof}[Proof of Theorem \ref{thm:extension}]
Let $G\cc (X,\mu )$ be a p.m.p.\ action of $G$ such that $N\ltimes (X,\mu )$ admits a stability sequence. Let $K\cc (Z,\eta )$ be a free p.m.p.\ action of $K$ and let $G\cc (X,\mu ) \otimes (Z,\eta )$ be the diagonal product action, where $G$ acts on the second coordinate via the quotient map to $K$. In what follows we will often identify an element of $G$ with its image in $K$.

By Theorem \ref{thm:criterion}, it suffices to show that the translation groupoid $G\ltimes (X,\mu )\otimes (Z,\eta )$ admits a stability sequence. We will construct such a sequence $(T_n)_{n\in \N}$ which is moreover contained in $[N\ltimes (X,\mu )\otimes (Z,\eta )]$. Let $F_0\subseteq F_1\subseteq \cdots$ be an exhaustion of $G$ by finite subsets. By Theorem 3.1 of \cite{BT-D11}, since $K$ is amenable, for each $n\geq 0$ we may find a measurable function $\varphi _n : Z\ra K$ such that $\eta (C_n)>1-2^{-n}$, where $C_n = \{ z\in Z\csuchthat (\forall g\in F_n ) \ \varphi _n(g\cdot z) = g\varphi _n(z) \}$. Let $1_K\in Q_0\subseteq Q_1\subseteq \cdots$ be an exhaustion of $K$ by finite subsets such that for each $n\geq 0$ we have $\eta (D_n)>1-2^{-n}$, where $D_n = \{ z\in Z\csuchthat \varphi _n (z)\in Q_n \}$. Let $Z_n = \bigcap _{m\geq n}C_m\cap D_m$, so that $Z_0\subseteq Z_1\subseteq \cdots$, and $\eta (\bigcup _n Z_n) = 1$. After ignoring a null set we may assume that $\bigcup _n Z_n = Z$. Fix a section $\sigma : K \ra G$ for the map $G\ra K$ with $\sigma (1_K)=1_G$, and let $\rho : G\times K \ra N$ be the associated Schreier cocycle $\rho (g,k)= \sigma (gk)^{-1}g\sigma (k) \in N$.

By assumption, the groupoid $N\rtimes (X,\mu )$ admits a stability sequence $(S_i)_{i\in \N}$. After replacing $(S_i)_{i\in \N}$ by a subsequence if necessary we may assume that there exists an asymptotically invariant sequence $(B_i)_{i\in \N}$ for $N\rtimes (X,\mu )$ such that $\lim _i \mu (S_i^0B_i\triangle B_i ) >0$. Fix a sequence $\mc{B}_0\subseteq \mc{B}_1\subseteq \cdots$ of finite algebras of measurable subsets of $X$ whose union generates the measure algebra of $X$. By moving to a subsequence $(S_{i_n})_{n\in \N}$ and $(B_{i_n})_{n\in \N}$, which we will call $(S_n)_{n\in \N}$ and $(B_n)_{n\in \N}$ respectively, we may ensure that
\begin{enumerate}
\item[{\bf (C1)}] $\mu (S_n^0(\sigma (k)^{-1}\cdot A)\, \triangle \, \sigma (k)^{-1}\cdot A ) < 1/n$ for all $k\in Q_n$ and $A\in \mc{B}_n$.

\item[{\bf (C2)}] $\mu (W_n)>1-2^{-n}$, where
\[
W_n = \{ x\csuchthat S_n(\rho (g,k)\cdot (\sigma (k)^{-1}\cdot x))= \rho (g,k)S_n(\sigma (k)^{-1}\cdot x)\rho (g,k)^{-1} \text{ for all }g\in F_n, \ k \in Q_n \} .
\]

\item[{\bf (C3)}] $\mu (\rho (g,k)^{-1} B_n \triangle B_n ) < 1/n$ for all $g\in F_n$, $k\in Q_n$.
\end{enumerate}
Let $X_n = \bigcap _{m\geq n}W_m$, so that $X_0\subseteq X_1\subseteq \cdots$, and $\mu (\bigcup _n X_n) = 1$. After ignoring a null set we may assume that $\bigcup _n X_n = X$. For each $n\in \N$ and $(x,z)\in X\times Z$ define
\[
T_n(x,z) = \sigma (\varphi _n (z))S_n(\sigma (\varphi _n (z))^{-1}\cdot x ) \sigma (\varphi _n (z))^{-1}  \in N .
\]
Then $T_n^0$ is an automorphism of $(X ,\mu )\otimes (Z , \eta )$, since for each fiber $(X_z,\mu _z ) := (X , \mu )\otimes ( \{ z \} , \updelta _z )$,  the restriction of $T_n^0$ to $X_z$ is an automorphism $T_{n,z}^0 : (X_z,\mu _z)\ra (X_z,\mu _z )$. We now verify that $(T_n)_{n\in \N}$ satisfies properties (i)-(iv) of Definition \ref{def:stabseq} with respect to the groupoid $\mc{G}:=G\ltimes (X,\mu )\otimes (Z,\eta )$.

(i): It suffices to prove (i) for rectangles $D=A\times C$, with $A\in \bigcup _n \mc{B}_n$. For $z\in Z\setminus C$ we have $\mu _z(T_{n,z}^0(A\times C)_z\triangle (A\times C)_z ) =0$ for all $n$. For $z\in C$, if $n$ is large enough then $A\in \mc{B}_n$ and $z\in Z_n$, so $\varphi _n (z)\in Q_n$, hence
\begin{align*}
\mu _z (T_{n,z}^0(A\times C)_z\triangle (A\times C)_z) &= \mu (\sigma (\varphi _n (z))\cdot S^0_n(\sigma (\varphi _n (z))^{-1}\cdot A )\triangle A) \\
&= \mu (S^0_n(\sigma (\varphi _n (z))^{-1}\cdot A)\triangle \sigma (\varphi _n(z))^{-1}\cdot A ) < 1/n
\end{align*}
by {\bf (C1)}. Therefore, $(\mu \otimes \eta )(T_n^0(A\times C)\triangle (A\times C)) = \int_Z\mu _z(T_{n,z}^0(A\times C)_z \triangle (A\times C)_z ) \, d\eta \ra 0$ as $n\ra \infty$.

(ii${}'$): Fix $g\in G$. For all large enough $n\in \N$ we have $g \in F_n$, so if $(x,z)\in X_n\times Z_n$ then $\varphi _n(g\cdot z) = g\cdot \varphi _n(z)$ and $\varphi _n (z)\in Q_n$, hence by {\bf (C2)}, $\rho (g,\varphi _n (z))^{-1}S_n(\rho (g,\varphi _n (z))\cdot \sigma (\varphi _n (z))^{-1}\cdot x)\rho (g,\varphi _n (z)) = S_n(\sigma (\varphi _n (z))^{-1}\cdot x)$. Therefore, for all large enough $n$, if $(x,z)\in X_n\times Z_n$ then we have
\begin{align*}
T_n (g\cdot &(x, z)) =T_n(g\cdot x, g\cdot z ) \\
&= \sigma (\varphi _n (g\cdot z))S_n(\sigma (\varphi _n (g\cdot z))^{-1}\cdot (g\cdot x) ) \sigma (\varphi _n (g\cdot z))^{-1} \\
&= \sigma (g\cdot \varphi _n (z))S_n (\sigma (g \cdot \varphi _n (z))^{-1}\cdot (g\cdot x ))\sigma (g\cdot \varphi _n (z))^{-1} \\
&= g\sigma (\varphi _n (z))\rho (g,\varphi _n (z))^{-1}S_n(\rho (g,\varphi _n (z))\sigma (\varphi _n (z))^{-1}\cdot x )\rho (g,\varphi _n (z))\sigma (\varphi _n (z))^{-1}g^{-1} \\
&= g\sigma (\varphi _n (z))S_n(\sigma (\varphi _n (z))^{-1}\cdot x )\sigma (\varphi _n (z))^{-1}g^{-1} \\
&= gT_n (x,z)g^{-1},
\end{align*}
and since $X\times Z = \bigcup _n (X_n\times Z_n )$ the proof of (ii) is complete.

For each $n\in \N$ define the set $A_n = \{ (x,z)\in X\times Z\csuchthat x \in \sigma (\varphi _n (z)) \cdot B_n \}$. We will verify (iii) and (iv) using the sequence $(A_n)_{n\in \N}$.

(iii): Fix $g\in G$. For all large enough $n\in \N$ we have $g\in F_n$, so for $(x,z)\in X\times Z_n$ we have $(x,z)\in g^{-1}\cdot A_n \IFF x\in \sigma (\varphi _n (z))\cdot \rho (g,\varphi _n (z))^{-1}\cdot B_n$. Hence, by {\bf (C3)},
\[
(\mu \otimes \eta )(g^{-1}\cdot A_n \triangle A_n ) \leq \eta (Z\setminus Z_n) + \int _Z \mu (\rho (g,\varphi _n (z))^{-1}\cdot B_n \triangle B_n ) \, d\eta < \eta (Z\setminus Z_n) + 1/n \ra 0
\]
as $n\ra \infty$. This shows that $(A_n)_{n\in \N}$ is asymptotically invariant for $G\cc (X,\mu )\otimes (Z,\eta )$.

(iv): We have $T_n^0 A_n = \{ (x,z) \in X\times Z \csuchthat x\in \sigma (\varphi _n (z))\cdot S_n^0B_n ) \}$. It follows that $(\mu\otimes \eta )(T_n^0A_n\triangle A_n ) = \mu (S_n^0B_n\triangle B_n ) \not\ra 0$. This completes the proof.
\end{proof}

\begin{remark}\label{rem:overview}
Let $1\ra N \ra G \ra K \ra 1$ be a short exact sequence in which $K$ is amenable. Let $G\cc (X,\mu )$ be a p.m.p.\ action of $G$ and let $K\cc (Z,\eta )$ be a free p.m.p.\ action of $K$. Let $G\cc (X,\mu )\otimes (Z,\eta )$ be the diagonal product action where $G$ acts on $(Z,\eta )$ via the quotient map to $K$. The above proof constructs a map which takes asymptotically central sequences in $[N\ltimes (X,\mu )]$ to asymptotically central sequences in $[G\ltimes (X,\mu )\otimes (Z,\eta )]$, and which moreover takes stability sequences for $N\ltimes (X,\mu )$ to stability sequences for $G\ltimes (X,\mu )\otimes (Z,\eta )$. In addition, it follows from the construction that if $(S_n)_{n\in \N}$ is a sequence in $[N\ltimes (X,\mu ) ]$ witnessing that the outer automorphism group of $[\mc{R}^N_X]$ is not Polish, then the image of $(S_n)_{n\in \N}$ under this map will be a sequence witnessing that the outer automorphism group of $[\mc{R}^G_{X\times Z}]$ is not Polish (see \cite[Chapter I, \S 7]{Ke10}).
\end{remark}

\begin{remark}\label{rem:finindex}
A variation of the proof of Theorem \ref{thm:extension} shows that if a group $G$ contains a finite index subgroup $H$ which is stable, then $G$ is stable. Take a p.m.p.\ action $H\cc (X_0 ,\mu _0 )$ such that $H\ltimes (X_0,\mu _0)$ admits a stability sequence $(S_n)_{n\in \N}$, and let $(B_n)_{n\in \N}$ be a sequence of asymptotically invariant sets for the action with $\lim _n \mu _0 (S^0_nB_n\triangle B_n ) \neq 0$. Let $G\cc (X,\mu )$ be the induced action, i.e., $X=X_0\times G/H$, $\mu$ is the product of $\mu _0$ with normalized counting measure, and $g\cdot (x_0 , kH ) = (\rho (g,kH)\cdot x_0 , gkH )$ for $g\in G$, $x_0 \in X_0$, $kH \in G/H$, where $\rho (g,kH)=\sigma (gkH)^{-1}g\sigma (kH)\in H$ and $\sigma :G/H \ra G$ is a section for the projection map $G\ra G/H$ with $\sigma (1H)=1$. Define $T_n \in [G\ltimes (X,\mu )]$ by $T_n (x _0 ,kH ) = \sigma (kH)S_n (x_0) \sigma (kH)^{-1} \in G$, and define $A_n = B_n\times G/H$. Then, using the sequence $(A_n)_{n\in\N}$, an argument similar to the proof of Theorem \ref{thm:extension} shows that $(T_n)_{n\in \N}$ is a stability sequence for $G\ltimes (X,\mu )$.
\end{remark}

\subsection{Proof of Theorem \ref{thm:extension2}} Throughout this subsection we work under the assumption that $1\ra N\ra G \ra K \ra 1$ is a short exact sequence in which $K$ is amenable.

\begin{lemma}\label{lem:LM}
Suppose that $N=LM$, where $L$ and $M$ are commuting normal subgroups of $N$ such that $M$ is amenable and $[N:L]=\infty$. Let $N/M \cc (X,\mu )$ and $N/L\cc (Y,\nu )$ be free p.m.p.\ actions of $N/M$ and $N/L$ respectively. Let $N$ act on $(X,\mu )$ and $(Y,\nu )$ via the quotient maps to $N/M$ and $N/L$ respectively, and let $N\cc (X,\mu )\otimes (Y,\nu )$ be the diagonal product action. Then the translation groupoid $N\ltimes (X,\mu )\otimes (Y,\nu )$ admits a stability sequence.
\end{lemma}

\begin{proof}
Let $C= L\cap M$ and let $M_0=M/C$. The action $M\cc (Y,\nu )$ descends to a free action $M_0 \cc (Y,\nu )$. Since the group $M_0$ is amenable, the equivalence relation $\mc{R}^{M_0}_Y$ is treeable. Therefore, by Theorem 1.1 of \cite{Ki14}, the p.m.p.\ groupoids $C\times M_0\ltimes (Y,\nu )$ and $M\ltimes (Y,\nu )$ are isomorphic. Here, $C\times M_0 \ltimes (Y,\nu )$ is the translation groupoid associated to the action $C\times M_0 \cc (Y,\nu )$, where $C$ acts trivially. Since $M_0$ is amenable and acts freely on $(Y,\nu )$, the groupoid $M_0\ltimes (Y,\nu )$ admits a stability sequence $(S_n)_{n\in\N}$ \cite{JS87}. For each $n\in \N$ define $S_n' \in [C\times M_0\ltimes (Y,\nu )]$ by $S_n'(y)= (1_C ,S_n(y))\in C\times M_0$. Then $(S_n')_{n\in \N}$ is a stability sequence for $C\times M_0\ltimes (Y,\nu )$. The image of this sequence under the above isomorphism is then a stability sequence $(T_n)_{n\in \N}$ for $M\ltimes (Y,\nu )$. Define $\wh{T}_n \in [M\ltimes (X,\mu )\otimes (Y,\nu )]$ by $\wh{T}_n(x,y) = T_n(y)$. Then $(\wh{T}_n)_{n\in \N}$ is a stability sequence for $N\ltimes (X,\mu )\otimes (Y,\nu )$.
\end{proof}

\begin{proof}[Proof of stability from hypothesis {\bf (H1)}]
Let $G/M \cc (X , \mu )$ and $G/L \cc (Y,\nu )$ be free p.m.p.\ actions of $G/M$ and $G/L$ respectively. Then $G$ acts on $(X,\mu )$ and $(Y,\nu )$ via the quotient maps to $G/M$ and $G/L$ respectively. Let $G\cc (X,\mu )\otimes (Y,\nu )$ be the diagonal product of these action.  By Lemma \ref{lem:LM}, the groupoid $N\ltimes (X,\mu )\otimes (Y,\nu )$ admits a stability sequence, hence $G$ is stable by Theorem \ref{thm:extension}.
\end{proof}

\begin{proof}[Proof of stability from hypothesis {\bf (H6)}]
Let $N^*= N - \{ 1 \}$. Let $G$ act on $N^*$ by conjugation and consider the corresponding generalized Bernoulli action $G\cc (X,\mu )= ([0,1]^{N^*}, \lambda ^{N^*})$ given by $(g\cdot x )(h)= x(g^{-1}hg)$. Let $(c_n)_{n\in \N}$ and $(d_n)_{n\in \N}$ be sequences witnessing that $N$ is doubly asymptotically commutative. The proof of Proposition 9.8 of \cite{Ke10} shows that the sequence $(c_n)_{n\in \N}$, viewed as a sequence in $[N\ltimes (X,\mu )]$, is a stability sequence for $N\ltimes (X,\mu )$. Theorem \ref{thm:extension} then implies that $G$ is stable.
\end{proof}

We now show that {\bf (H3)} follows from each of the hypotheses {\bf (H2)}, {\bf (H4)}, and {\bf (H5)} (in fact, it can also be shown that {\bf (H3)} follows from {\bf (H6)}, but we will not need this). This is obvious for {\bf (H2)}. Assume now that {\bf (H4)} holds, so that $A\cap C_G(g)$ has finite index in $A$ for all $g\in N$. Then we can find a decreasing sequence $A=A_0\geq A_1 \geq \cdots$ of finite index subgroups of $A$ such that $N = \bigcup _m C_N(A_m)$. Then for all $m\in \N$ the pair $(C_N(A_m), A_m )$ does not have property (T), since $A_m$ is finite index in $A$ and $(N,A)$ does not have property (T). This shows that {\bf (H3)} holds. Finally, assume that {\bf (H5)} holds. After moving to a subsequence of $(c_n)_{n\in \N}$, we may assume that each of the subgroups $A_i = \langle \{ c_n\csuchthat n\geq i \} \rangle$, $i\in \N$, is abelian, so that $N = \bigcup _iC_N(A_i)$, where the union is increasing. Then each of the groups $(C_N(A_i), A_i )$ does not have property (T), since $A_i$ is infinite and $N$ has the Haagerup property. This verifies {\bf (H3)}. It remains to deduce stability of $G$ from {\bf (H3)}.

\begin{proof}[Proof of stability from hypothesis {\bf (H3)}]
We may assume that $N$ is not doubly asymptotically commutative, since otherwise we are done by the proof of stability from {\bf (H6)}. Let $(L_m)_{m\in \N}$ and $(D_m)_{m\in \N}$ be given by {\bf (H3)}. Let $F_0\subseteq F_1\subseteq \cdots$ be an increasing sequence of finite sets which exhaust $G$.

\begin{claim}\label{claim:sequence}
There exists an increasing sequence $H_0\leq H_1\leq \cdots$, of finitely generated subgroups of $N$ with $N=\bigcup _{m\in \N}H_m$, along with sequences $C_0\geq C_1\geq \cdots$, and $(N_m)_{m \in \N}$ of subgroups of $N$ such that, for all $m\in \N$,
\begin{itemize}
\item[1.] $C_m, H_m \leq N_m$, and $C_m = C_N(H_m) = Z(N_m)$,
\item[2.] The pair $(N_m, C_m )$ does not have relative property \emph{(T)},
\item[3.] $gH_mg^{-1} \leq H_{m+1}$ for all $g\in F_{m+1}$,
\item[4.] $C_m\geq g^{-1}C_{m+1}g$ for all $g\in F_{m+1}$.
\end{itemize}
\end{claim}

\begin{proof}[Proof of Claim \ref{claim:sequence}]
Since $N$ is not doubly asymptotically commutative there exists a finitely generated subgroup $H_0\leq N$ such that $C_N(H_0)$ is abelian. After moving to a subsequence of $(L_m)_{m\in \N}$ if necessary we may assume that $H_0\leq L_0$. We may then extend $H_0$ to a sequence $H_0\leq H_1\leq H_2\leq \cdots$, of finitely generated subgroups of $N$ with $N= \bigcup _{m\in\N} H_m$, and $H_m\leq L_m$ for all $m\in \N$. After moving to a further subsequence if necessary we may assume that property 3.\ is satisfied for all $m\in \N$. Let $C_m = C_N(H_m)$, so that $C_m$ is an abelian group containing $D_m$. Then the pair $(H_mC_m , D_m )$  does not have property (T), since $(H_mD_m,D_m )$ does not have property (T) and $(H_mC_m)/(H_mD_m)$ is amenable \cite{Jo05}. It follows that $(H_mC_m, C_m )$ does not have property (T). Let $N_m = H_mC_m$. Then $C_m\leq Z(N_m)\leq C_H(N_m)\leq C_H(H_m)= C_m$, so that both 1.\ and 2.\ are satisfied, and 4.\ follows from 1.\ and 3. \qedhere[Claim \ref{claim:sequence}]
\end{proof}

Fix an increasing sequence $S_0\subseteq S_1\subseteq \cdots$ of finite sets such that $S_m$ generates $H_m$ for all $m\geq 0$. Note that if $\pi _m$ is an irreducible unitary representation of $N_m$, then Schur's Lemma implies that $\pi _m (c)$ is a scalar multiple of the identity for all $c\in C_m$, since $C_m=Z(N_m)$. Property 4.\ of the claim then shows that $\pi _m (g^{-1}cg)$ is also a scalar for all $c\in C_{n}$, $g\in F_{n}$, $n>m$. The following is based on Lemma 4.1 of \cite{Ki13a}.

\begin{lemma}\label{lem:avg}
There exist sequences $(\pi _m )_{m\in \N}$, $(\xi _m )_{m\in \N}$, and $(c_m)_{m\in \N}$, where for each $m\in \N$, $\pi _m$ is an irreducible unitary representation of $N_{m}$, $\xi _m \in \Es{H}_{\pi _m}$ is a unit vector, and $c_m$ is an element of $C_{m}$, such that
\begin{enumerate}
\item[(i)] $\| \pi _m (s)\xi _m - \xi _m \| < 2^{-m}$ for all $s\in S_m$,
\item[(ii)] $| \pi _m (c_m) - 1 | > 1$,
\item[(iii)] $| \pi _j (g^{-1}c_mg) - 1 | < 2^{-m}$ for all $g\in F_m$ and $j<m$.
\end{enumerate}
\end{lemma}

\begin{proof}
Let $m\geq 0$ and assume inductively that we have already found $(\pi _j )_{j<m}$, $(\xi _j )_{j<m}$, and $(c_j)_{j<m}$. By compactness there exists a nonempty finite set $P\subseteq C_m$ such that for each $c\in C_m$ there exists some $d\in P$ with $\sup _{j<m}\sup _{g\in F_m} |\pi _j (g^{-1}cg)-\pi _j (g^{-1}dg) | < 2^{-m}$. Since $(N_m,C_m )$ does not have property (T), we may find an irreducible unitary representation $\pi _m$ of $N_m$ which has no nonzero $C_m$-invariant vector, satisfying $\sup _{d\in P} |\pi _m (d)-1| < 1/3$, along with a unit vector $\xi _m \in \Es{H}_{\pi _m}$, such that $\| \pi _m (s)\xi _m - \xi _m \| < 2^{-m}$ for all $s\in S_m$. Since $\pi _m (C_m)$ is nontrivial, there exists some $d_m \in C_m$ satisfying $|\pi _m (d_m) - 1 | > 4/3$. By our choice of $P$ there exists $d\in P$ such that $\sup _{j<m}\sup _{g\in F_m} | \pi _j (g^{-1}d_mg)-\pi _j(g^{-1}dg) | < 2^{-m}$. Let $c_m=d^{-1}d_m$. Then $\sup _{j<m}\sup _{g\in F_m}|\pi _j (g^{-1}c_mg) - 1 |< 2^{-m}$, and
\begin{align*}
|\pi _m (c_m) - 1 | &= |\pi _m (d_m) - \pi _m (d) | \geq |\pi _m (d_m) -1 | - | \pi _m (d) - 1| > 4/3 - 1/3 = 1 .\qedhere
\end{align*}
\end{proof}

For each $m\in \N$ let $\sigma _m: G/N_m \ra G$ be a section for the projection map $G\ra G/N_m$ with $\sigma _m (1N_m)=1$, and let $\rho _m : G\times G/N_m\ra N_m$ be the corresponding Schreier cocycle, $\rho _m (g,hN_m) = \sigma _m(ghN_m)^{-1}g\sigma _m(hN_m)$. Let $\wt{\pi}_m = \mathrm{Ind}_{N_m}^G (\pi _m )$ be the induced representation, i.e., $\wt{\pi}_m$ is the representation of $G$ on $\Es{H}_m = \Es{H}_{\pi _m}\otimes \ell ^2(G/N_m)$ given by $\wt{\pi}_m (g)(\xi \otimes \delta _{hN_m} ) = \pi _m(\rho _m(g,k))(\xi )\otimes \delta _{ghN_m}$. Let $\wt{\xi} _m= \xi _m \otimes \delta _{1N_m}$, so that $\wt{\xi} _m\in \Es{H}_m$ is a unit vector satisfying $\| \wt{\pi} _m (s)\wt{\xi} _m - \wt{\xi} _m \| = \| \pi _m (s)\xi _m - \xi _m \| < 2^{-m}$ for all $s\in S_m$ by property (i). By property (ii), for each $m\in \N$ we have $\| \wt{\pi}_m(c_m)\wt{\xi}_m - \wt{\xi}_m \|  = \| (\pi _m (c_m) -1 )\xi _m \|  =  |\pi _m (c_m)-1 | >1$. Moreover, it follows from property (iii) that $\lim _{m\ra\infty} \| \wt{\pi }_j (c_m)\eta - \eta \| = 0$ for all $j\in \N$ and $\eta \in \Es{H}_j$. Let $G\cc (\Omega , \nu )$ denote the Gaussian action associated to the representation $\bigoplus _m \wt{\pi}_m$. As in the proof of Theorem 1.1.(i) of \cite{Ki13a}, we conclude that the sequence $(c_m)_{m\in \N}$, viewed as a sequence in the full group $[N\ltimes (\Omega , \nu )]$, is a stability sequence for $N\ltimes (\Omega ,\nu )$. We can now apply Theorem \ref{thm:extension} to conclude that $G$ is stable.
\end{proof}

The following proposition shows that, aside from the relative property (T) condition, the hypothesis {\bf (H4)} has a natural expression in terms of a conjugation invariant mean on $G$ which concentrates on $A$, as long as we assume that $A$ is finitely generated.

\begin{proposition}
Let $G$ be a countable group. Then the following are equivalent:
\begin{enumerate}
\item There exists an atomless conjugation invariant mean $\bm{m}$ on $G$ and an infinite finitely generated abelian group $A\leq G$ with $\bm{m}(A)>0$.
\item There exists an infinite finitely generated abelian group $A\leq G$ with $[G: \mathrm{comm}_G(A)]<\infty$ such that $\mathrm{comm}_G(A)/N$ is amenable, where $N$ is the kernel of the modular homomorphism from $\mathrm{comm}_G(A)$ into the abstract commensurator of $A$.
\end{enumerate}
Furthermore, the statement $(1')$, obtained from $(1)$ by replacing \emph{"$\bm{m}(A)>0$"} with \emph{"$\bm{m}(A)=1$"}, is equivalent to the statement $(2')$, obtained from $(2)$ by replacing \emph{"$[G:\mathrm{comm}_G(A)]<\infty$"} with \emph{"$G=\mathrm{comm}_G(A)$"}.
\end{proposition}

\begin{proof}
$(2)\Ra (1)$: Let $(A_n)_{n\in \N}$ enumerate the finite index subgroups of $A$, and for each $n\in \N$ choose a nonidentity element $a_n \in \bigcap _{i\leq n}A_i$. Let $\bm{m}_0$ be an accumulation point of $(\delta _{a_n})_{n\in \N}$ in the space of means on $G$. Then $\bm{m}_0$ concentrates on $A$, and is invariant under conjugation by $N$ since each $g\in N$ commutes with $a_n$ for cofinitely many $n\in \N$. Let $G_0 = \mathrm{comm}_G(A)$ and let $\bm{m}_{G_0/N}$ be a translation-invariant mean on $G_0/N$. Then the mean $\bm{m}_1 = \int _{gN\in G_0/N}g\bm{m}_0g^{-1}\, d\bm{m}_{G_0/N}$ is invariant under conjugation by $G_0$, and $\bm{m}_1$ concentrates on $A$ since each of the means $g\bm{m}_0g^{-1}$, $g\in G_0$, concentrates on $A$. Finally, the mean $\bm{m} = \frac{1}{[G:G_0]}\sum _{gG_0\in G/G_0} g\bm{m}_1g^{-1}$ is invariant under conjugation by $G$ and satisfies $\bm{m}(A)\geq \frac{1}{[G:G_0]}>0$. This also shows the implication $(2')\Ra (1')$.

$(1)\Ra (2)$: We may assume that $A$ has minimal rank among all $\bm{m}$-non-null finitely generated abelian subgroups of $G$. If $g\in G$ is such that $\bm{m}(gAg^{-1} \cap A ) >0$, then $gAg^{-1}\cap A$ has the same rank as $A$, so $gAg^{-1}\cap A$ has finite index in $A$. Let $G_0=\mathrm{comm}_G(A)$ and suppose toward a contradiction that $[G:G_0]=\infty$. Let $(g_n)_{n\in \N}$ be a sequence with $g_iG_0\neq g_jG_0$ for all $i\neq j$. Then, for all $i\neq j$, the group $g_j^{-1}g_iAg_i^{-1}g_j \cap A$ does not have finite index in $A$, hence $0 = \bm{m}(g_j^{-1}g_iAg_i^{-1}g_j \cap A )=\bm{m} (g_iAg_i^{-1}\cap g_jAg_j^{-1} )$. Therefore for all $n >0$ we have $1\geq \bm{m}(\bigcup _{i< n} g_i A g_i^{-1} )  =\sum _{i<n}\bm{m} (g_iAg_i^{-1})=n\cdot \bm{m}(A)$, a contradiction, since $\bm{m}(A)>0$. This shows that $[G:G_0]<\infty$.

Let $N = \{ g\in G_0 \csuchthat [A:A\cap C_G(g)]<\infty \}$ be the kernel of the modular homomorphism $G_0\ra \mathrm{comm}(A)$, and let $\varphi : G_0 \ra G_0/N$ be the projection to the group $G_0/N$. Suppose toward a contradiction that $\varphi (G_0)$ is nonamenable. For a subgroup $B\leq A$ let $N(B) = \{ g \in G_0\csuchthat [B:B\cap C_G(g)] <\infty \}$. Then $N(B)$ is a subgroup of $G_0$ with $N\leq N(B)$ for all $B\leq A$. Since $A$ is a finitely generated abelian group it satisfies the maximal condition on subgroups. Let $B_0\leq A$ be a maximal subgroup from the collection $\{ B\leq A\csuchthat \varphi (N(B)) \text{ is nonamenable}\}$; this collection is nonempty since it contains the trivial group by hypothesis. Since $N(A) = N$, the group $B_0$ has infinite index in $A$, and hence $\bm{m}(B_0)=0$. Observe that for any $a\in A - B_0$, if we let $B_1 = \langle a , B_0 \rangle$, then the group $\varphi (N(B_1))$ is amenable by maximality of $B_0$, and we have $C_{N(B_0)}(a) \leq N(B_1)$, so the group $\varphi (C_{N(B_0)}(a))$ is amenable. Let $Y$ be the saturation of $A - B_0$ under conjugation by $N(B_0)$. Then $\varphi (C_{N(B_0)}(y))$ is amenable for all $y\in Y$. It follows that for all $y\in Y$, the translation action $C_{N(B_0)}(y)\cc G_0/N$ is amenable. In addition, the conjugation action $N(B_0)\cc Y$ is amenable since $\bm{m}(Y)>0$. It follows (using Theorem \ref{thm:isoperim} for example) that the translation action $N(B_0)\cc G_0/N$ is amenable, and hence the group $\varphi (N(B_0))$ is amenable, a contradiction.

For $(1')\Ra (2')$, we take $A$ to have minimal rank among all $\bm{m}$-conull finitely generated abelian subgroups of $G$. Then for all $g\in G$ the group $gAg^{-1}\cap A$ is $\bm{m}$-conull, so $gAg^{-1}\cap A$ has finite index in $A$, and therefore $G = \mathrm{comm}_G(A)$. The rest proceeds as above.
\end{proof}

\subsection{Proof of Theorem \ref{thm:linearstable}}
\begin{proof}[Proof of Theorem \ref{thm:linearstable}]
(1)$\Ra$(2): Suppose first that $G$ is stable as witnessed by the free ergodic action $G\cc (X, \mu )$. Then $G\ltimes (X,\mu )$ admits a stability sequence $(T_n)_{n\in \N}$ \cite{JS87}. Let $\mc{U}$ be a non-principal ultrafilter on $\N$ and for $D\subseteq G$ let $\bm{m} (D) = \lim _{n\ra\mc{U} } \mu ( \{ x\in X\csuchthat T_n(x)\in D \} )$. Then $\bm{m}$ is a conjugation invariant mean on $G$, so $\bm{m}(\mathscr{I}(G))=1$ by Theorem \ref{thm:I(G)}.xiii. We may therefore assume without loss of generality that $(T_n)_{n\in \N}$ is contained in $[\mathscr{I}(G)\ltimes (X,\mu )]$. It follows that every subgroup of $G$ containing $\mathscr{I}(G)$ is stable. Let $N=C_G(\mathscr{I}(G))\mathscr{I}(G)$ and note that $\mathscr{I}(N) = \mathscr{I}(G)$ since $G/N$ is amenable. If $[N: C_G(\mathscr{I}(G))] = \infty$ then the image of $\mathscr{I}(G)$ in the amenable group $G/C_G(\mathscr{I}(G))$ is infinite, so the pair $(G/C_G(\mathscr{I}(G)) , N/C_G(\mathscr{I}(G)))$ does not have property (T), and hence $(G,\mathscr{I}(G))$ does not have property (T). We may therefore assume that $[N: C_G(\mathscr{I}(G))]<\infty$. This implies that the group $Z(N) = C_G(\mathscr{I}(G))\cap \mathscr{I}(G)$ has finite index in $\mathscr{I}(G)= \mathscr{I}(N)$. Suppose toward a contradiction that $(G,\mathscr{I}(G))$ has property (T). Since the group $G/N$ is amenable, the pair $(N, \mathscr{I}(G))$ has property (T), hence $(N,Z(N))$ has property (T) \cite{Jo05}. The group $N$ is stable since $\mathscr{I}(G)\leq N$, so Theorem 1.1.(2) of \cite{Ki13a} implies that $N/Z(N)$ is stable, and in particular $N/Z(N)$ is inner amenable \cite{JS87}. This is a contradiction since, by Theorem \ref{thm:I(G)}, every conjugation invariant mean on $N/Z(N)$ concentrates on the group $\mathscr{I}(N/Z(N)) = \mathscr{I}(N)/Z(N)$, which is finite. We conclude that the pair $(G , \mathscr{I}(G) )$ does not have property (T).

(2)$\Ra$(1): Assume that $(G,\mathscr{I}(G))$ does not have property (T). It follows that $\mathscr{I}(G)$ is infinite. Let $L= C_G(\mathscr{I}(G))$, let $M=\mathscr{I}(G)$, and let $N=LM$ so that $Z(N) = L\cap M$. The group $G/N$ is amenable by Theorem \ref{thm:I(G)}. If $[N:L]=\infty$ then hypothesis {\bf (H1)} holds, so $G$ is stable by Theorem \ref{thm:extension2}. If $[N:L]<\infty$ then $(N,Z(N))$ does not have property (T), since $(N,\mathscr{I}(G))$ does not have property (T) and $[\mathscr{I}(G):Z(N) ] =[N:L]$. This shows that hypothesis {\bf (H2)} holds, so $G$ is stable by Theorem \ref{thm:extension2}.
\end{proof}

\begin{remark}\label{rem:final}
It follows from Theorems \ref{thm:I(G)}, \ref{thm:linear}, and \ref{thm:linearstable}, that a linear group $G$ is inner amenable but not stable if and only if the group $\mathscr{I}(G)$ is infinite and is finite index over its center $C = C_G(\mathscr{I}(G))\cap \mathscr{I}(G)$, and the the pair $( C_G(\mathscr{I}(G)), C)$ has property (T).
\end{remark}

\subsection{Groups of piecewise projective homeomorphisms} Justin Moore has observed that an adaptation of the arguments of Brin-Squier \cite{BS85} and Monod \cite{Mo13} shows the following

\begin{lemma}\label{lem:Justin}
Let $G$ be a countable subgroup of $H(\R )$. Then the second derived subgroup $G''$ is either abelian or doubly asymptotically commutative.
\end{lemma}

\begin{proof} Assume first that $G$ is finitely generated. Then the set $U = \bm{\mathrm{P}}^1\setminus \mathrm{fix}(G)$ has finitely many connected components, each of which is an open interval. If $V\subseteq U$ is a union of a subset of these connected components then let $\varphi _{V}:G \ra H(\R )$ denote the homomorphism which sends $g\in G$ to the map $\varphi _V(g)$ which coincides with $g$ on $V$ and which is the identity elsewhere. In what follows, we fix an orientation of $\bm{\mathrm{P}}^1\setminus \{ \infty \}$.

\begin{claim}\label{claim:push}
For any compact subset $K\subseteq U$ there exists an element $g\in G$ such that $g(K)\cap K =\emptyset$.
\end{claim}

\begin{proof}[Proof of Claim \ref{claim:push}]
By induction on the number $n$ of connected components of $U$. If $n=1$ then it suffices to show that for any $p\in U$ we have $\sup _{g\in G}g(p) = \sup U$. Suppose otherwise and let $q= \sup _{g\in G}g(p) < \sup U$. Then $q\in U$, so we may find some $g\in G$ with $g(q)\neq q$, and after replacing $g$ by $g^{-1}$ if necessary we may assume that $g(q)>q$. If $(g_n)_{n\in \N}$ is any sequence in $G$ with $g_n(p)\ra q$ then $q\geq g(g_n(p))\ra g(q)>q$, a contradiction. Assume now that $U$ has $n+1$ connected components and fix one such component $V$. After making $K$ larger if necessary we may assume that $K\cap V$ is a closed interval. Apply the base of the induction to the group $\varphi _{V} (G)$ to obtain a group element $h\in G$ with $h(K\cap V)$ disjoint from $K\cap V$. Since $K\cap V$ is an interval, after replacing $h$ by $h^{-1}$ if necessary, this means that $\inf h(K\cap V) > \sup (K\cap V)$. Let $L = (K\setminus V) \cup h(K\setminus V)$. Then $L$ is a compact subset of $U\setminus V$, so we may apply the induction hypothesis to the group $\varphi _{U\setminus V}(G)$ to obtain a group element $f\in G$ satisfying $f(L)\cap L =\emptyset$. After replacing $f$ by $f^{-1}$ if necessary we may assume that $f(p)\geq p$, where $p= \inf h(K\cap V)$. Take $g=fh$. Then $\inf g(K\cap V) = f(p)\geq p >\sup (K\cap V)$, so $g(K\cap V)$ is disjoint from $K\cap V$. In addition, $g(K\setminus V)\cap (K\setminus V) =f(h(K\setminus V))\cap (K\setminus V) \subseteq f(L)\cap L = \emptyset$. Since $V$ is $G$-invariant, this shows that $g(K)\cap K=\emptyset$.
\qedhere[Claim \ref{claim:push}]
\end{proof}

Assume that $G''$ is nonabelian and fix two non-commuting elements $c_0,d_0\in G''$. As shown in Lemma 14 of \cite{Mo13}, the closure of the support of any element of $G''$ is a compact subset of $U$. Fix a finite subset $Q\subseteq G''$ and let $K$ be the union of the closures of the supports of all elements of $Q\cup \{ c_0,d_0 \}$. Apply the claim to find an element $g\in G$ with $g(K)\cap K=\emptyset$. Let $c=gc_0g^{-1}$ and let $d=gd_0g^{-1}$ so that $c,d\in G''$ and $cd\neq dc$. Then $c$ and $d$ both commute with each element of $Q$ since the support of $c$ and of $d$ are disjoint from the support of each element of $Q$. This shows that $G''$ is doubly asymptotically commutative.

When $G$ is not finitely generated we may write $G$ as an increasing union $G=\bigcup _n G_n$ with each $G_n$ finitely generated. Then $G''=\bigcup _n G_n''$, so if $G''$ is nonabelian then $G_n''$ is nonabelian for all large enough $n$. Now note that double asymptotic commutativity is preserved by directed unions.
\qedhere
\end{proof}

\begin{proof}[Proof of Theorem \ref{thm:HR}]
By \cite{Mo13}, $H(\R )$ is torsionfree, so any nontrivial amenable subgroup of $H(\R )$ is infinite, hence stable. If $G$ is a nonamenable subgroup of $H(\R )$ then $G''$ is nonabelian, hence doubly asymptotically commutative by Lemma \ref{lem:Justin}. Hypothesis {\bf (H6)} holds, using the short exact sequence $1\ra G''\ra G\ra G/G'' \ra 1$, hence Theorem \ref{thm:extension2} shows that $G$ is stable.
\end{proof}

\noindent{\bf Acknowledgements:} The author warmly thanks Justin Moore for explaining Lemma \ref{lem:Justin}, Adrian Ioana for helpful conversations and for allowing the inclusion of the joint Theorem \ref{thm:superrigid1}, and Yoshikata Kida for pointing out {\bf (H4)} from Theorem \ref{thm:extension2} and Corollary \ref{cor:gBS} and for a number of insightful comments on an earlier draft of this paper which helped improve the statement of Theorem \ref{thm:extension2}. The author was supported by NSF grant DMS 1303921.

\appendix
\appendixpage

\section{Pseudocost} \label{sec:pseudocost}

Pseudocost is a modification of cost which was defined and studied by the author in \cite{T-D12c} in order to extend several properties of cost for finitely generated groups to the non-finitely generated setting. The pseudocost of a p.m.p.\ equivalence relation $\mc{R}$ on $(X,\mu )$ is defined as $\mathscr{PC}(\mc{R})=\inf _{(\mc{R}_n)}\liminf _n \mathscr{C}(\mc{R}_n)$, where the infimum is taken over all increasing exhaustive sequences $(\mc{R}_n)_{n\in \N}$ of subequivalence relations of $\mc{R}$. For a group $G$, the values $\mathscr{PC}(G)$ and $\mathscr{PC}^*(G)$ are defined to be the infimum and supremum respectively, of the pseudocosts of orbit equivalence relations generated by free p.m.p.\ actions of $G$. The inequality $\mathscr{PC}(\mc{R})\leq \mathscr{C}(\mc{R})$ always holds, and it is shown in \cite{T-D12c} that equality holds whenever $\mathscr{C}(\mc{R})$ is finite (hence $\mathscr{PC}(G)=\mathscr{C}(G)$ and $\mathscr{PC}^*(G)=\mathscr{C}^*(G)$ whenever $G$ is finitely generated), and that $\mathscr{PC}(\mc{R})=1$ if and only if $\mathscr{C}(\mc{R})=1$. Using Gaboriau's inequality $\beta ^{(2)}_1(\mc{R}) \leq \mathscr{C}(\mc{R}) -1$ from \cite{Ga02}, it is easy to see that in fact $\beta ^{(2)}_1(\mc{R})\leq \mathscr{PC}(\mc{R}) -1$.

\begin{proposition}\label{prop:PC}
Let $G$ be a countable group and suppose that $\mathscr{PC}(G)>r$. Then there exists a finitely generated subgroup $G_0\leq G$ such that $\mathscr{PC}(H)>r$ for all intermediate subgroups $G_0\leq H \leq G$. In particular, if $(H_n)_{n\in \N}$ is a sequence of subgroups of $G$ with $G=\liminf _n H_n$ then $\mathscr{PC}(G)\leq \liminf _n \mathscr{PC}(H_n)$.
\end{proposition}

\begin{proof}
Suppose there is no such subgroup $G_0$. Then we may find an increasing exhaustive sequence $(G_n)_{n\geq 1}$ of finitely generated subgroups of $G$ and for each $n\geq 1$ an intermediate subgroup $G_n\leq G_n'\leq G$ with $\mathscr{PC}(G_n')\leq r$. For each $n\geq 1$ let $G_n' \cc (X_n, \mu _n )$ be a free p.m.p.\ action with $\mathscr{PC}(\mc{R}^{G_n'}_{X_n} )\leq r$. Then, setting $H_{0}=1$, by Lemma 6.14.(4) of \cite{T-D12c}, for each $n\geq 1$ we may find a finitely generated group $H_n$ with $\langle G_n, H_{n-1}\rangle \leq H_n\leq G_n'$, along with an equivalence relation $\mc{R}_n$ with $\mc{R}_{X_n}^{\langle G_n,H_{n-1}\rangle}\subseteq \mc{R}_n\subseteq \mc{R}_{X_n}^{H_n}$ and $\mathscr{C}(\mc{R}_n)<r+ 1/n$. Let $G\cc (X,\mu ) = \prod _n (X_n,\mu _n )^{G/G_n'}$ be the diagonal product of the coinduced actions. For each $n\geq 1$ the action $G_n' \cc (X,\mu )$ factors onto $G_n'\cc (X_n, \mu _n )$ via the projection $p_n : (X,\mu ) \ra (X_n , \mu _n )$, so the equivalence relation $\wh{\mc{R}}_n = \mc{R}_X^{G_n'}\cap (p_n\times p_n )^{-1}(\mc{R}_n)$ satisfies $\mc{R}_X^{\langle G_n,H_{n-1} \rangle }\subseteq \wh{\mc{R}}_n \subseteq \mc{R}_X^{H_n}$ and $\mathscr{C}(\wh{\mc{R}}_n)\leq \mathscr{C}(\mc{R}_n)<r+1/n$. It follows that $\mathscr{PC}(G)\leq \mathscr{PC}(\mc{R}^G_X)\leq \liminf _n \mathscr{C}(\wh{\mc{R}}_n) \leq r$, a contradiction.
\end{proof}

The next proposition is a minor modification of arguments due to Furman, Gaboriau, and Kechris (see \cite[VI.24.(3)]{Ga00} and \cite[Lemma 24.7, Proposition 35.4]{KM04}).

\begin{proposition}\label{prop:Furman1} Let $\mc{S}$ be a $wq$-normal (see \S\ref{sec:groupoid}) subequivalence relation of the p.m.p.\ equivalence relation $\mc{R}$ on $(X,\mu )$. Then $\mathscr{PC}(\mc{R})\leq \mathscr{PC}(\mc{S})$.
\end{proposition}

\begin{proof}
It suffices to deal with the case where $\mc{S}$ is $q$-normal in $\mc{R}$. Then we may find a countable subset $(\phi _n)_{n\geq 1}$ of $Q_{\mc{R}}(\mc{S})$ which generates $\mc{R}$. Let $A_n$ and $B_n$ denote the domain and range respectively of $\phi _n$. Given $\epsilon >0$, for each $n\geq 1$ the equivalence relation $\mc{S}_{B_n}^{\phi}\cap \mc{S}_{A_n}$ on $A_n$ is aperiodic, so it has a measurable complete section $C_n\subseteq A_n$ with $\mu (C_n)<\epsilon /2^n$. Then for each $n\geq 1$ and $x\in A_n$ there exists a path from $x$ to $\phi _n(x)$ in $\mc{S} \cup \{ (y, \phi _n (y) ) \csuchthat y\in C_n \}$, namely there is some $y\in C_n$ with $(x,y)\in \mc{S}_{B_n}^{\phi}\cap \mc{S}_{A_n}$, so that $(x,y), (\phi _n (x),\phi _n (y))\in \mc{S}$, and hence the path from $x$ to $y$ to $\phi _n (y)$ to $\phi _n (x)$ works. It follows that $\mc{R}$ is generated by $\mc{S}$ along with $( \phi _n |C_n )_{n\in \N}$. Therefore, any increasing exhaustion $(\mc{S}_i)_{i\in \N}$ of $\mc{S}$ gives rise to a corresponding exhaustion $(\mc{R}_i)_{i\in \N}$ of $\mc{R}$ with $\mathscr{C}(\mc{R}_i) \leq \mathscr{C}(\mc{S}_i)+ \sum _{n\in \N} \mu (C_n) < \mathscr{C}(\mc{S}_i)+\epsilon$, and hence $\mathscr{PC}(\mc{R}) \leq \mathscr{PC}(\mc{S}) + \epsilon$. Letting $\epsilon \ra 0$ completes the proof.
\end{proof}

\begin{proposition}\label{prop:Furman2} Let $H$ be a $wq$-normal subgroup of $G$. Then for any free p.m.p.\ action $G\cc (X,\mu )$, we have $\mathscr{PC}(\mc{R}^G_X)\leq \mathscr{PC}(\mc{R}^H_X)$. In addition, we have $\mathscr{PC}^*(G)\leq \mathscr{PC}^*(H)$ and $\mathscr{PC}(G)\leq \mathscr{PC}(H)$.
\end{proposition}

\begin{proof}
The first two inequalities follow from Proposition \ref{prop:Furman1}, and the inequality $\mathscr{PC}(G)\leq \mathscr{PC}(H)$ is then obtained by coinduction.
\end{proof}

\section{A strongly almost free amenable action of the free group}

In \cite{vD90}, van Douwen constructs an amenable faithful action of a free group which is {\bf almost free}, i.e., every nonidentity element fixes at most finitely many points. The property that we need for Example \ref{ex:faction} is somewhat stronger however; let us say that an action of a group $G$ on a set $X$ is {\bf strongly almost free} if the associated action of $G$ on the collection $\mathscr{P}_{\mathrm{f}}(X)$, of all finite subsets of $X$, is almost free. Equivalently, this means that each nonidentity element of $G$, when viewed as a permutation on $X$, contains only finitely many finite cycles in its cycle decomposition.

\begin{theorem}\label{thm:safree}
Let $G$ be a free group of rank $r\in \{ 2,3,\dots , \infty \}$. There exists a transitive amenable action of $G$ on a countable set $X$ which is strongly almost free.
\end{theorem}

For the proof, we will assume that $G$ is finitely generated; the proof of the infinitely generated case is similar. Before presenting the proof we establish some notation. Fix a free generating set $S$ for $G$. The construction of the action will use the formalism of {\bf $S$-digraphs}; the reader is referred to \cite{KM02} for background.

Given a connected folded $S$-digraph $\Gamma$, let $\Gamma ^*$ denote the unique connected folded $S$-digraph extending $\Gamma$ which is $2|S|$-regular and satisfies $\mathrm{\emph{Core}}(\Gamma ^*,v ) = \mathrm{\emph{Core}}(\Gamma , v )$ for all $v\in V(\Gamma )$. We let $G$ act on $V(\Gamma ^*)$ in the natural way, i.e., $g\cdot v$ is the terminus of the unique path in $\Gamma ^*$ with origin $v$ and label $g$. If we fix a vertex $v_0\in V(\Gamma )$, then $\Gamma ^*$ is isomorphic to the Schreier graph of the action $G \cc G /G_{v_0}$ with respect to the generating set $S$. By a {\bf cycle} in $\Gamma$ we mean a path in $\Gamma$ whose origin and terminus coincide. A cycle $c$ in $\Gamma$ is {\bf cyclically reduced} if its label is a cyclically reduced word in $S\cup S^{-1}$. If a cycle $c$ in $\Gamma ^*$ is cyclically reduced then it is contained in $\mathrm{\emph{Core}}(\Gamma ,v)$ for any $v\in V(\Gamma ^*)$, so taking $v\in V(\Gamma )$ shows that $c$ is contained in $\Gamma$. We make two observations:
\begin{enumerate}
\item[{\bf (a)}] If $p$ is a reduced path in $\Gamma ^*$ with origin and terminus in $\Gamma$, then $p$ is contained in $\Gamma$.

\begin{proof} Otherwise, let $p_0$ be a reduced path of minimal length starting at $u\in V(\Gamma )$ and ending at $v\in V(\Gamma )$, and which is not contained in $\Gamma$. Since $\Gamma$ is connected there is a reduced path $p_1$ from $v$ to $u$ in $\Gamma$. Then the concatenation of $p_0$ followed by $p_1$ is a cycle at $u\in V(\Gamma )$ which is cyclically reduced by minimality of $p_0$, and hence is contained in $\Gamma$, a contradiction.\end{proof}

\item[{\bf (b)}] If $w$ is a cyclically reduced word in $S\cup S^{-1}$, and if the orbit of $x\in V(\Gamma ^*)$ under $w$ is finite, then $x\in V(\Gamma )$.

\begin{proof} If $w^k\cdot x = x$ for some $x\in V(\Gamma ^*)$, $k\geq 1$, then the cycle rooted at $x$ with label $w^k$ is cyclically reduced, hence is contained in $\Gamma$.\end{proof}
\end{enumerate}

\begin{proof}[Proof of Theorem \ref{thm:safree}]
For $n\in \N$, let $C_n$ denote the collection of all nonempty cyclically reduced words on $S\cup S^{-1}$ of length at most $n$. We will construct a sequence $\Gamma _n$, $n\in \N$, of finite, folded $S$-digraphs such that for all $n$:
\begin{itemize}
\item[(1)] $\Gamma _n$ contains a vertex of degree strictly less than $2|S|$;
\item[(2)] $V(\Gamma _n)$ contains a set which is $(S,1/n )$-invariant for the action $G\cc V(\Gamma _n ^*)$;
\item[(3)] If $n\geq 1$ then $\Gamma _{n-1}\subseteq \Gamma _n$, and for each $w\in C_{n-1}$, every finite orbit of $w$ in $V(\Gamma _n ^*)$ is contained in $V(\Gamma _{n-1})$.
\end{itemize}
Assume first that such a sequence has been constructed and we will complete the proof. Take $\Gamma _{\infty} = \bigcup _n \Gamma _n$. Then $V(\Gamma _{\infty}^*)$ is infinite by (1), and the action $G \cc V(\Gamma _{\infty}^*)$ is amenable by property (2). Property (3) ensures that if $w\in C_m$ for some $m\in \N$, then every finite cycle in the cycle decomposition of $w$ in $V(\Gamma _{\infty}^*)$ is contained in the finite set $V(\Gamma _m)$, so $w$ fixes only finitely many finite subsets of $V(\Gamma _{\infty}^*)$. The theorem then follows, since every $g\in G$ is conjugate to a cyclically reduced word.

To define the sequence $(\Gamma _n)_{n\in \N}$ we start by taking $\Gamma _0$ to consist of a single vertex with no edges. Assume inductively that $\Gamma _{n-1}$ has been defined satisfying (1), (2), and (3). Let $N$ be a finite index normal subgroup of $G$ with $C_n \cap N = \emptyset$. Then the group $G/N'$ is abelian-by-finite and it is torsionfree by Theorem 2 of \cite{Hi55}. Let $\Delta _n$ be the Schreier graph for the action $G\cc G/N'$ with respect to the generating set $S$, with root vertex $1N'\in G/N'$. Since $G/N'$ is amenable, there exists a natural number $k>0$ such that the $(k-1)$-ball in $\Delta _n$ contains a set which is $(S, 1/n )$-invariant. Let $B_k$ denote the induced subgraph on the $k$-ball in $\Delta _n$. Fix $u\in V(\Gamma _{n-1})$ having degree strictly less than $2|S|$; by symmetry we may assume that there there exists $s\in S$ such that $u$ has no outgoing edge in $\Gamma _{n-1}$ with label $s$. Since $G/N'$ is torsionfree and $s\not\in N'$, there exists a vertex $v\in V(B_k)$ which has no incoming edge in $B_{k}$ with label $s$. Let $\Gamma _n$ be the graph obtained from the disjoint union of $\Gamma _{n-1}$ and $B_k$ by attaching a directed path $p$ from $u$ to $v$ of length $2n$, with each edge in $p$ having label $s$ (so $|V(\Gamma _n )| = |V(\Gamma _{n-1})| + |V(B_k)| + 2n-1$). Properties (1) and (2) are immediate. Fix now $w\in C_{n-1}$ and we will verify (3). Each orbit of $w$ in $V(\Gamma _n ^*)$ which meets $V(\Gamma _n ^*)\setminus V(\Gamma _n )$ is infinite by {\bf (b)}. Let $O$ be an orbit of $w$ which is contained in $V(\Gamma _n)$. By {\bf (a)}, for each $x\in O$, the path $p_x$ in $\Gamma _n ^*$ having origin $x$ and label $w$ is contained in $\Gamma _n$. Therefore, since $w$ is cyclically reduced and has length less than $n$, if $O$ contains a vertex $x\in V(p)\setminus \{ u, v \}$ then either $p_x$ or $p_{w^{-1}\cdot x}$ is contained in $p$, and hence $w=s^i$ for some $1\leq |i|<n$. But then $w^k\cdot x = s^{ik}\cdot x\neq x$ for all $k\geq 1$, contradicting that $O$ is finite. It follows that $O$ is contained either in $V(\Gamma _{n-1})$ or in $V(B_k)$. But $O$ cannot be contained in $V(B_k)$ since $w\not\in N'$ and $G/N'$ is torsionfree. So $O$ is contained in $V(\Gamma _{n-1})$. This completes the proof of (3).
\end{proof}

\bibliographystyle{plain}
\bibliography{biblio}
\end{document}